\documentclass[twoside,a4paper]{article}
\usepackage[T1]{fontenc}
\usepackage{amssymb}
\usepackage{graphicx}
\usepackage[cmtip,all]{xy}
\usepackage[dvipsnames]{xcolor}
\usepackage{amsmath,bm}
\usepackage{amsmath,amssymb}
\usepackage[makeroom]{cancel}
\usepackage{mathtools}
\usepackage[colorlinks,allcolors=purple]{hyperref}
\usepackage[noabbrev,capitalize,nameinlink]{cleveref}
\usepackage[frak=esstix]{mathalpha}
\usepackage[margin=20mm]{geometry}
\usepackage{amsmath}
\usepackage{mathrsfs}
\usepackage{stackengine}
\usepackage{amsthm}
\usepackage{stackengine,scalerel}
\usepackage{trfsigns}   % Laplace/Fourier Transformation
\usepackage{mathrsfs}
\newsavebox\foobox
\newlength{\foodim}
\newcommand{\slantbox}[2][0]{\mbox{%
        \sbox{\foobox}{#2}%
        \foodim=#1\wd\foobox
        \hskip \wd\foobox
        \hskip -0.5\foodim
        \pdfsave
        \pdfsetmatrix{1 0 #1 1}%
        \llap{\usebox{\foobox}}%
        \pdfrestore
        \hskip 0.5\foodim
}}

\def\Fourier{\slantbox[-.45]{$\mathscr{F}$}}

\usepackage{fancyhdr}
\newtheorem{theorem}{Theorem}[section]
\newtheorem{lemma}{Lemma}[section] %[theorem]

\newtheorem{proposition}{Proposition}[section] %[theorem]

\newtheorem{remark}{Remark}[section]  %[theorem]

\numberwithin{equation}{section}
\DeclareRobustCommand{\rchi}{{\mathpalette\irchi\relax}}
\newcommand{\irchi}[2]{\raisebox{\depth}{$#1\chi$}}

\newcommand\obullet[1]{\ThisStyle{\ensurestackMath{%
  \stackon[1pt]{\SavedStyle#1}{\SavedStyle\kern.6\LMpt\bullet}}}}
\newcommand\ocirc[1]{\ThisStyle{\ensurestackMath{%
  \stackon[1pt]{\SavedStyle#1}{\SavedStyle\kern.6\LMpt\circ}}}}
\title{Acoustic Cavitation using Resonating Micro-Bubbles.\\ Analysis in the Time-Domain}

\author{Arpan Mukherjee\footnote{Radon Institute (RICAM), Austrian Academy of
Sciences, Altenbergerstrasse 69, A-4040, Linz, Austria (arpan.mukherjee@oeaw.ac.at). This author is supported by the Austrian Science Fund (FWF): P32660.} \ and Mourad Sini\footnote{Radon Institute (RICAM), Austrian Academy of
Sciences, Altenbergerstrasse 69, A-4040, Linz, Austria (mourad.sini@oeaw.ac.at). This author is partially supported by the Austrian Science Fund (FWF): P32660.}}

\begin{document}
\maketitle
\begin{abstract}

  We study the time-domain acoustic wave propagation in the presence of a micro-bubble. This micro-bubble is characterized by a mass density and bulk modulus which are both very small as compared to the ones of the \textcolor{black}{uniform and homogeneous background medium}. The goal is to estimate the amount of pressure that is created very near (at a distance proportional to the radius of the bubble) to the bubble. We show that, at that small distance, the dominating field is reminiscent to the wave created by a point-like obstacle modeled formally by a Dirac-like heterogeneity with support at the location of the bubble and the contrast between the bubble and background material as the scattering coefficient. As a conclusion, we can tune the bubbles material properties so that the pressure near it reaches a desired amount. Such design might be useful in the purpose of acoustic cavitation where one needs enough, but not too much, pressure to eliminate unwanted anomalies. The mathematical analysis is done using time-domain integral equations and asymptotic analysis techniques. A well known feature here is that the contrasting scales between the bubble and the background generate resonances (mainly the Minnaert one) in the time-harmonic regime. Such critical scales, and the generated resonances, are also reflected in the time-domain estimation of the acoustic wave. In particular, reaching the desired amount of pressure near the location of the bubble is possible only with such resonating bubbles.
\end{abstract}

\section{Motivation and Results}

When micro-bubbles are subject to an ultrasonic field, in liquids, they start to grow and at high frequencies they can even collapse. Such a phenomenon is known as the acoustic cavitation. During the collapsing phase, the pressure (and eventually the temperature) will rise in the surrounding liquid. This phenomenon of growth and collapse from low to high amplitude of incident acoustic field is a non-linear behaviour which is modelled by the Rayleigh-Plesset mathematical model, or more generally the Keller-Miksis model which describes the evolution, or dynamics, of the radius of the spherical bubble, see \cite{Brennen, Yasui}. The mathematical model describing the dynamics of non-spherical bubbles is described in \cite{C-M-P-T-1, C-M-P-T-2}. In the current work, we are interested in  quantifying the amount of pressure that can be created around the micro-bubbles, of general shapes, when they are excited by incident acoustic with low to moderate amplitude. In such regimes, the change of size of the bubbles can be neglected and the propagation of the generated pressure can be modelled by the following linearized acoustic problem, see \cite{C-M-P-T-1, C-M-P-T-2},
\begin{align}\label{mainfor}
\begin{cases}   \mathrm{k}^{-1}(\mathrm{x})u_{\mathrm{t}\mathrm{t}}- \text{div} \rho^{-1}(\mathrm{x)}\nabla u = 0 \quad \quad \text{in} \quad (\mathbb{R}^{3}\setminus \partial\Omega) \times (0,\mathrm{T})\\
 u\big|_{+} = u\big|_{-}  \quad \quad \quad \quad \quad \quad \quad \quad \quad \quad \ \text{on} \quad \partial\Omega\\
 \rho_{\textcolor{black}{\mathrm{m}}}^{-1}\partial_\nu u\big|_{+} = \rho_{\textcolor{black}{\mathrm{c}}}^{-1} \partial_\nu u\big|_{-} \quad \quad \quad \quad \quad \text{on} \quad \partial \Omega \\
 u(\mathrm{x},0)= u_\mathrm{t}(\mathrm{x},0) = 0 \quad \quad \quad  \quad \quad \ \text{for} \ \mathrm{x} \in \mathbb{R}^3,
\end{cases}
\end{align}
where $\rho = \rho_{\mathrm{c}}\rchi_{\Omega} + \rho_{m}\rchi_{\mathbb{R}^{3}\setminus\overline{\Omega}}$ is the mass density and $\mathrm{k} = \mathrm{k}_{\mathrm{c}}\rchi_{\Omega} + \mathrm{k}_{\mathrm{m}}\rchi_{\mathbb{R}^{3}\setminus\overline{\Omega}}$ is the bulk modulus of the bubble and acoustic medium respectively. Moreover, $\partial_\nu$ denotes the outward normal vector and we use the notation $\partial_\nu \big|_{\pm}$ indicating $
    \partial_\nu u \big|_{\pm}(\mathrm{x},\mathrm{t}) = \lim_{\mathrm{h}\to 0}\nabla u(\mathrm{x}\pm \mathrm{h}\nu_\mathrm{x},\mathrm{t})\cdot \nu_\mathrm{x},$
where $\nu$ being the outward normal vector to $\partial\Omega$. 
\newline

\noindent
Our goal is to estimate the amount of pressure, $u$, that is created very near (at a distance proportional to the radius of the bubble) to the bubble. The motivation of this study comes from the proposed therapy modality using bubbles, see for instance \cite{roy, roy-2}. Another motivation comes from the drugs delivery using bubbles as vehicles. The main principle in such modality is that after injecting the bubbles near to the target to be cured, one applies ultrasound waves, as incident pressures, so that the bubbles will be compressed, and eventually collapses, which can result in vascular occlusion (or tissue destruction), see \cite{roy, roy-2}. Such modalities which aimed to be non-invasive therapies use very high ultrasound intensity which might have undesirable effects for the surroundings. However, reducing the intensity of the ultrasound should be compensated so that the desired high pressure could be reached locally near the bubbles. To achieve this goal, we analyse qualitatively and quantitatively the dependence of the generated pressure, near the injected bubble, in terms of the acoustic properties as well as the geometry of the bubbles coupled with acoustic properties of the background where the bubbles is injected. Such analysis might help tuning these parameters (i.e. geometry/acoustic properties of the bubble) to reach a desired pressure locally near the bubble.  
\newline

\noindent
To describe correctly the scales needed in the mathematical analysis, we assume the bubble to be of the form $\Omega = \delta \mathrm{B} + \mathrm{z}$, with a bounded \textcolor{black}{$\mathcal{C}^2$-regular} domain $\mathrm{B}$ in $\mathbb{R}^{3}$ which is centred at origin, where $\delta$ denotes its maximum \textcolor{black}{radius} and $\mathrm{z}$ represent\textcolor{black}{s} the position. The coefficients $\rho$ and $\mathrm{k}$ are also assumed to be piece-wise constants, with one constant outside $\Omega$, i.e. 
\begin{align}\label{con}
    \rho(\mathrm{x})\equiv \rho_\mathrm{m}, \quad \quad \mathrm{k}(\mathrm{x}) \equiv \mathrm{k}_\mathrm{m} \quad \quad \text{outside of the bounded domain}\ \Omega.
\end{align}
and other constants $\rho_\mathrm{c}$ and $\mathrm{k}_\mathrm{c}$ in $\Omega$ satisfying the following scaling properties
\begin{align}\label{cond-bubble}
    \rho_\mathrm{c} = \overline{\rho_\mathrm{c}}\delta^2, \quad \mathrm{k}_\mathrm{c} = \overline{\mathrm{k}_\mathrm{c}}\delta^2 \quad \text{and} \quad \frac{\rho_\mathrm{c}}{\mathrm{k}_\mathrm{c}} \sim 1 \ \text{as}\ \delta\ll1.
\end{align}
These scales mean that one should design the bubble so that its size and the type of materials (as its mass density and bulk modulus) should be linked via (\ref{cond-bubble}) and compared to the background via (\ref{con}).
\newline

\noindent
We consider acoustic incident waves of the form of causal point-sources
\begin{equation}
    u^i(\mathrm{x},\mathrm{t}, \mathrm{x}_0):=\frac{\lambda(\mathrm{t}-\mathrm{c}_0^{-1}\vert \mathrm{x}-\mathrm{x}_0\vert)}{\vert \mathrm{x}-\mathrm{x}_0\vert},
\end{equation}
where $\mathrm{x}_0$ is a source point located away from the micro-bubble $\Omega$ and $\lambda$ is a smooth function having support in the half line $(0, +\infty)$. We also denote by $\mathrm{c}_0= \sqrt{\frac{\mathrm{k}_\mathrm{m}}{\rho_\mathrm{m}}}$ the constant wave speed in $\mathbb{R}^3\setminus\overline{\Omega}$. Other types of incident acoustic waves could be used as well.
\bigbreak
\noindent
We now state the main result of this work. 
\begin{theorem}\label{1.1}
Assume that $\mathrm{B}$ is a bounded and $\textcolor{black}{\mathcal{C}^2}$-regular domain in $\mathbb{R}^3$ and the two coefficients $\rho$ and $\mathrm{k}$ satisfy the conditions (\ref{con}) and (\ref{cond-bubble}). \textcolor{black}{In addition, we assume that $\lambda:\mathbb R \to \mathbb R$ is a causal function and of class $\mathcal{C}^9(\mathbb R).$\footnote{\textcolor{black}{This required order of regularity is explained in Remark \ref{r2.1}.}}} Let $u:=u^\mathrm{i}+u^\mathrm{s}$ be the solution of the hyperbolic problem (\ref{mainfor}).
Let us consider $\mathrm{x}\in \mathbb{R}^3\setminus\overline{\Omega}$ such that $\textbf{dist}(\mathrm{x},\Omega)\sim \delta^\mathrm{q}$ and therefore, $|\mathrm{x}-\mathrm{z}| \sim \delta +\delta^\mathrm{q}$ where $\mathrm{q}\in [0,1]$. Then we have the following approximation of $u^s$ 
\begin{align}\label{main-estimate}
    u^\mathrm{s}(\mathrm{x},\mathrm{t})
    = \frac{\omega_\mathrm{M}\rho_\mathrm{m}|\mathrm{B}|}{4\pi\overline{\mathrm{k}_\mathrm{c}}} \delta \frac{1}{|\partial\Omega|}\int_{\partial\Omega}\frac{1}{|\mathrm{x}-\mathrm{y}|}d\sigma_\mathrm{y}\int_0^{\mathrm{t}-\mathrm{c}_0^{-1}|\mathrm{x}-\mathrm{z}|} \sin\big(\omega_\mathrm{M}(\mathrm{t}-\mathrm{c}_0^{-1}|\mathrm{x}-\mathrm{z}|-\tau)\big)u_{\mathrm{t}\mathrm{t}}^\mathrm{i}(\mathrm{z},\mathrm{\tau}) d\tau + \mathcal{O}(\delta^{2-\mathrm{q}}),
\end{align}
where $\omega_\mathrm{M} := \sqrt{
\frac{2\overline{\mathrm{k}_\mathrm{c}}}{\mathrm{A}_{\partial\mathrm{B}}\rho_\mathrm{m}}}$ and $\displaystyle
    \mathrm{A}_{\partial\mathrm{B}} := \frac{1}{|\partial\mathrm{B}|}\int_{\partial\mathrm{B}}\int_{\partial\mathrm{B}}\frac{(\mathrm{x}-\mathrm{y})\cdot\nu}{|\mathrm{x}-\mathrm{y}|}d\sigma_\mathrm{x}d\sigma_\mathrm{y}.
$
\end{theorem}

\noindent
The dominant part in (\ref{main-estimate}) is reminiscent the wave field generated by a point-like obstacle, in the time-domain acoustic wave propagation, supported on the location $\mathrm{z}$ of the bubble with a scattering coefficient given by $\frac{\omega_\mathrm{M}\rho_\mathrm{m}|\mathrm{B}|}{\overline{\mathrm{k}_\mathrm{c}}} \delta$. This is formally the solution for the acoustic wave problem with a singular potential (or speed of propagation), of Dirac type, supported on the point $\mathrm{z}$. The constant $\omega_\mathrm{M}$ is the Minnaert frequency that is known in the wave propagation in the presence of bubble in the time-harmonic regime, see \cite{Habib-Minnaert, Ahcene-bubbles}. Actually, if we take the formal Fourier transform of the dominant term in  (\ref{main-estimate}), then we end-up with the dominant term of the acoustic wave in the harmonic regime (see \cite{Habib-Minnaert, Ahcene-bubbles}).
\newline

\noindent
Assume for simplicity, here, that $\Omega$ is a sphere (of center $z$ and radius $\delta$). Then we can show that the function $x \rightarrow\displaystyle \mathrm{Q}(\mathrm{x}):=\frac{1}{|\partial\Omega|}\int_{\partial\Omega}\frac{1}{|\mathrm{x}-\mathrm{y}|}d\sigma_\mathrm{y}$ has the following behavior $\mathrm{Q}(\mathrm{x})=\vert \mathrm{x}-\mathrm{z}\vert^{-1}$ for $\mathrm{x}\in \mathbb{R}^3\setminus{\Omega}$ and $\mathrm{Q}(\mathrm{x})=\delta^{-1}$ for $\mathrm{x} \in \overline{\Omega}$. Therefore, the dominating term in (\ref{main-estimate}) behaves as $\mathrm{R}(\mathrm{x})\; \delta^{1-\mathrm{q}}$ for $\mathrm{x} \in \mathbb{R}^3\setminus{\overline{\Omega}}$ such that $\textbf{dist}(\mathrm{x},\Omega)= \delta^\mathrm{q}$, $\mathrm{q} \in [0, 1]$, where $\mathrm{R}(\mathrm{x})$ is controllable knowing the geometry and acoustic properties of the bubble. As we can tune the bubble's properties, through its size and eventually the bulk-related constant $\overline{\mathrm{k}}_\mathrm{c}$, we can reach any desired value of the pressure on and near the bubble while it decreases away from it.
\newline

%Therefore, the amplitude of the wave behaves as $A(x)\; \delta^{1-q}$ for $\textbf{dist}(\mathrm{x},\Omega)= \delta^\mathrm{q}$, $q \in [0, 1]$, and then $A(x)$ for $x \in \Omega$. These properties suggest the behavior, shown in Fig \ref{Fig}, of this dominant part. 
\noindent
The estimation of the field until a distance $\delta$ to the bubble is derived and justified while the estimate inside the bubble is not yet justified. However, we do believe that the expansion in  (\ref{main-estimate}) is also valid inside the bubble $\Omega$ as well. Indeed, in the time-harmonic regime, such estimates everywhere in the space are already justified using a different approach, based on resolvent estimates of singularly perturbed Laplacian, see \cite{Man-Pos-Sin}. Extending such everywhere-estimates to the time-domain setting is highly desirable and this topic will be considered in the future. \textcolor{black}{Finally, let us mention the two contributions, \cite{liu2} and \cite{liu1}, where the authors modeled and analyzed the elastic waves generated by an injected acoustic bubble. The analysis of acoustic cavitation through elastic interrogations can have useful applications in industry, and we plan to extend our analysis to this model in the time domain as well.}
\newline

\noindent
Let us now analyse more closely the form of the time dependent term appearing in the dominant term (\ref{main-estimate}). Using integration by parts and the zero initial conditions satisfied by $u^{i}$, we show that
$\displaystyle
\int_0^{\mathrm{t}-\mathrm{c}_0^{-1}|\mathrm{x}-\mathrm{z}|} \sin\big(\omega_\mathrm{M}(\mathrm{t}-\mathrm{c}_0^{-1}|\mathrm{x}-\mathrm{z}|-\tau)\big)u_{\mathrm{t}\mathrm{t}}^\mathrm{i}(\mathrm{z},\mathrm{\tau}) d\tau$ $$=\omega^2_\mathrm{M} \int_0^{\mathrm{t}-\mathrm{c}_0^{-1}|\mathrm{x}-\mathrm{z}|} \sin\big(\omega_\mathrm{M}(\mathrm{t}-\mathrm{c}_0^{-1}|\mathrm{x}-\mathrm{z}|-\tau)\big)u^\mathrm{i}(\mathrm{z},\mathrm{\tau}) d\tau - \omega_\mathrm{M} u^i(t-\mathrm{c}_0^{-1}|\mathrm{x}-\mathrm{z}|).
$$
With this decomposition, we see that the dominant part in (\ref{main-estimate}) decomposes into two reflected waves:
$$
\mathrm{U}_1(\mathrm{x},\mathrm{t}):=\frac{\omega^2_\mathrm{M}\rho_\mathrm{m}|\mathrm{B}|}{4\pi\overline{\mathrm{k}_\mathrm{c}}} \delta \frac{1}{|\partial\Omega|}\int_{\partial\Omega}\frac{1}{|\mathrm{x}-\mathrm{y}|}d\sigma_\mathrm{y}\; u^i(t-\mathrm{c}_0^{-1}|\mathrm{x}-\mathrm{z}|),
$$
which we call the primary reflected wave, and
$$
\mathrm{U}_2(\mathrm{x},\mathrm{t}):=\frac{\omega^3_\mathrm{M}\rho_\mathrm{m}|\mathrm{B}|}{4\pi\overline{\mathrm{k}_\mathrm{c}}} \delta \frac{1}{|\partial\Omega|}\int_{\partial\Omega}\frac{1}{|\mathrm{x}-\mathrm{y}|}d\sigma_\mathrm{y}\int_0^{\mathrm{t}-\mathrm{c}_0^{-1}|\mathrm{x}-\mathrm{z}|} \sin\big(\omega_\mathrm{M}(\mathrm{t}-\mathrm{c}_0^{-1}|\mathrm{x}-\mathrm{z}|-\tau)\big)u^\mathrm{i}(\mathrm{z},\mathrm{\tau}) d\tau
$$
which we call the secondary reflected wave.
\bigskip

\begin{enumerate}
\item The primary reflected wave is nothing but the incident wave time-shifted, with $\mathrm{c}_0^{-1}|\mathrm{x}-\mathrm{z}|$, and 'amplified' with the amplitude $\displaystyle\frac{\omega^2_\mathrm{M}\rho_\mathrm{m}|\mathrm{B}|}{4\pi\overline{\mathrm{k}_\mathrm{c}}} \delta \frac{1}{|\partial\Omega|} \int_{\partial\Omega}\frac{1}{|\mathrm{x}-\mathrm{y}|}d\sigma_\mathrm{y}=\frac{|\mathrm{B}|}{A_{\partial B}}\frac{\delta}{2\pi|\partial\Omega|} \int_{\partial\Omega}\frac{1}{|\mathrm{x}-\mathrm{y}|}d\sigma_\mathrm{y}\sim \frac{|\mathrm{B}|}{\mathrm{A}_{\partial \mathrm{B}}} \delta^{1-\mathrm{q}}$ for $x\in \mathbb{R}^3\setminus{\overline{\Omega}}$ with $\textbf{dist}(\mathrm{x},\Omega)= \delta^\mathrm{q}$. \textcolor{black}{Therefore, this term can be amplified until an order limited by the volume/area ratio $\frac{|\mathrm{B}|}{\mathrm{A}_{\partial \mathrm{B}}}$. It is reasonable to think that the maximum ratio will be reached for a sphere-shaped bubble.} 
\bigskip

\item The secondary reflected wave appears as a resonant (i.e. oscillating) field. The multiplying coefficient, that we call its amplitude, $\displaystyle\frac{\omega^3_\mathrm{M}\rho_\mathrm{m}|\mathrm{B}|}{4\pi\overline{\mathrm{k}_\mathrm{c}}} \delta \frac{1}{|\partial\Omega|}\int_{\partial\Omega}\frac{1}{|\mathrm{x}-\mathrm{y}|}d\sigma_\mathrm{y}$ has the form $\omega_\mathrm{M}\; \frac{|\mathrm{B}|}{A_{\partial B}} \delta^{1-\mathrm{q}}$ for $x\in \mathbb{R}^3\setminus{\overline{\Omega}}$ with $\textbf{dist}(\mathrm{x},\Omega)= \delta^\mathrm{q}$. As $\omega^2_\mathrm{M}$ is proportional to the scaled bulk modulus of the bubble, i.e. $\overline{\mathrm{k}}_\mathrm{c}$, therefore this amplitude can be increased by choosing the bubble with smaller bulk modulus $\overline{\mathrm{k}}_\mathrm{c}$ (\textcolor{black}{and not limited by the ratio $\frac{|\mathrm{B}|}{\mathrm{A}_{\partial \mathrm{B}}}$}). 
\newline

To show better understanding of the behavior of this 'oscillating' field, we consider the extreme case where the incident field $u^i$ is given by a wavefront $u^i( \mathrm{y},\mathrm{t}):=\frac{\delta(\mathrm{t}- \mathrm{c}_{\textcolor{black}{0}}^{-1}\vert \mathrm{y}-\mathrm{x}_0\vert)}{\vert \mathrm{y}-\mathrm{x}_0\vert}$. In this case, we have
$$
\mathrm{U}_2(\mathrm{x},\mathrm{t}):=\frac{\omega^3_\mathrm{M}\rho_\mathrm{m}|\mathrm{B}|}{4\pi\overline{\mathrm{k}_\mathrm{c}}} \delta \frac{1}{|\partial\Omega|}\int_{\partial\Omega}\frac{1}{|\mathrm{x}-\mathrm{y}|}d\sigma_\mathrm{y} \frac{\sin\big(\omega_\mathrm{M}(\mathrm{t}-\mathrm{c}_0^{-1}|\mathrm{x}-\mathrm{z}|-\mathrm{c}_0^{-1}|\mathrm{z}-\mathrm{x}_0|)\big)}{|\mathrm{z}-\mathrm{x}_0|}.
$$
This expression says, in particular, that there are times $\mathrm{t}$ when the sinus term reaches its maximum value $1$. Therefore, choosing properly the scaled bulk/size of the injected bubble would produce a desired pressure at any point $\mathrm{x}$ near the bubble, and at certain times.
\end{enumerate}

\noindent
To derive the asymptotic expansion in (\ref{main-estimate}), we use time-domain integral equation methods. The analysis is based on the time-domain Lippmann-Schwinger equation. To  characterize the dominating term, we reduce the computations to the stationary case where, due to the used scales of the bubble's mass density and bulk modulus, we retrieve the resonant behavior of the reflect field in the lines of \cite{Ahcene-bubbles} . A key point in the analysis, which is the main difficult part, is the derivation of a priori estimates to control the correcting and the remaining terms.  To derive these estimates, we use the approach by Lubich, see \cite{lubich}, to reduce the estimates to the Laplace domain, with the Laplace variable away from the real line and control the estimates in terms of this Laplace variable in appropriate weighed spaces. In the Laplace domain, we study the invertibility of the Lippmann-Schwinger operator using carefully (spectral) properties of the Newtonian as well as Magnetization type operators. The control of these estimates in terms of both the spacial-scales of the bubble properties (in terms of size, mass density and bulk modulus) and the Laplace variable in the weighed spaces is quite involved. This approach has been already used in \textcolor{black}{\cite{sini2}} where we needed to estimate 'only' the invertibility properties of the single layer operator. In the present work, we need to handle the full Lippmann-Schwinger operator which involves both the Newtonian and Magnetization type operators.  

\bigbreak

\noindent
The remaining part of the manuscript is divided into three sections. In Section \ref{sec-2}, we give the detailed proof of Theorem \ref{1.1} using some claimed a priori estimates. These estimates are proved in Section \ref{apri}. In Section \ref{appen}, as an appendix, we provide \textcolor{black}{a} few technical estimates that were used in Section \ref{sec-2} and Section \ref{apri}. 

\bigbreak

\noindent
In this manuscript, we use the notation $'\lesssim'$ to denote $'\le'$ with its right-hand side multiplied by a generic positive constant.

%--------------------------------------------------------------------------
%--------------------------------------------------------------------------

\section{Proof of Theorem \ref{1.1}}\label{sec-2} % Section

\textcolor{black}{In this section, we provide the asymptotic behaviour of the acoustic pressure field $u(\mathrm{x},\mathrm{t})$ to the solution of (\ref{mainfor}) as $\delta \ll 1$ for $(\mathrm{x},\mathrm{t})\in \mathbb{R}^3\setminus\overline{\Omega} \times (0,\mathrm{T}).$}

\subsection{Function Spaces}     % Subsection

\textcolor{black}{We start this section by introducing the appropriate Sobolev spaces used in the analysis to derive the needed a priori estimates.} For details on those spaces, we refer to \cite{hduong, hduong2, monk, lubich, sini}. To begin, we set
\begin{align}\nonumber
    \mathrm{H}_0^\mathrm{r}(0,\mathrm{T}) := \Big\{\mathrm{g}|_{(0,\mathrm{T})}: \mathrm{g} \in \mathrm{H}^\mathrm{r}(\mathbb{R})\ \text{with} \ \mathrm{g}\equiv 0\ \text{in}\ (-\infty,0)\Big\}, \quad \mathrm{r} \in\mathbb{R}.
\end{align}
We then denote $\mathcal{D}(\mathbb{R},\mathrm{X})$ as the corresponding space of smooth and compactly supported function with images in the Hilbert space $\mathrm{X}$. Accordingly, we define $\mathcal{D}'(\mathbb{R}_+,\mathrm{X})$ as the $\mathrm{X}$-valued distributions on $\mathbb{R}$ that vanishes as $\mathrm{t}<0$ and tempered distributions are defined as $\mathcal{S}'(\mathbb{R}_+,\mathrm{X})$. We then set
\begin{align}\nonumber
    \mathcal{L}'(\mathbb{R}_+,\mathrm{X}):= \Big\{\mathrm{f}\in \mathcal{D}'(\mathbb{R}_+,\mathrm{X}): \mathrm{e}^{-\sigma \mathrm{t}}\mathrm{f}\in \mathcal{S}'(\mathbb{R}_+,\mathrm{X})\quad \text{for some} \quad \sigma>0.
\end{align}
Afterwards, we define the following function space
\begin{align}\nonumber
    \mathrm{H}^\mathrm{r}_{0,\sigma}(0,\mathrm{T};\mathrm{X}) := \Big\{\mathrm{f} \in \mathcal{L}'(\mathbb{R}_+,\mathrm{X}): \mathrm{e}^{-\sigma \mathrm{t}}\partial_\mathrm{t}^\mathrm{k}\mathrm{f} \in \mathrm{L}_0^2(0,\mathrm{T};\mathrm{X}), \ \mathrm{k}=1,...\mathrm{r}\Big\},\quad \mathrm{r}\in \mathbb Z_+,
\end{align}
with the following norm
\begin{align}
    \mathrm{H}^\mathrm{r}_{0,\sigma}(0,\mathrm{T};\mathrm{X}) := \Big(\int_0^\mathrm{T}\mathrm{e}^{-2\sigma \mathrm{t}}\Big[\Vert \mathrm{f}\Vert_\mathrm{X}^2 + \sum_{\mathrm{k}=1}^\mathrm{r} \mathrm{T}^{2\mathrm{k}}\Vert \partial_\mathrm{t}^\mathrm{k}\mathrm{f} \Vert_\mathrm{X}^2\Big]d\mathrm{t}\Big)^\frac{1}{2}.
\end{align}
As a next step, we state an equivalent norm for $u \in \mathrm{H}^{\mathrm{s}}(\partial\Omega)$ according to the Slobodeckij norm for $\mathrm{s}\in (0,1)$, \cite[pp. 20]{Grisvard}.
\begin{align}\label{normd1}
    \Vert \mathrm{f}\Vert_{\mathrm{H}^{\mathrm{s}}(\partial\Omega)} := \Vert \mathrm{f} \Vert_{\mathrm{L}^2(\partial \Omega)}^2 + \int_{\partial\Omega}\int_{\partial\Omega}\frac{|\mathrm{f}(\mathrm{x})-\mathrm{f}(\mathrm{y})|^2}{|\mathrm{x}-\mathrm{y}|^{2+2\mathrm{s}}}d\sigma_\mathrm{x}d\sigma_\mathrm{y},
\end{align}
and we define the following dual norm for $\varphi \in \mathrm{H}^{-\frac{1}{2}}(\partial\Omega)$
\begin{align}\label{normd2}
 \Vert \varphi \Vert_{\mathrm{H}^{-\frac{1}{2}}(\partial\Omega)} &= \sup_{0\neq \mathrm{f} \in \mathrm{H}^{\frac{1}{2}}(\partial\Omega )} \dfrac{|\langle \varphi, \mathrm{f} \rangle_{\partial \Omega}|}{\Vert \mathrm{f} \Vert_{\mathrm{H}^{\frac{1}{2}}(\partial\Omega )}},
 \end{align}
\textcolor{black}{where $\langle\cdot,\cdot\rangle_{\partial\Omega}$ denotes the duality pairing between $\mathrm{H}^{\frac{1}{2}}(\partial\Omega)$ and $\mathrm{H}^{-\frac{1}{2}}(\partial\Omega)$. Let us also denote
\begin{align}
    \mathrm{H}_0^{-\frac{1}{2}}(\partial\Omega):= \Big\{\mathrm{f}\in \mathrm{H}^{-\frac{1}{2}}(\partial\Omega): \int_{\partial\Omega}\mathrm{f}(\mathrm{x})d\sigma_\mathrm{x} =0  \Big\}\; \text{and analogously we define}\ \mathrm{H}_0^{\frac{1}{2}}(\partial\Omega).  
\end{align}
In our analysis, we require several function spaces, which we denote using blackboard bold font to represent vector fields within $\mathbb R^3$. We first introduce the following function spaces
\begin{align}\nonumber
   \mathbb{H}(\text{div},\Omega):= \Big\{ \mathrm{f}\in \big(\mathrm{L}^{2}(\Omega)\big)^3 :\; \text{div}\; \mathrm{f} \in \mathrm{L}^{2}(\Omega)\Big\} \; \text{and}\
   \mathbb{H}(\text{curl},\Omega):= \Big\{ \mathrm{f}\in \big(\mathrm{L}^{2}(\Omega)\big)^3 :\; \text{curl}\; \mathrm{f} \in \big(\mathrm{L}^{2}(\Omega)\big)^3\Big\}.
\end{align}
We then consider the space of divergence-free as well as the space of irrotational vector fields 
\begin{align}\nonumber
   \mathbb{H}(\text{div}\;0,\Omega):= \Big\{ \mathrm{f}\in \mathbb{H}(\text{div},\Omega) :\; \text{div}\; \mathrm{f}= 0 \Big\} \; \text{and}\
   \mathbb{H}(\text{curl}\; 0,\Omega):= \Big\{ \mathrm{f}\in \mathbb{H}(\text{curl},\Omega) :\; \text{curl}\; \mathrm{f} = 0\Big\},
\end{align}
and their sub-spaces
\begin{align}\nonumber
   \mathbb{H}_0(\text{div}\;0,\Omega):= \Big\{ \mathrm{f}\in \mathbb{H}(\text{div}\;0,\Omega) :\; \nu\cdot\mathrm{f}= 0\; \text{on}\ \partial\Omega \Big\} \; \text{and}\
   \mathbb{H}_0(\text{curl}\; 0,\Omega):= \Big\{ \mathrm{f}\in \mathbb{H}(\text{curl}\;0,\Omega) :\; \nu \times \mathrm{f} = 0 \; \text{on}\ \partial\Omega\Big\},
\end{align}
respectively. We also use the grad-harmonic sub-space
\begin{align}\nonumber
    \nabla \mathbb{H}_{\text{arm}} := \Big\{ u \in \big(\mathrm{L}^{2}(\Omega)\big)^3: \exists \  \varphi \ \text{s.t.} \ u = \nabla \varphi,\; \varphi\in \mathrm{H}^1(\Omega)\; \text{and} \ \Delta \varphi = 0 \Big\}.
\end{align}
Throughout this paper, we use the notation $\mathcal{L}(\mathrm{X}; \mathrm{Z})$ to refer to the set of linear bounded operators, defined from $\mathrm{X}$ to $\mathrm{Z}$. Additionally, we define $\mathcal{L}(\mathrm{X})$ to be the same as $\mathcal{L}(\mathrm{X}; \mathrm{X})$. Furthermore, we use the standard Sobolev space of order $\mathrm{s}$ on $\Omega$, which we denote as $\mathrm{H}^\mathrm{s}(\Omega)$. Let us finally mention that the operators we employ are defined using a bold symbol.}

%------------------------------------------------------------------------------
%------------------------------------------------------------------------------

\subsection{Asymptotic Expansion of the Acoustic Pressure Field}   % Subsection

As the incident wave $u^i$ satisfies
\begin{align}\label{wave1}
    \mathrm{k}_\mathrm{m}^{-1}u_{\mathrm{t}\mathrm{t}}^\mathrm{i}- \nabla \cdot\rho_\mathrm{m}^{-1}\nabla u^\mathrm{i} = 0 \quad \quad \text{in} \quad \mathbb{R}^3\setminus{\{\mathrm{x}_0\}}\times (0,\mathrm{T}), \ \text{with}\ \mathrm{x}_0 \in \mathbb{R}^3\setminus\overline{\Omega},
\end{align}
then the scattered wave $u^\mathrm{s}$ satisfies 
\begin{align}\label{wellposed}
    \mathrm{k}^{-1}u_{\mathrm{t}\mathrm{t}}^\mathrm{s}- \text{div}\rho^{-1}\nabla u^\mathrm{s} = (\mathrm{k}_\mathrm{m}^{-1} - \mathrm{k}^{-1})u^\mathrm{i}_{\mathrm{t}\mathrm{t}} - \text{div} (\rho_\mathrm{m}^{-1} - \rho^{-1})\nabla u^\mathrm{i},\ \text{in}\ \mathbb{R}^3\times(0,\mathrm{T}),
\end{align}
with zero initial condition. \textcolor{black}{We denote $\mathrm{F}(\mathrm{x},\mathrm{t}):= (\mathrm{k}_\mathrm{m}^{-1} - \mathrm{k}^{-1})u^\mathrm{i}_{\mathrm{t}\mathrm{t}} - \text{div} (\rho_\mathrm{m}^{-1} - \rho^{-1})\nabla u^\mathrm{i}.$ The following Lemma states the well-posedness and same regularity properties of (\ref{wellposed}) $\big(\text{or}\ (\ref{mainfor})\big).$
\newline
\begin{lemma} \label{wellpose}
    Let $\mathrm{F}$ belongs to $\mathrm{H}^{\mathrm{r}+1}_{0,\sigma}\big(0,\mathrm{T};\mathrm{H}^{-1}(\mathbb{R}^3)\big)$ that is supported in the region $\Omega\times(0,\mathrm{T})$ and $\mathrm{r}\in \mathbb N$. There exists a unique solution $u:= u^\mathrm{i}+u^\mathrm{s}$, such that $u^\mathrm{s}$ belongs to $\mathrm{H}^\mathrm{r}_{0,\sigma}\big(0,\mathrm{T};\mathrm{H}^{1}(\mathbb{R}^3)\big)$ for the problem (\ref{mainfor}).
\begin{proof}
See Section \ref{well1}.
\end{proof}
\end{lemma}
\noindent
The well-posedness of (\ref{mainfor}), seen as a transmission problem is studied in \cite{sayas-transmission}. In Lemma \ref{wellpose}, we also derive the regularity-in-time properties that are needed in our subsequent analysis, see Remark \ref{r2.1}.
\begin{remark}\label{r2.1}
Assuming that $\lambda \in \mathrm{C}^9(\mathbb R)$ and has its support in the positive real line, then by choosing $\mathrm{r}=8$, it follows that $\mathrm{F} \in \mathrm{H}^9_{0,\sigma}\big(0,\mathrm{T};\mathrm{H}^{-1}(\mathbb{R}^3)\big)$. As a result, $u^\mathrm{s} \in \mathrm{H}^8_{0,\sigma}\big(0,\mathrm{T};\mathrm{H}^{1}(\mathbb{R}^3)\big)$ and then in $\mathcal{C}^7\big(0,\mathrm{T};\mathrm{H}^{1}(\mathbb{R}^3)\big)$. 
\newline
Additionally, we have the equation $\nabla\cdot\rho^{-1}\nabla u^\mathrm{s}= \mathrm{F}-\mathrm{k}^{-1}u^\mathrm{s}_{\mathrm{t}\mathrm{t}}$, where the right hand side belongs to $\mathrm{H}^6_{0,\sigma}\big(0,\mathrm{T};\mathrm{L}^{2}(\Omega)\big).$ This implies that $\rho^{-1}\nabla u^\mathrm{s} \in \mathrm{H}^6_{0,\sigma}\big(0,\mathrm{T};\mathbb{H}(\text{div},\Omega)\big).$ Hence, $\partial_\nu u^\mathrm{s}$ is in $\mathrm{H}^6_{0,\sigma}\big(0,\mathrm{T};\mathrm{H}^{-\frac{1}{2}}(\partial\Omega)\big)$ and eventually, $\partial_\nu u^\mathrm{s}\in \mathcal{C}^5\big(0,\mathrm{T};\mathrm{H}^{-\frac{1}{2}}(\partial\Omega)\big).$ These regularity properties of $u^\mathrm{s}$ and $\partial_\nu u^\mathrm{s}$ are used in the following asymptotic analysis.
\end{remark}}
\noindent
The scattered wave $u^\mathrm{s}$ satisfies the following
\begin{align} \label{wave2}
\nonumber
    \mathrm{k}_\mathrm{m}^{-1}u_{\mathrm{t}\mathrm{t}}^\mathrm{s}- \text{div}\rho_\mathrm{m}^{-1}\nabla u^\mathrm{s} &= \mathrm{k}_\mathrm{m}^{-1}u_{\mathrm{t}\mathrm{t}}- \text{div}\rho_\mathrm{m}^{-1}\nabla u
    \\ \nonumber &= \mathrm{k}_\mathrm{m}^{-1}u_{\mathrm{t}\mathrm{t}}+ \mathrm{k}^{-1}(\mathrm{x}) u_{\mathrm{t}\mathrm{t}}- \mathrm{k}^{-1}(\mathrm{x})u_{\mathrm{t}\mathrm{t}}- \text{div}\rho_\mathrm{m}^{-1}\nabla u + \text{div}\rho^{-1}(\mathrm{x)}\nabla u - \text{div}\rho^{-1}(\mathrm{x})\nabla u
    \\ &= (\mathrm{k}_\mathrm{m}^{-1} - \mathrm{k}^{-1}(\mathrm{x}))u_{\mathrm{t}\mathrm{t}} - \text{div} (\rho_\mathrm{m}^{-1} - \rho^{-1}(\mathrm{x}))\nabla u, \quad \text{in}\ \mathbb{R}^3\times(0,\mathrm{T}).
\end{align}
We set $\mathrm{c}:=\sqrt{\frac{\mathrm{k}}{\rho}}$ as bulk modulus and mass density may generate a different propagation speed of the pressure field inside and outside of the bubble $\Omega$. Then we consider that the background velocity given by $\mathrm{c}_0= \sqrt{\frac{\mathrm{k}_\mathrm{m}}{\rho_\mathrm{m}}}$ and we denote the unperturbed Green function of the background medium satisfying the wave equation (\ref{wave1}) by
\begin{align}
    \mathrm{G}(\mathrm{x},\mathrm{t}) = \rho_\mathrm{m}\frac{\delta_0(\mathrm{t}-\mathrm{c}_0^{-1} |\mathrm{x}|)}{4\pi |\mathrm{x}|}\quad \textcolor{black}{\text{in}\; \mathbb{R}^3\times \mathbb{R},}
\end{align}
where $\delta_0$ is the Dirac delta distribution and $\mathbb{G}$ is known as the fundamental solution of the wave equation. Let $u^\mathrm{f}$ be the solution of
$\frac{\rho_\mathrm{m}}{\mathrm{k}_\mathrm{m}}u_{\mathrm{t}\mathrm{t}}^\mathrm{f} - \Delta u^\mathrm{f} = \mathrm{f}$, in $\mathbb{R}^3$ with $u^\mathrm{f}(\mathrm{x},0) = u_\mathrm{t}^\mathrm{f}(\mathrm{x},0)=0$, where $\mathrm{f}\in \textcolor{black}{\mathrm{H}^{\mathrm{r}+1}_{0,\sigma}}\big(0,\mathrm{T};\mathrm{H}^{-1}(\mathbb{R}^3)\big)$ has a compact support. Then we have
\begin{align}
    u^\mathrm{f} = \bm{\mathrm{V}}_\Omega\big[\mathrm{f}\big](\mathrm{x},\mathrm{t}) := \int_\mathbb{R}\int_\Omega \mathrm{G}(\mathrm{x}-\mathrm{y},\mathrm{t}-\tau)\mathrm{f}(\mathrm{y},\tau)d\mathrm{y}d\tau, \quad \text{for}, \quad (\mathrm{x},\mathrm{t})\in \mathbb{R}^3\times \mathbb{R},
\end{align}
and it lies in $\textcolor{black}{\mathrm{H}^\mathrm{r}_{0,\sigma}}\big(0,\mathrm{T};\mathrm{H}^{1}(\mathbb{R}^3)\big)$.
\newline

\noindent
Therefore, we deduce the following Lippmann-Schwinger equation for (\ref{wave2})
\begin{align}\label{ls}
    u + (\mathrm{k}_\mathrm{c}^{-1}- \mathrm{k}_\mathrm{m}^{-1})\int_{\textcolor{black}{\mathbb R}}\int_\Omega \mathrm{G}(\mathrm{x}-\mathrm{y};\mathrm{t}-\tau) u_{\mathrm{t}\mathrm{t}}\ d\mathrm{y}d\tau - (\rho_\mathrm{c}^{-1}- \rho_\mathrm{m}^{-1})\text{div}\int_{\textcolor{black}{\mathbb R}}\int_\Omega \mathrm{G}(\mathrm{x}-\mathrm{y};\mathrm{t}-\tau)\nabla u \ d\mathrm{y}d\tau = u^\mathrm{i}.
\end{align}
Let us now denote $\beta := \mathrm{k}_\mathrm{c}^{-1}- \mathrm{k}_\mathrm{m}^{-1} \ \text{and} \ \alpha: = \rho_\mathrm{c}^{-1} - \rho_\mathrm{m}^{-1} $. Here we introduce the basic concepts of acoustic layer potentials in the time domain  and how they can be used to represent the solution of the problem given by (\ref{wave2}) and with zero initial conditions.
\newline

\noindent
\textcolor{black}{We start by rewriting the volume integral equation (\ref{ls})} using retarded potentials and integration by parts as follows:
\begin{align}\label{3.6}
    u + \gamma \int_\Omega\frac{\rho_\mathrm{m}}{4\pi|\mathrm{x}-\mathrm{y}|}u_{\mathrm{t}\mathrm{t}}(\mathrm{y},\mathrm{t}-\mathrm{c}_0^{-1}|\mathrm{x}-\mathrm{y}|)d\mathrm{y} + \alpha \int_{\partial \Omega} \frac{\rho_\mathrm{m}}{4\pi|\mathrm{x}-\mathrm{y}|}\partial_\nu u(\mathrm{y}, \mathrm{t}- \mathrm{c}_0^{-1}|\mathrm{x}-\mathrm{y}|)d\sigma_\mathrm{y} = u^\mathrm{i},
\end{align}
where we denote by $\gamma := \beta-\alpha \frac{\rho_\mathrm{c}}{\mathrm{k}_\mathrm{c}}$. 
\newline

\noindent
For $\mathrm{x}$ outside $\Omega$ we rewrite (\ref{3.6}) as
\begin{align}\label{f1}
    u^\mathrm{s}(\mathrm{x},\mathrm{t})
    &=\nonumber -  \gamma\int_\Omega\frac{\rho_\mathrm{m}}{4\pi|\mathrm{x}-\mathrm{y}|} u_{\mathrm{t}\mathrm{t}}(\mathrm{y},\mathrm{t}-\mathrm{c}_0^{-1}|\mathrm{x}-\mathrm{y}|)d\mathrm{y} - \alpha \int_{\partial \Omega} \frac{\rho_\mathrm{m}}{4\pi|\mathrm{x}-\mathrm{y}|}\partial_\nu u(\mathrm{y}, \mathrm{t}- \mathrm{c}_0^{-1}|\mathrm{x}-\mathrm{y}|)d\sigma_\mathrm{y}
    \\ \nonumber &= -\frac{\rho_\mathrm{m}}{4\pi}\Bigg[\frac{\gamma}{|\mathrm{x}-\mathrm{z}|}\int_\Omega u_{\mathrm{t}\mathrm{t}}(\mathrm{y},\mathrm{t}-\mathrm{c}_0^{-1}|\mathrm{x}-\mathrm{z}|)d\mathrm{y} + \gamma \int_\Omega u_{\mathrm{t}\mathrm{t}}(\mathrm{y},\mathrm{t}-\mathrm{c}_0^{-1}|\mathrm{x}-\mathrm{z}|) \Big(\frac{1}{|\mathrm{x}-\mathrm{y}|}- \frac{1}{|\mathrm{x}-\mathrm{z}|}\Big)d\mathrm{y}
    \\ \nonumber &-\gamma \int_\Omega\frac{u_{\mathrm{t}\mathrm{t}}(\mathrm{y},\mathrm{t}-\mathrm{c}_0^{-1}|\mathrm{x}-\mathrm{z}|)- u_{\mathrm{t}\mathrm{t}}(\mathrm{y},\mathrm{t}-\mathrm{c}_0^{-1}|\mathrm{x}-\mathrm{y}|)}{|\mathrm{x}-\mathrm{y}|}d\mathrm{y} + \frac{\alpha}{|\partial\Omega|}\int_{\partial\Omega}\frac{1}{|\mathrm{x}-\mathrm{y}|}\int_{\partial\Omega}\partial_\nu u(\mathrm{y},\mathrm{t}-\mathrm{c}_0^{-1}|\mathrm{x}-\mathrm{y}|)d\sigma_\mathrm{y} 
    \\ &+ \alpha\int_{\partial\Omega}\Big(\frac{1}{|\mathrm{x}-\mathrm{y}|}-\frac{1}{|\partial\Omega|}\int_{\partial\Omega}\frac{1}{|\mathrm{x}-\mathrm{y}|}\Big) \partial_\nu u (\mathrm{y},\mathrm{t}-\mathrm{c}_0^{-1}|\mathrm{x}-\mathrm{y}|) d\sigma_\mathrm{y}\Bigg].
\end{align}
Given that $\mathrm{z} \in \Omega$ and $\mathrm{x}\in \mathbb{R}^3\setminus\overline{\Omega}$, we estimate using Taylor's series expansion
\begin{align}\label{t2}
 u_{\mathrm{t}\mathrm{t}}(\mathrm{y},\mathrm{t}-\mathrm{c}_0^{-1}|\mathrm{x}-\mathrm{y}|)- u_{\mathrm{t}\mathrm{t}}(\mathrm{y},\mathrm{t}-\mathrm{c}_0^{-1}|\mathrm{x}-\mathrm{z}|) =  \mathrm{c}_0^{-1} (\mathrm{y}-\mathrm{z})\nabla|\mathrm{x}-\mathrm{z}^*|\partial_\mathrm{t}^3 u(\mathrm{y},\mathrm{t}_0^*),
\end{align}
and
\begin{align}\label{tt2}
\frac{1}{|\mathrm{x}-\mathrm{y}|}-\frac{1}{|\mathrm{x}-\mathrm{z}|} = (\mathrm{y}-\mathrm{z})\nabla\frac{1}{|\mathrm{x}-\mathrm{z}^*|},
\end{align}
where, $ \mathrm{t}_0^*\in (\mathrm{t}-\mathrm{c}_0^{-1}|\mathrm{x}-\mathrm{y}|,\mathrm{t}-\mathrm{c}_0^{-1}|\mathrm{x}-\mathrm{z}|)$ and $\mathrm{z}^* \in \Omega.$ 
\newline

\noindent
Then, using the approximations (\ref{t2}) and (\ref{tt2}) in (\ref{f1}) we obtain
\begin{align}\label{3.11}
    u^\mathrm{s}(\mathrm{x},\mathrm{t})
    &= \nonumber -\frac{\rho_\mathrm{m}}{4\pi}\Bigg[\frac{\gamma}{|\mathrm{x}-\mathrm{z}|}\int_\Omega u_{\mathrm{t}\mathrm{t}}(\mathrm{y},\mathrm{t}-\mathrm{c}_0^{-1}|\mathrm{x}-\mathrm{z}|)d\mathrm{y} + \gamma\nabla\frac{1}{|\mathrm{x}-\mathrm{z}^*|} \int_\Omega u_{\mathrm{t}\mathrm{t}}(\mathrm{y},\mathrm{t}-\mathrm{c}_0^{-1}|\mathrm{x}-\mathrm{z}|)(\mathrm{y}-\mathrm{z}) d\mathrm{y}
    \\ \nonumber &- \frac{\gamma}{|\mathrm{x}-\mathrm{z}|} \int_\Omega\mathrm{c}_0^{-1} (\mathrm{y}-\mathrm{z})\nabla|\mathrm{x}-\mathrm{z}^*|\partial_\mathrm{t}^3 u(\mathrm{y},\mathrm{t}_0^*)d\mathrm{y} - \gamma\nabla\frac{1}{|\mathrm{x}-\mathrm{z}^*|} \int_\Omega\mathrm{c}_0^{-1} (\mathrm{y}-\mathrm{z})\nabla|\mathrm{x}-\mathrm{z}^*|\partial_\mathrm{t}^3 u(\mathrm{y},\mathrm{t}_0^*)d\mathrm{y}
    \\&+\nonumber \frac{\alpha}{|\partial\Omega|}\int_{\partial\Omega}\frac{1}{|\mathrm{x}-\mathrm{y}|}\int_{\partial\Omega}\partial_\nu u(\mathrm{y},\mathrm{t}-\mathrm{c}_0^{-1}|\mathrm{x}-\mathrm{y}|)d\sigma_\mathrm{y} 
    \\  &+ \alpha\int_{\partial\Omega}\Big(\frac{1}{|\mathrm{x}-\mathrm{y}|}-\frac{1}{|\partial\Omega|}\int_{\partial\Omega}\frac{1}{|\mathrm{x}-\mathrm{y}|}\Big) \partial_\nu u (\mathrm{y},\mathrm{t}-\mathrm{c}_0^{-1}|\mathrm{x}-\mathrm{y}|) d\sigma_\mathrm{y}\Bigg].
\end{align}
We now perform the subsequent computations. First,
\begin{align}\label{3.14}
    \Big|\frac{\gamma\rho_\mathrm{m}}{4\pi|\mathrm{x}-\mathrm{z}|}\int_\Omega u_{\mathrm{t}\mathrm{t}}(\mathrm{y},\mathrm{t}-\mathrm{c}_0^{-1}|\mathrm{x}-\mathrm{z}|)d\mathrm{y}\Big| \lesssim |\mathrm{x}-\mathrm{z}|^{-1}\delta^{\frac{3}{2}}\Vert \partial_\mathrm{t}^\mathrm{2}u(\cdot,\mathrm{t})\Vert_{\mathrm{L}^2(\Omega)}.
\end{align}
We then approximate the following term.
\begin{align}\label{3.12}
   \Big| \frac{\gamma\rho_\mathrm{m}}{4\pi}\nabla\frac{1}{|\mathrm{x}-\mathrm{z}^*|} \int_\Omega u_{\mathrm{t}\mathrm{t}}(\mathrm{y},\mathrm{t}-\mathrm{c}_0^{-1}|\mathrm{x}-\mathrm{z}|)(\mathrm{y}-\mathrm{z}) d\mathrm{y}\Big| &\lesssim \nonumber \frac{\gamma\rho_\mathrm{m}}{4\pi}\Big|\nabla\frac{1}{|\mathrm{x}-\mathrm{z}^*|}\Big|\Vert \cdot-\mathrm{z}\Vert_{\mathrm{L}^2(\Omega)} \Vert \partial_\mathrm{t}^2u(\cdot,\mathrm{t})\Vert_{\mathrm{L}^2(\Omega)}
    \\ &= \mathcal{O}\Big(\delta^\frac{5}{2} |\mathrm{x}-\mathrm{z}^*|^{-2}\Vert \partial_\mathrm{t}^2u(\cdot,\mathrm{t})\Vert_{\mathrm{L}^2(\Omega)}\Big).
\end{align}
Further, we have
\begin{align}\label{f3}
\nonumber
   &= \nonumber \Big| \mathrm{c}_0^{-1}\frac{\gamma\rho_\mathrm{m}}{4\pi|\mathrm{x}-\mathrm{z}|}\nabla|\mathrm{x}-\mathrm{z}^*|\int_\Omega (\mathrm{y}-\mathrm{z})\partial_\mathrm{t}^3(\mathrm{y},\mathrm{t}_0^*) d\mathrm{y}
    + \mathrm{c}_0^{-1}\frac{\gamma\rho_\mathrm{m}}{4\pi}\nabla|\mathrm{x}-\mathrm{z}^*|\nabla\frac{1}{|\mathrm{x}-\mathrm{z}^*|}\int_\Omega (\mathrm{y}-\mathrm{z})^2\partial_\mathrm{t}^3(\mathrm{y},\mathrm{t}_0^*) d\mathrm{y}\Big|
    \\ &\lesssim\nonumber\frac{\gamma\rho_\mathrm{m}}{4\pi}\Bigg[ \frac{1}{|\mathrm{x}-\mathrm{z}|}\Big|\nabla|\mathrm{x}-\mathrm{z}^*|\Big| \Vert \cdot-\mathrm{z}\Vert_{\mathrm{L}^2(\Omega)} \Vert \partial_\mathrm{t}^3u(\cdot,\mathrm{t}_0^*)\Vert_{\mathrm{L}^2(\Omega)} + \Big|\nabla|\mathrm{x}-\mathrm{z}^*|\Big|\Big|\nabla\frac{1}{|\mathrm{x}-\mathrm{z}^*|}\Big| \Vert (\cdot-\mathrm{z})^2\Vert_{\mathrm{L}^2(\Omega)} \Vert \partial_\mathrm{t}^3u(\cdot,\mathrm{t}_0^*)\Vert_{\mathrm{L}^2(\Omega)}\Bigg]
    \\ &= \mathcal{O}(\delta^\frac{5}{2} |\mathrm{x}-\mathrm{z}|^{-1}\Vert \partial_\mathrm{t}^3u(\cdot,\mathrm{t}_0^*)\Vert_{\mathrm{L}^2(\Omega)}).
\end{align}
Using (\ref{3.14}), (\ref{3.12}), and (\ref{f3}) in (\ref{3.11}), we obtain for $\mathrm{x}\in \mathbb{R}^{\textcolor{black}{3}}\setminus\overline{\Omega}$ such that $\textbf{dist}(\mathrm{x},\Omega)\sim \delta^\mathrm{q}$ \big(and then $|\mathrm{x}-\mathrm{z}| \sim \delta+\delta^\mathrm{q}$\big), where $\mathrm{q}\in [0,1],$ 
\begin{align} \label{finalestimate}
    u^\mathrm{s}(\mathrm{x},\mathrm{t})
    &= \nonumber \frac{\alpha\rho_\mathrm{m}}{4\pi}\frac{1}{|\partial\Omega|}\int_{\partial\Omega}\frac{1}{|\mathrm{x}-\mathrm{y}|}\int_{\partial\Omega}\partial_\nu u(\mathrm{y},\mathrm{t}-\mathrm{c}_0^{-1}|\mathrm{x}-\mathrm{y}|)d\sigma_\mathrm{y} 
    \\ \nonumber &- \frac{\alpha\rho_\mathrm{m}}{4\pi}\int_{\partial\Omega}\Big(\frac{1}{|\mathrm{x}-\mathrm{y}|}-\frac{1}{|\partial\Omega|}\int_{\partial\Omega}\frac{1}{|\mathrm{x}-\mathrm{y}|}\Big) \partial_\nu u (\mathrm{y},\mathrm{t}-\mathrm{c}_0^{-1}|\mathrm{x}-\mathrm{y}|) d\sigma_\mathrm{y}
    + \mathcal{O}\Big(\delta^\frac{5}{2} |\mathrm{x}-\mathrm{z}|^{-2}\Vert \partial_\mathrm{t}^2u(\cdot,\mathrm{t})\Vert_{\mathrm{L}^2(\Omega)}\Big)
    \\ &+ \mathcal{O}\Big(\delta^\frac{5}{2} |\mathrm{x}-\mathrm{z}|^{-1}\Vert \partial_\mathrm{t}^3u(\cdot,\mathrm{t})\Vert_{\mathrm{L}^2(\Omega)}\Big) + \mathcal{O}\Big(\delta^\frac{3}{2} |\mathrm{x}-\mathrm{z}|^{-1}\Vert \partial_\mathrm{t}^2u(\cdot,\mathrm{t})\Vert_{\mathrm{L}^2(\Omega)}\Big).
\end{align}
\bigbreak
\noindent
We state the following two Propositions that will be useful to estimate the reminder part of the previous expression.
\begin{proposition}\label{p1}
For $u = u^\mathrm{i}+u^\mathrm{s}$ as the solution of (\ref{mainfor}) we have the following estimates
\begin{align}
\Vert \partial_\mathrm{t}^\mathrm{k}u(\cdot,\mathrm{t})\Vert_{\mathrm{L}^2(\Omega)} \lesssim \delta^{\frac{3}{2}}, \quad \mathrm{t}\in [0,\mathrm{T}], \quad \mathrm{k}=0,1,\ldots.
\end{align}
\end{proposition}
\begin{proof}
See Section \ref{apri} for the proof.
\end{proof}
and
\begin{proposition}\label{p3}
We have the following estimate
\begin{align}
\Vert \partial_\mathrm{t}^\mathrm{k}\partial_\nu u(\cdot, \mathrm{t})\Vert_{\mathrm{H}^{-\frac{1}{2}}(\partial\Omega)} \sim \delta^2, \quad \mathrm{t}\in [0,\mathrm{T}], \quad \text{for}\quad \mathrm{k} =0,1,\ldots.
\end{align}
\end{proposition}
\begin{proof}
See Section \ref{apri} for the proof.
\end{proof}
\noindent
Next, we look at the Lippmann-Schwinger equation:
\begin{align}
    u(\mathrm{x},\mathrm{t}) + \gamma\bm{\mathrm{V}}_\Omega\Big[u_{\mathrm{t}\mathrm{t}}\Big](\mathrm{x},\mathrm{t})+ \alpha \bm{\mathcal{S}}_{\partial\Omega} \Big[\partial_\nu u\Big](\mathrm{x},\mathrm{t}) = u^\mathrm{i}\quad \text{in}\ \Omega.
\end{align}
The retarded single-layer potential operator $\bm{\mathcal{S}}_{\partial\Omega}$ is defined by
\begin{align}
    \bm{\mathcal{S}}_{\partial\Omega}\big[f\big](\mathrm{x},\mathrm{t}) := \int_{\partial \Omega} \frac{\rho_\mathrm{m}}{4\pi|\mathrm{x}-\mathrm{y}|}f(\mathrm{y},\mathrm{t}-\mathrm{c}_0^{-1}|\mathrm{x}-\mathrm{y}|)d\sigma_\mathrm{y},\quad \text{where} \quad (\mathrm{x},\mathrm{t}) \in \textcolor{black}{\Omega} \times(0,\mathrm{T}),
\end{align}
and we define the corresponding retarded volume potential $\bm{\mathrm{V}}_\Omega$ by
\begin{align}
    \bm{\mathrm{V}}_\Omega\big[f\big](\mathrm{x},\mathrm{t}) := \int_\Omega\frac{\rho_\mathrm{m} }{4\pi|\mathrm{x}-\mathrm{y}|}f(\mathrm{y},\mathrm{t}-\mathrm{c}_0^{-1}|\mathrm{x}-\mathrm{y}|)d\mathrm{y}, \quad \textcolor{black}{(\mathrm{x},\mathrm{t}) \in \Omega \times(0,\mathrm{T})}.
\end{align}
We have the jump relation $\partial_\nu^\pm\bm{\mathcal{S}}_{\partial\Omega}[f] = \mp\frac{\rho_\mathrm{m}}{2}+ \bm{\mathcal{K}}^\mathrm{t}[f].$ Consequently, taking the Neumann trace of the above equation we obtain 
\begin{align}\label{be1}
    \Big(1+\frac{\alpha\rho_\mathrm{m}}{2}\Big)\partial_\nu u + \gamma \partial_\nu\bm{\mathrm{V}}_\Omega\Big[u_{\mathrm{t}\mathrm{t}}\Big]+ \alpha \bm{\mathcal{K}}^\mathrm{t}\Big[\partial_\nu u\Big] = \partial_\nu u^\mathrm{i}\quad \textcolor{black}{\text{on}\; \partial\Omega},
\end{align}
where the adjoint double layer operator $\bm{\mathcal{K}}^\mathrm{t}$ is defined for $(\mathrm{x},\mathrm{t}) \in \partial\Omega \times(0,\mathrm{T})$ by
\begin{align}
    \bm{\mathcal{K}}^\mathrm{t}\Big[f\Big](\mathrm{x},\mathrm{t}) := -\rho_\mathrm{m}\Big[\mathrm{c}_0^{-1}\int_{\partial \Omega} \frac{f_\mathrm{t}(\mathrm{y}, \mathrm{t}- \mathrm{c}_0^{-1}|\mathrm{x}-\mathrm{y}|)}{4\pi|\mathrm{x}-\mathrm{y}|}\frac{(\mathrm{x}-\mathrm{y})\cdot\nu}{|\mathrm{x}-\mathrm{y}|}d\sigma_\mathrm{y} + \int_{\partial \Omega} \frac{f(\mathrm{y}, \mathrm{t}- \mathrm{c}_0^{-1}|\mathrm{x}-\mathrm{y}|)}{4\pi|\mathrm{x}-\mathrm{y}|}\frac{(\mathrm{x}-\mathrm{y})\cdot\nu}{|\mathrm{x}-\mathrm{y}|^2}d\sigma_\mathrm{y}\Big].
\end{align}
Then, from (\ref{be1}) and after taking the integration with respect to $\partial\Omega$ we obtain the following
\begin{align}\label{givefor}
(1+\frac{\alpha\rho_\mathrm{m}}{2})\int_{\partial\Omega} \partial_\nu u + \gamma \int_{\partial\Omega}\partial_\nu\bm{\mathrm{V}}_\Omega\Big[u_{\mathrm{t}\mathrm{t}}\Big]d\sigma_\mathrm{x} + \alpha \int_{\partial\Omega}  \bm{\mathcal{K}}^\mathrm{t}\Big[\partial_\nu u\Big] = \int_{\partial\Omega} \partial_\nu u^\mathrm{i}.
\end{align}
The term $\displaystyle\int_{\partial\Omega}  \bm{\mathcal{K}}^\mathrm{t}\Big[\partial_\nu u\Big]$ is expanded explicitly using Taylor's series expansion as follows:
\begin{align}
    \int_{\partial\Omega}  \bm{\mathcal{K}}^\mathrm{t}\Big[\partial_\nu u\Big] &=\nonumber\rho_\mathrm{m}\Bigg[ -\cancel{\mathrm{c}_0^{-1} \int_{\partial\Omega}(\partial_\nu u)_\mathrm{t}(\mathrm{y}, \mathrm{t}) \int_{\partial\Omega}\frac{(\mathrm{x}-\mathrm{y})\cdot\nu}{|\mathrm{x}-\mathrm{y}|^2}d\sigma_\mathrm{x}d\sigma_\mathrm{y}} + \mathrm{c}_0^{-2}\int_{\partial\Omega}(\partial_\nu u)_{\mathrm{t}\mathrm{t}}(\mathrm{y}, \mathrm{t}) \int_{\partial\Omega}\frac{(\mathrm{x}-\mathrm{y}).\nu}{|\mathrm{x}-\mathrm{y}|}d\sigma_\mathrm{x}d\sigma_\mathrm{y} 
    \\ \nonumber &- \mathrm{c}_0^{-3}\int_{\partial\Omega}(\partial_\nu u)_{\mathrm{t}\mathrm{t}\mathrm{t}}(\mathrm{y}, \mathrm{t}_1) \underbrace{\int_{\partial\Omega}(\mathrm{x}-\mathrm{y})\cdot\nu d\sigma_\mathrm{x}}_{\textcolor{black}{= 3|\Omega|}} d\sigma_\mathrm{y} - \int_{\partial\Omega}\partial_\nu u(\mathrm{y}, \mathrm{t}) \underbrace{\int_{\partial\Omega}\frac{(\mathrm{x}-\mathrm{y})\cdot\nu}{|\mathrm{x}-\mathrm{y}|^3}d\sigma_\mathrm{x}}_{\textcolor{black}{=\frac{1}{2}}}d\sigma_\mathrm{y}
    \\ \nonumber &+ \cancel{\mathrm{c}_0^{-1} \int_{\partial\Omega}(\partial_\nu u)_\mathrm{t}(\mathrm{y}, \mathrm{t}) \int_{\partial\Omega}\frac{(\mathrm{x}-\mathrm{y})\cdot\nu}{|\mathrm{x}-\mathrm{y}|^2}d\sigma_\mathrm{x}d\sigma_\mathrm{y}} - \frac{1}{2} \mathrm{c}_0^{-2}\int_{\partial\Omega}(\partial_\nu u)_{\mathrm{t}\mathrm{t}}(\mathrm{y}, \mathrm{t}) \int_{\partial\Omega}\frac{(\mathrm{x}-\mathrm{y})\cdot\nu}{|\mathrm{x}-\mathrm{y}|}d\sigma_\mathrm{x}d\sigma_\mathrm{y} 
    \\ &+ \frac{1}{3!}\mathrm{c}_0^{-3}\int_{\partial\Omega}(\partial_\nu u)_{\mathrm{t}\mathrm{t}\mathrm{t}}(\mathrm{y}, \mathrm{t}_2) \int_{\partial\Omega}(\mathrm{x}-\mathrm{y})\cdot\nu d\sigma_\mathrm{x}d\sigma_\mathrm{y}\Bigg].
\end{align}
where $\mathrm{t}_1, \mathrm{t}_2 \in \big(\mathrm{t}-\mathrm{c}^{-1}_0|\mathrm{x}-\mathrm{y}|,\mathrm{t}\big)$ and we denote $\mathrm{t}_\mathrm{m} := \max(\mathrm{t}_1, \mathrm{t}_2)$.
Hence, we obtain
\begin{align}\label{tdouble}
\int_{\partial\Omega}  \bm{\mathcal{K}}^\mathrm{t}\Big[\partial_\nu u\Big] =  \frac{\rho_\mathrm{m}}{2} \mathrm{c}_0^{-2}\int_{\partial\Omega}\partial_\mathrm{t}^2\partial_\nu u(\mathrm{y}, \mathrm{t}) \int_{\partial\Omega}\frac{(\mathrm{x}-\mathrm{y})\cdot\nu}{|\mathrm{x}-\mathrm{y}|}d\sigma_\mathrm{x}d\sigma_\mathrm{y} - \frac{\rho_\mathrm{m}}{2} \int_{\partial\Omega}\partial_\nu u + \mathcal{O}\big(\alpha\delta^4 \Vert \partial_\mathrm{t}^3\partial_\nu u(\cdot,\mathrm{t}_\mathrm{m})\Vert_{\mathrm{H}^{-\frac{1}{2}}(\partial\Omega)}\big).
\end{align}
We rewrite the first term of (\ref{tdouble}) as 
\begin{align}
    \int_{\partial\Omega}\partial_\mathrm{t}^2\partial_\nu u(\mathrm{y}, \mathrm{t}) \int_{\partial\Omega}\frac{(\mathrm{x}-\mathrm{y})\cdot\nu}{|\mathrm{x}-\mathrm{y}|}d\sigma_\mathrm{x}d\sigma_\mathrm{y} &= \nonumber \mathrm{A}_{\partial\Omega}\int_{\partial\Omega}\partial_\mathrm{t}^2\partial_\nu u(\mathrm{y}, \mathrm{t}) 
    + \int_{\partial\Omega}\partial_\mathrm{t}^2\partial_\nu u(\mathrm{y}, \mathrm{t})\Big[\int_{\partial\Omega}\frac{(\mathrm{x}-\mathrm{y})\cdot\nu}{|\mathrm{x}-\mathrm{y}|}d\sigma_\mathrm{x} - \mathrm{A}_{\partial\Omega}\Big], 
\end{align}
where
\begin{align}\label{av}
    \mathrm{A}_{\partial\Omega} = \frac{1}{|\partial \Omega|}\int_{\partial\Omega}\int_{\partial\Omega}\frac{(\mathrm{x}-\mathrm{y})\cdot\nu}{|\mathrm{x}-\mathrm{y}|}d\sigma_\mathrm{x}d\sigma_\mathrm{y} \sim \delta^2.
\end{align}
We then use (\ref{tdouble}) and (\ref{av}) to rewrite the equation (\ref{givefor}) as
\begin{align}\label{doubleest}
\int_{\partial\Omega} \partial_\nu u &+ \nonumber \frac{\alpha\rho_\mathrm{m}}{2}     \mathrm{A}_{\partial\Omega} \mathrm{c}_0^{-2}\int_{\partial\Omega}\partial_\mathrm{t}^2\partial_\nu u +  \gamma \int_{\partial\Omega}\partial_\nu\bm{\mathrm{V}}_\Omega\Big[u_{\mathrm{t}\mathrm{t}}\Big] d\sigma_\mathrm{x} =   \int_{\partial\Omega} \partial_\nu u^\mathrm{i} \\ &- \frac{\alpha\rho_\mathrm{m}}{2}\mathrm{c}_0^{-2}\int_{\partial\Omega}\partial_\mathrm{t}^2\partial_\nu u(\mathrm{y}, \mathrm{t}_\text{m})\Big[\int_{\partial\Omega}\frac{(\mathrm{x}-\mathrm{y})\cdot\nu}{|\mathrm{x}-\mathrm{y}|}d\sigma_\mathrm{x} - \mathrm{A}_{\partial\Omega}\Big]
+ \mathcal{O}\Big(\alpha\delta^4 \Vert \partial_\mathrm{t}^3\partial_\nu u(\cdot,\mathrm{t}_\mathrm{m})\Vert_{\mathrm{H}^{-\frac{1}{2}}(\partial\Omega)}\Big).
\end{align}
We set
\begin{align}
\textbf{err.}^{(1)} = \nonumber\frac{\alpha\rho_\mathrm{m}}{2}\mathrm{c}_0^{-2}\int_{\partial\Omega}\partial_\mathrm{t}^2\partial_\nu u(\mathrm{y}, \mathrm{t}_\text{m})\Big[\int_{\partial\Omega}\frac{(\mathrm{x}-\mathrm{y})\cdot\nu}{|\mathrm{x}-\mathrm{y}|}d\sigma_\mathrm{x} - \mathrm{A}_{\partial\Omega}\Big].
\end{align}
Again using Taylor's series expansion, we rewrite the expression $\partial_\nu\bm{\mathrm{V}}_\Omega\Big[u_{\mathrm{t}\mathrm{t}}\Big]$ as follows:
\begin{align}\label{tv}
\partial_\nu\bm{\mathrm{V}}_\Omega\Big[u_{\mathrm{t}\mathrm{t}}\Big] &:= \nonumber \partial_\nu\int_\Omega\frac{\rho_\mathrm{m}}{4\pi|\mathrm{x}-\mathrm{y}|}\partial_\mathrm{t}^2u(\mathrm{y},\mathrm{t}-\mathrm{c}^{-1}_0|\mathrm{x}-\mathrm{y}|)d\mathrm{y}
\\ &= \rho_\mathrm{m}\Bigg[\partial_\nu\int_\Omega\frac{1}{4\pi|\mathrm{x}-\mathrm{y}|}\partial_\mathrm{t}^2u(\mathrm{y},\mathrm{t})d\mathrm{y} - \underbrace{\partial_\nu\int_\Omega\frac{\partial_\mathrm{t}^{\textcolor{black}{3}}u(\mathrm{y},\mathrm{t})}{4\pi|\mathrm{x}-\mathrm{y}|}\mathrm{c}^{-1}_0|\mathrm{x}-\mathrm{y}|d\mathrm{y}}_{=0} + \partial_\nu\int_\Omega\frac{\partial_\mathrm{t}^{\textcolor{black}{4}}u(\mathrm{y},\mathrm{t}_3)}{4\pi|\mathrm{x}-\mathrm{y}|}\mathrm{c}^{-2}_0|\mathrm{x}-\mathrm{y}|^2d\mathrm{y}\Bigg],
\end{align}
where $\mathrm{t}_3 \in \big(\mathrm{t}-\mathrm{c}^{-1}_0|\mathrm{x}-\mathrm{y}|,\mathrm{t}\big)$.
\newline

\noindent
Using equation (\ref{tv}) and applying Green's identities to the equation (\ref{doubleest}), we get
\begin{align}
\int_{\partial\Omega} \partial_\nu u + \frac{\alpha\rho_\mathrm{m}}{2}     \mathrm{A}_{\partial\Omega} \mathrm{c}_0^{-2}\int_{\partial\Omega}(\partial_\nu u)_{\mathrm{t}\mathrm{t}} + \gamma\rho_\mathrm{m} \int_{\Omega}\Delta \bm{\mathcal{N}}_{\textbf{Lap},\Omega}\Big[\partial_\mathrm{t}^2u\Big] d\mathrm{x} &= \nonumber \nonumber \int_{\partial\Omega} \partial_\nu u^\mathrm{i} + \mathcal{O}\Big(\alpha\delta^4 \Vert \partial_\mathrm{t}^3\partial_\nu u(\cdot,\mathrm{t}_\mathrm{m})\Vert_{\mathrm{H}^{-\frac{1}{2}}(\partial\Omega)}\Big) \\&+ \textbf{err.}^{(1)} + \textbf{err.}^{(2)},
\end{align} 
where, $\displaystyle\bm{\mathcal{N}}_{\textbf{Lap},\Omega}\Big[\partial_\mathrm{t}^2u\Big]:=\int_\Omega\frac{1}{4\pi|\mathrm{x}-\mathrm{y}|}\partial_\mathrm{t}^2u(\mathrm{y},\mathrm{t})d\mathrm{y}$ is the Newtonian potential that corresponds to the Laplace equation and
\begin{align}
    \textbf{err.}^{(2)} := \mathrm{c}^{-2}_0 \gamma\rho_\mathrm{m}\int_{\partial\Omega} \partial_\nu \int_\Omega |\mathrm{x}-\mathrm{y}| \partial_\mathrm{t}^{\textcolor{black}{4}}u(\mathrm{y},\mathrm{t}_3) d\mathrm{y}d\sigma_{\mathrm{x}}.
\end{align}
As we have $\Delta \bm{\mathcal{N}}_{\textbf{Lap},\Omega}\mathrm{f}=-\mathrm{f}$, then
\begin{align}\label{tdouble1}
\int_{\partial\Omega} \partial_\nu u + \frac{\alpha\rho_\mathrm{m}}{2}     \mathrm{A}_{\partial\Omega} \mathrm{c}_0^{-2}\int_{\partial\Omega}(\partial_\nu u)_{\mathrm{t}\mathrm{t}} - \gamma\rho_\mathrm{m}\int_\Omega\partial_\mathrm{t}^2u &= \nonumber  \int_{\partial\Omega} \partial_\nu u^\mathrm{i} + \mathcal{O}\Big(\alpha\delta^4 \Vert \partial_\mathrm{t}^3\partial_\nu u(\cdot,\mathrm{t}_\mathrm{m})\Vert_{\mathrm{H}^{-\frac{1}{2}}(\partial\Omega)}\Big) \\&+ \textbf{err.}^{(1)} + \textbf{err.}^{(2)}.  
\end{align}    
Before moving on, we note that we will be using the estimate of Proposition \ref{p3} i.e. 
\begin{align}
\Vert \partial_\mathrm{t}^\mathrm{k}\partial_\nu u(\cdot, \mathrm{t})\Vert_{\mathrm{H}^{-\frac{1}{2}}(\partial\Omega)} \sim \delta^2 \quad \text{for}\quad \mathrm{k} =0,1,...
\end{align}
We denote $\displaystyle \mathrm{A}(\mathrm{y})= \int_{\partial\Omega}\frac{(\mathrm{x}-\mathrm{y})\cdot\nu}{|\mathrm{x}-\mathrm{y}|}d\sigma_\mathrm{x}.$ Then, due to the equation's time regularity, we can estimate $\textbf{err.}^{(1)}$ by utilizing the same equation as developed in Section \ref{apri}, in equation (\ref{important}) 
\begin{align}
   \nonumber\Big((\frac{1}{\alpha}+\frac{{\rho_\mathrm{m}}}{2})+ \bm{\mathcal{K}}^*_\textbf{Lap}\Big)\Big[\partial_\mathrm{t}^2\partial_\nu u\Big]
    &= \frac{1}{\alpha}\partial_\nu u^\mathrm{i} 
    - \rho_\mathrm{m}\Bigg[\frac{1}{2} \mathrm{c}^{-2}_0\int_{\partial\Omega}\partial_\mathrm{t}^4\partial_\nu u(\mathrm{y}, \mathrm{t}) \frac{(\mathrm{x}-\mathrm{y})\cdot\nu_\mathrm{x}}{|\mathrm{x}-\mathrm{y}|}d\sigma_\mathrm{y}
    \\ \nonumber&+ \frac{5}{6}\mathrm{c}^{-3}_0\int_{\partial\Omega}\partial_\mathrm{t}^5\partial_\nu u(\mathrm{y}, \mathrm{t}_\mathrm{m} ) (\mathrm{x}-\mathrm{y})\cdot\nu_\mathrm{x} d\sigma_\mathrm{y}
    \\ &+ \frac{\gamma}{\alpha}\partial_\nu\int_\Omega\frac{1}{4\pi|\mathrm{x}-\mathrm{y}|}\partial_\mathrm{t}^4u(\mathrm{y},\mathrm{t})d\mathrm{y} + \mathrm{c}^{-2}_0\frac{\gamma}{\alpha} \partial_\nu\int_\Omega|\mathrm{x}-\mathrm{y}|\partial_\mathrm{t}^6u(\mathrm{y},\mathrm{t}_3)d\mathrm{y}\Bigg].
\end{align}
We \textcolor{black}{first} set
\begin{align}
    &\nonumber\mathrm{a}_1:= \frac{1}{\alpha}\partial_\mathrm{t}^2\partial_\nu u^\mathrm{i},\quad \mathrm{a}_2 := \int_{\partial\Omega}\partial_\mathrm{t}^4\partial_\nu u(\mathrm{y},\mathrm{t}) \frac{(\mathrm{x}-\mathrm{y})\cdot\nu_\mathrm{x}}{|\mathrm{x}-\mathrm{y}|}d\sigma_\mathrm{y}, \quad \mathrm{a}_3:= \displaystyle\frac{5}{6}\mathrm{c}^{-3}_0 \int_{\partial\Omega}(\mathrm{x}-\mathrm{y})\cdot\nu_\mathrm{x}\partial_\mathrm{t}^5\partial_\nu u(\mathrm{y},\mathrm{t}_\mathrm{m}) d\sigma_\mathrm{y}
    \\ \nonumber &\mathrm{a}_4 := \frac{\gamma}{\alpha} \partial_\nu\int_\Omega\frac{1}{4\pi|\mathrm{x}-\mathrm{y}|}\partial_\mathrm{t}^4u(\mathrm{y},\mathrm{t})d\mathrm{y} \quad \text{and} \quad \mathrm{a}_5:= \mathrm{c}^{-2}_0\frac{\gamma}{\alpha} \partial_\nu\int_\Omega|\mathrm{x}-\mathrm{y}|\partial_\mathrm{t}^6u(\mathrm{y},\mathrm{t}_3)d\mathrm{y}.
\end{align}
\textcolor{black}{Here, $\bm{\mathcal{K}}_\textbf{Lap}$ denotes the double layer operator that corresponds to the Laplace equation, defined as
$\displaystyle\bm{\bm{\mathcal{K}}}_\textbf{Lap}\big[f\big](\mathrm{x}):= \rho_\mathrm{m}\int_{\partial\Omega} \frac{(\mathrm{x}-\mathrm{y})\cdot\nu_\mathrm{y}}{|\mathrm{x}-\mathrm{y}|^3}f(\mathrm{y})d\sigma_\mathrm{y}$ and $\bm{\bm{\mathcal{K}}}^*_\textbf{Lap}$ is its adjoint.} 
\newline

\noindent
We then arrive at the following estimate:
\begin{align}\label{er1}
\textbf{err.}^{(1)} &:=\Big|\nonumber\frac{\alpha\rho_\mathrm{m}}{2}\mathrm{c}_0^{-2}\int_{\partial\Omega}\partial_\mathrm{t}^2\partial_\nu u(\mathrm{y}, \mathrm{t}_\text{m})\Big[\int_{\partial\Omega}\frac{(\mathrm{x}-\mathrm{y})\cdot\nu}{|\mathrm{x}-\mathrm{y}|}d\sigma_\mathrm{x} - \mathrm{A}_{\partial\Omega}\Big]\Big|
\\ \nonumber &\lesssim \alpha \int_{\partial\Omega}(\frac{1}{\alpha}+\frac{\rho_\mathrm{m}}{2} +\bm{\bm{\mathcal{K}}}_\textbf{Lap} )^{-1}\Big)[\mathrm{A}(\mathrm{y})-\mathrm{A}_{\partial\Omega}](\mathrm{a}_1 + \mathrm{a}_2 + \mathrm{a}_3 + \mathrm{a}_4 + \mathrm{a}_5)d\sigma_\mathrm{y}
\\ \nonumber &\lesssim \alpha \Big\Vert (\frac{1}{\alpha}+\frac{\rho_\mathrm{m}}{2} +\bm{\mathcal{K}}_\textbf{Lap} )^{-1}[\mathrm{A}(\cdot)-\mathrm{A}_{\partial\Omega}] \Big\Vert_{\mathrm{H}^{\frac{1}{2}}_0(\partial\Omega)} \Vert \mathrm{a}_1 + \mathrm{a}_2 + \mathrm{a}_3 + \mathrm{a}_4 + \mathrm{a}_5 \Vert_{\mathrm{H}^{-\frac{1}{2}}_0(\partial\Omega)}
\\ &= \alpha \delta^3 \Vert \mathrm{a}_1 + \mathrm{a}_2 + \mathrm{a}_3 + \mathrm{a}_4 + \mathrm{a}_5 \Vert_{\mathrm{H}^{-\frac{1}{2}}_0(\partial\Omega)}.
\end{align}
\begin{proposition}\label{propo}
We have the estimate 
\begin{align}
    \Vert \mathrm{a}_1 + \mathrm{a}_2 + \mathrm{a}_3 + \mathrm{a}_4 + \mathrm{a}_5 \Vert_{\mathrm{H}^{-\frac{1}{2}}_0(\partial\Omega)} = \mathcal{O}(\delta^3).
\end{align}
\end{proposition}
\begin{proof}
See Section \ref{appen} for the proof.
\end{proof}
\noindent
Consequently, we obtain 
\begin{align}\label{er11}
    \textbf{err.}^{(1)} = \mathcal{O}(\delta^4).
\end{align}
Using Divergence Theorem, we observe that the integral $\displaystyle\int_{\partial\Omega}\partial_\nu\int_\Omega|\mathrm{x}-\mathrm{y}|\partial_\mathrm{t}^{\textcolor{black}{4}}u(\mathrm{y},\mathrm{t}_3)d\mathrm{y}d\sigma_\mathrm{x}$ behave as Newtonian Potential and then using Cauchy-Schwartz's inequality, we obtain
\begin{align}\label{er2}
    \textbf{err.}^{(2)}:= \Big|\int_{\partial\Omega}\partial_\nu\int_\Omega|\mathrm{x}-\mathrm{y}|\partial_\mathrm{t}^{\textcolor{black}{4}}u(\mathrm{y},\mathrm{t}_3)d\mathrm{y}d\sigma_\mathrm{x}\Big| \sim \delta^5.
\end{align}
Consequently, we use the equation $\frac{\rho_\mathrm{c}}{\mathrm{k}_\mathrm{c}}u_{\mathrm{t}\mathrm{t}}-\Delta u = 0$, estimates (\ref{er11}), (\ref{er2}), and Proposition \ref{p3} to obtain from equation (\ref{tdouble1})
\begin{align}
\int_{\partial\Omega} \partial_\nu u + \frac{\alpha\rho_\mathrm{m}}{2}     \mathrm{A}_{\partial\Omega} \mathrm{c}_0^{-2}\int_{\partial\Omega}(\partial_\nu u)_{\mathrm{t}\mathrm{t}} - \gamma\frac{\mathrm{k}_\mathrm{c}\rho_\mathrm{m}}{\rho_\mathrm{c}} \int_\Omega \Delta u&= \nonumber \int_{\partial\Omega} \partial_\nu u^\mathrm{i} + \mathcal{O}(\delta^4).
\end{align}     
Therefore, using integration by parts, we draw the following conclusion
\begin{align}\label{number}
\Big(1 - \gamma\frac{\mathrm{k}_\mathrm{c}\rho_\mathrm{m}}{\rho_\mathrm{c}}\Big)\int_{\partial\Omega} \partial_\nu u + \frac{\alpha\rho_\mathrm{m}}{2} \mathrm{A}_{\partial\Omega} \mathrm{c}_0^{-2}\int_{\partial\Omega}(\partial_\nu u)_{\mathrm{t}\mathrm{t}} =\int_{\partial\Omega} \partial_\nu u^\mathrm{i} + \mathcal{O}(\delta^4).
\end{align}
\textcolor{black}{Recall that the coefficients $\rho$ and $\mathrm{k}$ are piece-wise constants, with one constant outside $\Omega$, i.e. $\rho(\mathrm{x})\equiv \rho_\mathrm{m},\; \mathrm{k}(\mathrm{x}) \equiv \mathrm{k}_\mathrm{m}$ and other constants $\rho_\mathrm{c}$ and $\mathrm{k}_\mathrm{c}$ in $\Omega$ satisfying the following scaling properties
\begin{align}\label{scaling}
    \rho_\mathrm{c} = \overline{\rho_\mathrm{c}}\delta^2, \quad \mathrm{k}_\mathrm{c} = \overline{\mathrm{k}_\mathrm{c}}\delta^2 \quad \text{and} \quad \frac{\rho_\mathrm{c}}{\mathrm{k}_\mathrm{c}} \sim 1 \ \text{as}\ \delta\ll1.
\end{align}
Therefore, after a short hand-calculation, we get $1 - \gamma\frac{\mathrm{k}_\mathrm{c}\rho_\mathrm{m}}{\rho_\mathrm{c}} = \frac{\mathrm{k}_\mathrm{c}}{\rho_\mathrm{c}}\frac{1}{\mathrm{c}_0^2}$, where $\gamma:=\beta-\alpha \frac{\rho_\mathrm{c}}{\mathrm{k}_\mathrm{c}}$ with $\alpha:=\frac{1}{\rho_\mathrm{c}}-\frac{1}{\rho_\mathrm{m}}$ and $\beta:= \frac{1}{\mathrm{k}_\mathrm{c}}-\frac{1}{\mathrm{k}_\mathrm{m}}$ .}
\newline

\noindent
Then, we denote by $\displaystyle\mathrm{Y}(\mathrm{t}):= \int_{\partial\Omega} \partial_\nu u$ and rewrite (\ref{number}) as follows:
\begin{align}
\begin{cases}
   \textcolor{black}{ \frac{\alpha\rho_\mathrm{m}}{2}\frac{\rho_\mathrm{c}}{\mathrm{k}_\mathrm{c}}\mathrm{A}_{\partial\Omega}} \frac{\mathrm{d}^2}{\mathrm{d}\mathrm{t}^2} \mathrm{Y}(\mathrm{t}) + \mathrm{Y}(\mathrm{t}) &= \displaystyle\textcolor{black}{\frac{\rho_\mathrm{c}}{\mathrm{k}_\mathrm{c}}\mathrm{c}_0^2}\int_{\partial\Omega} \partial_\nu u^\mathrm{i} + \mathcal{O}(\delta^4), \quad \text{in} \quad (0,\mathrm{T}),
    \\ \mathrm{Y}(\mathrm{t}) = \frac{\mathrm{d}}{\mathrm{d}\mathrm{t}}\mathrm{Y}(\mathrm{t}) = 0.
\end{cases}
\end{align}
Therefore, the solution of the problem is given by
\begin{align}\label{lamda}
\mathrm{Y}(\mathrm{t}) = \mathrm{p}^{-\frac{1}{2}}\int_0^\mathrm{t}\sin \Big( \mathrm{p}^{-\frac{1}{2}}(\mathrm{t}-\mathrm{\tau})\Big) \mathrm{g}(\tau) d\tau,
\end{align}
where $\mathrm{p} := \textcolor{black}{ \frac{\alpha\rho_\mathrm{m}}{2}\frac{\rho_\mathrm{c}}{\mathrm{k}_\mathrm{c}}\mathrm{A}_{\partial\Omega}}$ and $\mathrm{g}(\mathrm{t}) := \textcolor{black}{\frac{\rho_\mathrm{c}}{\mathrm{k}_\mathrm{c}}\mathrm{c}_0^2}\displaystyle\int_{\partial\Omega} \partial_\nu u^\mathrm{i} + \mathcal{O}(\delta^4).$
\newline

\noindent
Next, we estimate the second term of (\ref{finalestimate}), which is represented by the following expression
\begin{align}
    \displaystyle\frac{\alpha\rho_\mathrm{m}}{4\pi}\int_{\partial\Omega}\Big(\frac{1}{|\mathrm{x}-\mathrm{y}|}-\frac{1}{|\partial\Omega|}\int_{\partial\Omega}\frac{1}{|\mathrm{x}-\mathrm{y}|}\Big) \partial_\nu u (\mathrm{y},\mathrm{t}-\mathrm{c}_0^{-1}|\mathrm{x}-\mathrm{y}|) d\sigma_\mathrm{y}.
\end{align}
We first denote by
\begin{align}
    \mathrm{f}_0 := \frac{1}{|\mathrm{x}-\mathrm{y}|}-\frac{1}{|\partial\Omega|}\int_{\partial\Omega}\frac{1}{|\mathrm{x}-\mathrm{y}|}d\sigma_\mathrm{y}, \quad \text{where}\ \mathrm{x}\in \mathbb{R}^3\setminus\overline{\Omega}.
\end{align}
Then, we take into consideration the equation that is developed in Section \ref{apri}, in equation (\ref{important}) 
\begin{align}\label{Taylor}
   &\nonumber(\frac{1}{\alpha}+\frac{\rho_\mathrm{m}}{2})\partial_\nu u(\mathrm{v},\mathrm{t}) - \rho_\mathrm{m}\int_{\partial\Omega}\partial_\nu u(\mathrm{y}, \mathrm{t}) \frac{(\mathrm{v}-\mathrm{y})\cdot\nu_\mathrm{v}}{|\mathrm{v}-\mathrm{y}|^3}d\sigma_\mathrm{y} 
    = \frac{1}{\alpha}\partial_\nu u^\mathrm{i} 
    - \rho_\mathrm{m}\Bigg[\frac{1}{2} \mathrm{c}^{-2}_0\int_{\partial\Omega}\partial_\mathrm{t}^2\partial_\nu u(\mathrm{y}, \mathrm{t}) \frac{(\mathrm{v}-\mathrm{y})\cdot\nu_\mathrm{v}}{|\mathrm{v}-\mathrm{y}|}d\sigma_\mathrm{y}
    \\ &+ \frac{5}{6}\mathrm{c}^{-3}_0\int_{\partial\Omega}\partial_\mathrm{t}^3\partial_\nu u(\mathrm{y}, \mathrm{t}_\mathrm{m} ) (\mathrm{v}-\mathrm{y})\cdot\nu_\mathrm{v} d\sigma_\mathrm{y}
   + \frac{\gamma}{\alpha}\partial_\nu\int_\Omega\frac{\partial_\mathrm{t}^2u(\mathrm{y},\mathrm{t})}{4\pi|\mathrm{v}-\mathrm{y}|}d\mathrm{y} + \mathrm{c}^{-2}_0\frac{\gamma}{\alpha} \partial_\nu\int_\Omega|\mathrm{v}-\mathrm{y}|\partial_\mathrm{t}^4u(\mathrm{y},\mathrm{t}_3)d\mathrm{y}\Bigg],
\end{align}
multiply it by $\mathrm{f}_0$ and integrate over $\partial\Omega$ to obtain
\begin{align}
    &\int_{\partial\Omega} \mathrm{f}_0 \partial_\nu  u(\mathrm{v},\mathrm{t})d\sigma_\mathrm{v} = \nonumber\frac{1}{\alpha} \int_{\partial\Omega} \mathrm{f}_0\bm{\mathrm{T}}^* \partial_\nu u^\mathrm{i}(\mathrm{v},\mathrm{t})d\sigma_\mathrm{v} - \frac{\rho_\mathrm{m}}{2} \mathrm{c}^{-2}_0 \int_{\partial\Omega} \mathrm{f}_0\bm{\mathrm{T}}^*\int_{\partial\Omega}\partial_\mathrm{t}^2\partial_\nu u(\mathrm{y}, \mathrm{t}) \frac{(\mathrm{v}-\mathrm{y})\cdot\nu_\mathrm{v}}{|\mathrm{v}-\mathrm{y}|}d\sigma_\mathrm{y}d\sigma_\mathrm{v}
    \\ \nonumber &- \frac{5\rho_\mathrm{m}}{6}\mathrm{c}^{-3}_0 \int_{\partial\Omega} \mathrm{f}_0\bm{\mathrm{T}}^*\int_{\partial\Omega}\partial_\mathrm{t}^3\partial_\nu u(\mathrm{y}, \mathrm{t}_\mathrm{m} ) (\mathrm{v}-\mathrm{y})\cdot\nu_\mathrm{v} d\sigma_\mathrm{y}d\sigma_\mathrm{v}
    \\ &- \frac{\gamma\rho_\mathrm{m}}{\alpha} \Big[\int_{\partial\Omega} \mathrm{f}_0\bm{\mathrm{T}}^* \partial_\nu\int_\Omega\frac{1}{4\pi|\mathrm{v}-\mathrm{y}|}\partial_\mathrm{t}^2u(\mathrm{y},\mathrm{t})d\mathrm{y}d\sigma_\mathrm{v}
    - \mathrm{c}^{-2}_0\int_{\partial\Omega} \mathrm{f}_0\bm{\mathrm{T}}^* \partial_\nu\int_\Omega|\mathrm{v}-\mathrm{y}|\partial_\mathrm{t}^4u(\mathrm{y},\mathrm{t}_3)d\mathrm{y}d\sigma_\mathrm{v}\Big],
\end{align}
where we set \textcolor{black}{$\bm{\mathrm{T}} := \Big[\big(\frac{1}{\alpha}+\frac{\rho_\mathrm{m}}{2}\big) +\bm{\mathcal{K}}_\textbf{Lap} \Big]^{-1}$ and } $\bm{\mathrm{T}}^*  := \Big[\big(\frac{1}{\alpha}+\frac{\rho_\mathrm{m}}{2}\big) +\bm{\mathcal{K}}^*_\textbf{Lap} \Big]^{-1}.$ 
\newline

\noindent
As $\mathrm{x}\in \mathbb{R}^3\setminus\overline{\Omega}$ and $\mathrm{y}\in \partial\Omega$, we do the following estimate with Taylor's series expansion for $\mathrm{z}\in \Omega$ with $|\mathrm{x}-\mathrm{z}|\sim \delta + \delta^\mathrm{q}$ and by triangle inequality we obtain
\begin{align}
    \Vert \mathrm{f}_0 \Vert_{\mathrm{L}_0^2(\partial\Omega)}^2 &\lesssim\nonumber \Big\Vert \frac{1}{|\mathrm{x}-\cdot|}\Big\Vert^2_{\mathrm{L}_0^2(\partial\Omega)} 
    \\ &=\nonumber\int_{\partial\Omega} \frac{1}{|\mathrm{x}-\mathrm{y}|^2}d\sigma_\mathrm{y}
    \\ &= \int_{\partial\Omega} \frac{1}{|\mathrm{x}-\mathrm{z}|^2}d\sigma_\mathrm{y} + \mathcal{O}\Big(\int_{\partial\Omega} \frac{1}{|\mathrm{x}-\mathrm{z}|^3}|\mathrm{y}-\mathrm{z}|d\sigma_\mathrm{y}\Big) \sim \delta^{2-2\mathrm{q}}.
\end{align}
Then using $\Vert \mathrm{f}_0 \Vert_{\mathrm{L}_0^2(\partial\Omega)} \sim \delta^{1-\mathrm{q}}$ and $\Vert \bm{\mathrm{T}} \Vert_{\mathcal{L}(\mathrm{L}_0^2(\partial\Omega))} \sim 1$, see \cite{ammari}, we do the following estimates to arrive at
\begin{align}\label{1}
 &\nonumber\Big|\frac{1}{\alpha}  \int_{\partial\Omega} \mathrm{f}_0\bm{\mathrm{T}}^*\big( \partial_\nu u^\mathrm{i}(\mathrm{v})\big)d\sigma_\mathrm{v}\Big| 
 \\ &\lesssim\Big|\frac{1}{\alpha}  \int_{\partial\Omega} \bm{\mathrm{T}}\big(\mathrm{f}_0\big) \partial_\nu u^\mathrm{i}(\mathrm{v})d\sigma_\mathrm{v}\Big| \lesssim \frac{1}{\alpha}\Vert \mathrm{f}_0 \Vert_{\mathrm{L}_0^2(\partial\Omega)} \Vert \bm{\mathrm{T}} \Vert_{\mathcal{L}(\mathrm{L}_0^2(\partial\Omega))} \Vert \partial_\nu u^\mathrm{i} \Vert_{\mathrm{L}^2_0(\partial\Omega)} \sim \delta^{4-\mathrm{q}},
\end{align}
and
\begin{align}\label{2}
 &\nonumber\Big|\frac{\rho_\mathrm{m}}{2} \mathrm{c}^{-2}_0 \int_{\partial\Omega} \mathrm{f}_0\bm{\mathrm{T}}^*\Big(\int_{\partial\Omega}\partial_\mathrm{t}^2\partial_\nu u(\mathrm{y}, \mathrm{t}) \frac{(\mathrm{v}-\mathrm{y})\cdot\nu_\mathrm{v}}{|\mathrm{v}-\mathrm{y}|}d\sigma_\mathrm{y}\Big)d\sigma_\mathrm{v}\Big| \\&\lesssim |\partial\Omega|^\frac{1}{2}\Vert \mathrm{f}_0 \Vert_{\mathrm{L}^2_0(\partial\Omega} \Vert \bm{\mathrm{T}} \Vert_{\mathcal{L}(\mathrm{L}^2_0(\partial\Omega))}\underbrace{\Big|\int_{\partial\Omega}\partial_\mathrm{t}^2\partial_\nu u(\mathrm{y}, \mathrm{t}) \frac{(\mathrm{v}-\mathrm{y})\cdot\nu_\mathrm{v}}{|\mathrm{v}-\mathrm{y}|}d\sigma_\mathrm{y}\Big|}_{\sim \delta^{\frac{5}{2}}\quad \text{by}\ (\ref{est7})}
 \sim \delta^{\frac{9}{2}-\mathrm{q}}.
\end{align}
Similarly, we get
\begin{align}\label{3}
&\nonumber\Big|\frac{5\rho_\mathrm{m}}{6}\mathrm{c}^{-3}_0 \int_{\partial\Omega} \mathrm{f}_0\bm{\mathrm{T}}^*\Big(\int_{\partial\Omega}\partial_\mathrm{t}^3\partial_\nu u(\mathrm{y}, \mathrm{t}_\mathrm{m} ) (\mathrm{v}-\mathrm{y})\cdot\nu_\mathrm{v} d\sigma_\mathrm{y}\Big)d\sigma_\mathrm{v}\Big|\\ &\lesssim |\partial\Omega|^\frac{1}{2}\Vert \mathrm{f}_0 \Vert_{\mathrm{L}^2_0(\partial\Omega} \Vert \bm{\mathrm{T}} \Vert_{\mathcal{L}(\mathrm{L}^2_0(\partial\Omega))} \underbrace{\Big|\int_{\partial\Omega}\partial_\mathrm{t}^3\partial_\nu u(\mathrm{y}, \mathrm{t}_\mathrm{m} ) (\mathrm{v}-\mathrm{y})\cdot\nu_\mathrm{v} d\sigma_\mathrm{y}\Big|}_{\sim \delta^4\quad \text{by}\ (\ref{est2})}
\sim \delta^{6-\mathrm{q}},
\end{align}
\begin{align}\label{4}
&\nonumber\Big|\frac{\gamma\rho_\mathrm{m}}{\alpha} \int_{\partial\Omega} \mathrm{f}_0\bm{\mathrm{T}}^*\Big( \partial_\nu\int_\Omega\frac{1}{4\pi|\mathrm{v}-\mathrm{y}|}\partial_\mathrm{t}^2u(\mathrm{y},\mathrm{t})d\mathrm{y}\Big)d\sigma_\mathrm{v}\Big|   \\ &\lesssim \frac{1}{\alpha}\Vert \mathrm{f}_0 \Vert_{\mathrm{L}^2_0(\partial\Omega} \Vert \bm{\mathrm{T}} \Vert_{\mathcal{L}(\mathrm{L}^2_0(\partial\Omega))} \underbrace{\Big\Vert\partial_\nu\int_\Omega\frac{1}{4\pi|\mathrm{v}-\mathrm{y}|}\partial_\mathrm{t}^2u(\mathrm{y},\mathrm{t})d\mathrm{y}\Big\Vert_{\mathrm{L}^2(\partial\Omega)}}_{\sim \delta^2\quad \text{by}\ (\ref{est45})}
\sim \delta^{5-\mathrm{q}},
\end{align} 
and
\begin{align}\label{5}
&\nonumber\Big|\mathrm{c}^{-2}_0\frac{\gamma\rho_\mathrm{m}}{\alpha} \int_{\partial\Omega} \mathrm{f}_0\bm{\mathrm{T}}^*\Big( \partial_\nu\int_\Omega|\mathrm{x}-\mathrm{y}|\partial_\mathrm{t}^4u(\mathrm{y},\mathrm{t}_3)d\mathrm{y}\Big)d\sigma_\mathrm{v}\Big|   \\&\lesssim  |\partial\Omega|^\frac{1}{2}\frac{1}{\alpha}\Vert \mathrm{f}_0 \Vert_{\mathrm{L}^2_0(\partial\Omega} \Vert \bm{\mathrm{T}} \Vert_{\mathcal{L}(\mathrm{L}^2_0(\partial\Omega))} \underbrace{\Big|\partial_\nu\int_\Omega|\mathrm{x}-\mathrm{y}|\partial_\mathrm{t}^4u(\mathrm{y},\mathrm{t}_3)d\mathrm{y}\Big|}_{\sim \delta^2\quad \text{by}\ (\ref{est451})}
\sim \delta^{7-\mathrm{q}}.  
\end{align}
Therefore, using the estimates (\ref{1}), (\ref{2}), (\ref{3}), (\ref{4}), (\ref{5}) we arrive at
\begin{align}\label{e2}
    \frac{\alpha\rho_\mathrm{m}}{4\pi}\int_{\partial\Omega}\Big(\frac{1}{|\mathrm{x}-\mathrm{y}|}-\frac{1}{|\partial\Omega|}\int_{\partial\Omega}\frac{1}{|\mathrm{x}-\mathrm{y}|}\Big) \partial_\nu u (\mathrm{y},\mathrm{t}-\mathrm{c}_0^{-1}|\mathrm{x}-\mathrm{y}|) d\sigma_\mathrm{y}  \sim \delta^{2-\mathrm{q}}, \ \text{where}\ \mathrm{q} \in [0,1].
\end{align}
\bigbreak
\noindent
With the aid of the aforementioned estimate (\ref{e2}) and Proposition \ref{p1}, we can obtain the following after inserting the value of $\mathrm{Y}(\mathrm{t})$ in (\ref{finalestimate}).  Additionally, if $\mathrm{x}\in \mathbb{R}^{\textcolor{black}{3}}\setminus\overline{\Omega}$ such that $\textbf{dist}(\mathrm{x},\Omega)\sim \delta^\mathrm{q}$ and then $|\mathrm{x}-\mathrm{z}| \sim \delta +\delta^\mathrm{q}$, where $\mathrm{q}\in [0,1]$, we have
\begin{align}\label{formula}
    u^\mathrm{s}(\mathrm{x},\mathrm{t})
    = \frac{\alpha \rho_\mathrm{m}\mathrm{p}^{-\frac{1}{2}}}{4\pi} \textcolor{black}{\frac{\rho_\mathrm{c}}{\mathrm{k}_\mathrm{c}}\mathrm{c}_0^2}\frac{1}{|\partial\Omega|}\int_{\partial\Omega}\frac{1}{|\mathrm{x}-\mathrm{y}|}d\sigma_\mathrm{y}\int_0^{\mathrm{t}-\mathrm{c}_0^{-1}|\mathrm{x}-\mathrm{z}|} \sin\Big(\mathrm{p}^{-\frac{1}{2}}(\mathrm{t}-\mathrm{c}_0^{-1}|\mathrm{x}-\mathrm{z}|-\tau)\Big)\int_{\partial\Omega} \partial_\nu u^\mathrm{i} d\tau  + \mathcal{O}(\delta^{2-\mathrm{q}}).
\end{align}
Then the approximation derived in (\ref{formula}) can be estimated more precisely with the following Lemma.
\begin{lemma}\label{lemma}
We have the following approximation of $u^s$ 
\begin{align}
    u^\mathrm{s}(\mathrm{x},\mathrm{t})
    = \frac{\omega_\mathrm{M}\rho_\mathrm{m}|\mathrm{B}|}{4\pi\overline{\mathrm{k}_\mathrm{c}}} \delta \frac{1}{|\partial\Omega|}\int_{\partial\Omega}\frac{1}{|\mathrm{x}-\mathrm{y}|}d\sigma_\mathrm{y}\int_0^{\mathrm{t}-\mathrm{c}_0^{-1}|\mathrm{x}-\mathrm{z}|} \sin\big(\omega_\mathrm{M}(\mathrm{t}-\mathrm{c}_0^{-1}|\mathrm{x}-\mathrm{z}|-\tau)\big)u_{\mathrm{t}\mathrm{t}}^\mathrm{i}(\mathrm{z},\mathrm{\tau}) d\tau + \mathcal{O}(\delta^{2-\mathrm{q}}),
\end{align}
where $\omega_\mathrm{M} = \sqrt{
\frac{2\overline{\mathrm{k}_\mathrm{c}}}{\mathrm{A}_{\partial\mathrm{B}}\rho_\mathrm{m}}}$ and $\displaystyle
    \mathrm{A}_{\partial\Omega} := \frac{1}{|\partial \Omega|}\int_{\partial\Omega}\int_{\partial\Omega}\frac{(\mathrm{x}-\mathrm{y})\cdot\nu}{|\mathrm{x}-\mathrm{y}|}d\sigma_\mathrm{x}d\sigma_\mathrm{y}= \delta^2 \mathrm{A}_{\partial\mathrm{B}}.
$
\end{lemma}
\begin{proof}
See Section \ref{appen} for the proof.
\end{proof}
\bigbreak
\noindent
This completes the proof of Theorem \ref{1.1}.

%--------------------------------------------------------------------------
%--------------------------------------------------------------------------

\section{A priori estimates}\label{apri}     % Section

\subsection{Preliminaries}          % Subsection

We state the scaling properties for both the space and time variables. Let us define $\mathrm{T}_\delta := \mathrm{T}/\delta$. For both functions $\varphi$ and $\psi$ on $\Omega \times (0,\mathrm{T})$ and $\mathrm{B}\times (0,\mathrm{T}_\delta)$ as well as on $\partial\Omega \times (0,\mathrm{T})$ and $\partial\mathrm{B}\times (0,\mathrm{T}_\delta)$, respectively, we use the notation
\begin{align}\nonumber
    \hat{\varphi}(\xi,\tau) = \varphi^\wedge(\xi,\tau) := \varphi(\delta\xi + \mathrm{z},\delta\tau), \quad \quad (\xi,\tau) \in \partial\mathrm{B}\times (0,\mathrm{T}_\delta) \quad \text{and} \quad (\xi,\tau) \in \mathrm{B}\times (0,\mathrm{T}_\delta)\quad \text{respectively},\\ \nonumber
    \check{\psi}(\mathrm{x},\mathrm{t}) = \psi^\vee(\mathrm{x},\mathrm{t}) := \psi(\frac{\mathrm{x}-\mathrm{z}}{\delta},\frac{\mathrm{t}}{\delta}), \quad \quad (\mathrm{x},\mathrm{t}) \in \partial\Omega\times (0,\mathrm{T}_\delta) \quad \text{and} \quad (\mathrm{x},\mathrm{t}) \in \Omega\times (0,\mathrm{T}_\delta)\quad \text{respectively}.
\end{align}
We also note that
\begin{align}\label{ss2}
    \frac{\partial^\mathrm{n}\varphi^\wedge(\cdot,\tau)}{\partial\tau^\mathrm{n}} = \frac{\partial^\mathrm{n}\varphi(\cdot,\delta\tau)}{\partial\tau^\mathrm{n}} = \delta^\mathrm{n}\frac{\partial^\mathrm{n}\varphi(\cdot,\mathrm{t})}{\partial\tau^\mathrm{n}}, \quad \quad \mathrm{n} \in \mathbb{Z}_{+}.
\end{align}

\subsection{Proof of Proposition \ref{p1}}          % Subsection

We begin by examining equation (\ref{wellposed}) in the Fourier-Laplace domain \cite{monk}. Here, we consider the transform parameter $\bm{\mathrm{s}}=\sigma+i\omega \in \mathbb{C}$, where $\sigma\in \mathbb{R}$ such that $\sigma>\sigma_0>0$ for some constant $\sigma_0$, and $\omega \in \mathbb{R}$. We define the function $u^{\ell}(\mathrm{x},\bm{\mathrm{s}})$ as follows:
\begin{align}
\nonumber
u^{\ell}(\mathrm{x},\bm{\mathrm{s}}) := \int_0^\infty u(\mathrm{x},\mathrm{t}) \exp{(-\bm{\mathrm{s}}\mathrm{t})}d\mathrm{t}.
\end{align}
Next, we consider equation (\ref{wellposed}) with $\mathrm{F}(\mathrm{x},\mathrm{t}):= (\mathrm{k}_\mathrm{m}^{-1} - \mathrm{k}^{-1})u^\mathrm{i}_{\mathrm{t}\mathrm{t}} - \text{div} (\rho_\mathrm{m}^{-1} - \rho^{-1})\nabla u^\mathrm{i}$, and take its Laplace transformation with respect to the time variable to obtain:
\begin{align}\label{wellesti1}
\mathrm{k}^{-1}\bm{\mathrm{s}}^2(u^\mathrm{s})^\ell- \text{div}\rho^{-1}\nabla (u^\mathrm{s})^\ell = \mathrm{F}^\ell(\mathrm{x},\bm{\mathrm{s}}),\ \text{in}\ \mathbb{R}^3\times(0,\mathrm{T}),
\end{align}
Furthermore, we derive the an expression through its variational form, as outlined in Section \ref{well1}. This expression is given by:
\begin{align}\label{vf}
\overline{\bm{\mathrm{s}}}|\bm{\mathrm{s}}|^2 \|\mathrm{k}^{-1}(u^\mathrm{s})^\ell(\cdot,\bm{\mathrm{s}})\|^2_{\mathrm{L}^2(\mathbb R^3)} + \overline{\bm{\mathrm{s}}} \left\|\rho^{-1} \nabla(u^\mathrm{s})^\ell(\cdot,\bm{\mathrm{s}})\right\|^2_{\mathrm{L}^2(\mathbb R^3)} = \overline{\bm{\mathrm{s}}} \big\langle\mathrm{F}^\ell(\cdot,\bm{\mathrm{s}}), (u^\mathrm{s})^\ell(\cdot,\bm{\mathrm{s}}) \big\rangle.
\end{align}
Then, after some straightforward calculations, we obtain the following inequality:
\begin{align}\label{ineq}
    &\nonumber\|\mathrm{k}^{-1}(u^\mathrm{s})^\ell(\cdot,\bm{\mathrm{s}})\|^2_{\mathrm{L}^2(\mathbb R^3)} + \left\|\rho^{-1} \nabla(u^\mathrm{s})^\ell(\cdot,\bm{\mathrm{s}})\right\|^2_{\mathrm{L}^2(\mathbb R^3)} \\&\ge\nonumber \frac{1}{2}\Big(|\mathrm{k}^{-1}(u^\mathrm{s})^\ell(\cdot,\bm{\mathrm{s}})\|_{\mathrm{L}^2(\mathbb R^3)} + \left\|\rho^{-1} \nabla(u^\mathrm{s})^\ell(\cdot,\bm{\mathrm{s}})\right\|_{\mathrm{L}^2(\mathbb R^3)}\Big)^2
    \\&\ge\frac{1}{2}\delta^{-2} \Vert (u^\mathrm{s})^\ell(\cdot,\bm{\mathrm{s}})\Vert_{\mathrm{H}^{1}(\Omega)} \Vert (u^\mathrm{s})^\ell(\cdot,\bm{\mathrm{s}})\Vert_{\mathrm{H}^{1}(\mathbb R^3)}.
\end{align}
We also observe that
\begin{align}\label{ineq2}
    \big\langle\mathrm{F}^\ell(\cdot,\bm{\mathrm{s}}), (u^\mathrm{s})^\ell(\cdot,\bm{\mathrm{s}}) \big\rangle_{\mathrm{H}^{-1}(\mathbb R^3),\mathrm{H}^{1}(\mathbb R^3)} \lesssim \delta^{-\frac{1}{2}} \Vert (u^\mathrm{s})^\ell(\cdot,\bm{\mathrm{s}})\Vert_{\mathrm{H}^{1}(\mathbb R^3)}
\end{align}
Thereafter, taking the real part of equation (\ref{vf}), utilizing the coercivity of the variational form and the estimates (\ref{ineq}), (\ref{ineq2}), we obtain the following estimate:
\begin{align}\label{addition}
\Vert (u^\mathrm{s})^\ell(\cdot,\bm{\mathrm{s}})\Vert_{\mathrm{H}^{1}(\Omega)} \le \delta^\frac{3}{2}\frac{2|\bm{\mathrm{s}}|}{\min{\{\sigma,\sigma^3\}}}.
\end{align}
Let us now define the inverse Laplace transform of $(u^\mathrm{s})^\ell(\mathrm{x},\cdot)$ for $\Re(\bm{\mathrm{s}}) =\sigma>0$ as:
\begin{align}
u^\mathrm{s}(\mathrm{x},\mathrm{t}):= \frac{1}{2\pi i}\int_{\sigma-i\infty}^{\sigma+i\infty}e^{\bm{\mathrm{s}}\mathrm{t}}(u^\mathrm{s})^\ell(\mathrm{x},\bm{\mathrm{s}})d\bm{\mathrm{s}} = \frac{1}{2\pi}\int_{-\infty}^{\infty}e^{(\sigma+i\omega)\mathrm{t}}(u^\mathrm{s})^\ell(\mathrm{x},\sigma+i\omega)d\omega.
\end{align}
Based on the estimate with respect to $\bm{\mathrm{s}}$ in Equation (\ref{addition}), $u^\mathrm{s}(\mathrm{x},\mathrm{t})$ is well-defined. Furthermore, one can demonstrate that $u^\mathrm{s}(\mathrm{x},\mathrm{t})$ is independent of $\sigma$ by applying a classical contour integration method, as detailed in \cite[pp. 39]{sayas}.
\newline
If we consider the Fourier transform with respect to the time variable $\mathrm{t}$, then we obtain $\Fourier_{\mathrm{t}\to \omega}\big(e^{-\sigma \mathrm{t}} \partial^\mathrm{k}_\mathrm{t}u^\mathrm{s}(\mathrm{x},\mathrm{t})\big)=\bm{\mathrm{s}}^\mathrm{k}(u^\mathrm{s})^\ell(\mathrm{x},\bm{\mathrm{s}})$, where $\bm{\mathrm{s}}=\sigma+i\omega$. Thus, for $\mathrm{r}\in \mathbb N$, we have the following estimate:
\begin{align}
\nonumber
  \|u^\mathrm{s}\|^2_{\mathrm{H}^{\mathrm{r}}_{0,\sigma}\big((0,\mathrm{T}); \mathrm{H}^1(\Omega)\big)} & = \int_0^{\mathrm{T}} e^{-2\sigma \mathrm{t}} \sum_{\mathrm{k}=0}^{\mathrm{r}} \mathrm{T}^{2\mathrm{k}} \|\partial^\mathrm{k}_\mathrm{t} u^\mathrm{s}(\cdot,\mathrm{t})\|^2_{\mathrm{H}^1(\Omega)} \,d\mathrm{t} \\
  & \nonumber \lesssim \int_{\mathbb{R}_+} \int_{\Omega} e^{-2\sigma \mathrm{t}} \sum_{\mathrm{k}=0}^{\mathrm{r}} \Big[|\partial^\mathrm{k}_\mathrm{t} u^\mathrm{s}(\mathrm{x},\mathrm{t})|^2 + |\partial^\mathrm{k}_\mathrm{t} \nabla u^\mathrm{s}(\mathrm{x},\mathrm{t})|^2\Big]  \, d\mathrm{x}\, d\mathrm{t} \\
  & \nonumber \lesssim \int_{\Omega} \int_{\mathbb{R}} \sum_{\mathrm{k}=0}^{\mathrm{r}} \Big[\big|\Fourier ( e^{-\sigma \mathrm{t}} \partial^\mathrm{k}_\mathrm{t} u^\mathrm{s}(\mathrm{x},\mathrm{t})\big|^2+ \big|\Fourier ( e^{-\sigma \mathrm{t}} \partial^\mathrm{k}_\mathrm{t} \nabla u^\mathrm{s}(\mathrm{x},\mathrm{t})\big|^2\Big] \, d\mathrm{t}\, d\mathrm{x} \\
  & \nonumber\lesssim \sum_{\mathrm{k}=0}^{\mathrm{r}} \int_{\sigma+i\mathbb{R}} |\bm{\mathrm{s}}|^{2\mathrm{k}} \|(u^\mathrm{s})^\ell(\cdot,\bm{\mathrm{s}})\|^2_{\mathrm{H}^1(\Omega)} \, d\bm{\mathrm{s}} \simeq \delta^3.
\end{align}
Consider the total field $u=u^\mathrm{s}+u^\mathrm{i}$, where $u^\mathrm{s}$ and $u^\mathrm{i}$ represent the scattered and incident fields, respectively. Assuming the smoothness of $u^\mathrm{i}$, we draw the following conclusions:
\begin{align}
\|u\|_{\mathrm{H}^{\mathrm{r}}_{0,\sigma}\big((0,\mathrm{T}); \mathrm{L}^2(\Omega)\big)} \sim \delta^\frac{3}{2}\quad \text{and}; \|\nabla u\|_{\mathrm{H}^{\mathrm{r}}_{0,\sigma}\big((0,\mathrm{T}); \mathrm{L}^2(\Omega)\big)} \sim \delta^\frac{3}{2}.
\end{align}
As a result, we estimate:
\begin{align}
\Vert u(\cdot,\mathrm{t})\Vert_{\mathrm{L}^2(\Omega)} \nonumber &\lesssim \Vert u\Vert_{\mathrm{H}_{0,\sigma}^\mathrm{r}(0,\mathrm{T};\mathrm{L}^2(\Omega))}
\\ &\lesssim \delta^{\frac{3}{2}},
\end{align}
and
\begin{align}
\Vert\nabla u(\cdot,\mathrm{t})\Vert_{\mathrm{L}^2(\Omega)} \nonumber &\lesssim \Vert \nabla u\Vert_{\mathrm{H}_{0,\sigma}^\mathrm{r}(0,\mathrm{T};\mathrm{L}^2(\Omega))}
\\ &\lesssim \delta^{\frac{3}{2}}.
\end{align}
Furthermore, we can deduce that:
\begin{align}\label{esti2}
\Vert \partial_\mathrm{t}^\mathrm{k}u(\cdot,\mathrm{t})\Vert_{\mathrm{L}^2(\Omega)} \lesssim \delta^{\frac{3}{2}}, \quad \mathrm{t}\in [0,\mathrm{T}], \quad \mathrm{k}=1,2,\cdot \cdot \cdot.
\end{align}
and
\begin{align}\label{gradu}
\Vert \partial_\mathrm{t}^\mathrm{k}\nabla u(\cdot,\mathrm{t})\Vert_{\mathrm{L}^2(\Omega)} \lesssim \delta^{\frac{3}{2}}, \quad \mathrm{t}\in [0,\mathrm{T}], \quad \mathrm{k}=1,2,\cdot \cdot \cdot.
\end{align}
This completes the proof of Proposition \ref{p1}.

\subsection{Proof of Proposition \ref{p3}}      % Subsection

In order to prove Proposition \ref{p3}, we need the following estimate of $\partial_\mathrm{t}^\mathrm{k}\nabla u(\cdot,\mathrm{t}),\ \mathrm{k}=0,1,\cdot \cdot \cdot.$ 
%\subsubsection{Estimate of $\partial_\mathrm{t}^\mathrm{k}\nabla u(\cdot,\mathrm{t}),\ \mathrm{k}=1,2,\cdot \cdot \cdot,$ }

\begin{proposition}\label{p3.1}
We have the following estimate for $\partial_\mathrm{t}^\mathrm{k}\nabla u(\cdot,\mathrm{t}),\ \mathrm{k}=0,1,\cdot \cdot \cdot.$
\begin{align}\nonumber
\Vert \partial_\mathrm{t}^\mathrm{k}\nabla u(\cdot,\mathrm{t})\Vert_{\mathbb{L}^2(\Omega)} \lesssim \delta^{\textcolor{black}{\frac{5}{2}}}, \quad \mathrm{t}\in [0,\mathrm{T}], \quad \mathrm{k}=0,1,\cdot \cdot \cdot.
\end{align}
\end{proposition}
\begin{proof}
In order to prove the desired estimate for $\partial_\mathrm{t}^\mathrm{k}\nabla u(\cdot,\mathrm{t}),\ \mathrm{k}=1,2,\cdot \cdot \cdot$, we need to improve the obtained estimate (\ref{gradu}). 
Now, let us consider the following Lippmann-Schwinger equation
\begin{align}
\nonumber
    u + \beta \rho_\mathrm{m}\int_\Omega \frac{u_{\mathrm{t}\mathrm{t}}(\mathrm{y},\mathrm{t}-\mathrm{c}_0^{-1}|\mathrm{x}-\mathrm{y}|)}{4\pi|\mathrm{x}-\mathrm{y}|}d\mathrm{y} - \alpha \rho_\mathrm{m}\text{div}\int_\Omega \frac{\nabla u(\mathrm{y},\mathrm{t}-\mathrm{c}_0^{-1}|\mathrm{x}-\mathrm{y}|)}{4\pi|\mathrm{x}-\mathrm{y}|}d\mathrm{y} = u^\mathrm{i}.
\end{align}  
Thereafter, we take gradient on the both side of the above equation and we rewrite it as follows:
\begin{align}
    \nabla u - \alpha\rho_\mathrm{m} \nabla\text{div}\int_\Omega \frac{\nabla u(\mathrm{y},\mathrm{t}-\mathrm{c}_0^{-1}|\mathrm{x}-\mathrm{y}|)}{4\pi|\mathrm{x}-\mathrm{y}|}d\mathrm{y} =\nabla u^\mathrm{i} - \beta\rho_\mathrm{m}\nabla \int_\Omega \frac{u_{\mathrm{t}\mathrm{t}}(\mathrm{y},\mathrm{t}-\mathrm{c}_0^{-1}|\mathrm{x}-\mathrm{y}|)}{4\pi|\mathrm{x}-\mathrm{y}|}d\mathrm{y}.
\end{align} 
The Taylor's expansion imply the following
\begin{align}
    \nabla u \nonumber &- \alpha\rho_\mathrm{m} \nabla\text{div}\int_\Omega\frac{1}{{4\pi|\mathrm{x}-\mathrm{y}|}}\nabla u(\mathrm{y},\mathrm{t})d\mathrm{y} + \underbrace{\alpha\rho_\mathrm{m} \nabla\text{div}\int_\Omega\frac{1}{{4\pi|\mathrm{x}-\mathrm{y}|}}\mathrm{c}_0^{-1}|\mathrm{x}-\mathrm{y}|\partial_{\mathrm{t}}\nabla u(\mathrm{y},\mathrm{t})d\mathrm{y}}_{\textcolor{black}{= 0}} \\ \nonumber &- \alpha\rho_\mathrm{m} \nabla\text{div}\int_\Omega\frac{1}{{4\pi|\mathrm{x}-\mathrm{y}|}}\mathrm{c}_0^{-4}|\mathrm{x}-\mathrm{y}|^2\partial^2_{\mathrm{t}}\nabla u(\mathrm{y},\mathrm{t}_1)d\mathrm{y} =\nabla u^\mathrm{i} - \beta\rho_\mathrm{m} \nabla \int_\Omega \frac{1}{4\pi|\mathrm{x}-\mathrm{y}|} \partial_\mathrm{t}^2u(\mathrm{y},\mathrm{t})d\mathrm{y} \\ &+ \underbrace{\beta\rho_\mathrm{m} \nabla \int_\Omega \frac{1}{4\pi|\mathrm{x}-\mathrm{y}|}\mathrm{c}_0^{-1}|\mathrm{x}-\mathrm{y}| \partial_\mathrm{t}^3u(\mathrm{y},\mathrm{t})d\mathrm{y}}_{\textcolor{black}{= 0}} - \beta\rho_\mathrm{m} \nabla \int_\Omega \frac{1}{4\pi|\mathrm{x}-\mathrm{y}|}\mathrm{c}_0^{-4}|\mathrm{x}-\mathrm{y}|^2 \partial_\mathrm{t}^4u(\mathrm{y},\mathrm{t}_2)d\mathrm{y},
\end{align} 
where $\mathrm{t}_1, \mathrm{t}_2 \in \big(\mathrm{t}-\mathrm{c}^{-1}_0|\mathrm{x}-\mathrm{y}|,\mathrm{t}\big).$
We set
\begin{align}
    \bm{\mathrm{P}}_\Omega\Big[\partial^2_{\mathrm{t}}\nabla u(\mathrm{y},\mathrm{t}_1)\Big]:=\nabla\text{div}\int_\Omega\frac{1}{{4\pi}}\mathrm{c}_0^{-4}|\mathrm{x}-\mathrm{y}|\partial^2_{\mathrm{t}}\nabla u(\mathrm{y},\mathrm{t}_1)d\mathrm{y},
\end{align}
and 
\begin{align}
    \bm{\mathrm{J}}_\Omega\Big[\partial^4_{\mathrm{t}}u(\mathrm{y},\mathrm{t}_2)\Big]:= \nabla \int_\Omega \frac{1}{4\pi}\mathrm{c}_0^{-4}|\mathrm{x}-\mathrm{y}| \partial_\mathrm{t}^4u(\mathrm{y},\mathrm{t}_2)d\mathrm{y}.
\end{align}
\textcolor{black}{We recall the Magnetization operator related to the Laplace operator in $\Omega$ defined as follows:
\begin{align}
    \nabla \bm{\mathrm{M}}^{(0)}_{\Omega}\big[f\big](\mathrm{x}) := \nabla \int_\Omega\mathop{\nabla}\limits_{y} \frac{1}{4\pi|\mathrm{x}-\mathrm{y}|} \cdot f(\mathrm{y})d\mathrm{y},
\end{align}
and similarly, in the scaled domain $\mathrm{B}$, we have $\nabla \bm{\mathrm{M}}^{(0)}_{\mathrm{B}}\big[f\big]$. }
\newline

\noindent
Therefore, we obtain that 
\begin{align}
    \nabla u + \alpha\nabla\bm{\mathrm{M}}^{(0)}_\Omega[\nabla u] = \nabla u^\mathrm{i} - \beta\nabla \bm{\mathrm{V}}_\Omega[u_{\mathrm{t}\mathrm{t}}] + \alpha\rho_\mathrm{m} \bm{\mathrm{P}}_\Omega\Big[\partial^2_{\mathrm{t}}\nabla u(\mathrm{y},\mathrm{t}_1)\Big] + \beta\rho_\mathrm{m} \bm{\mathrm{J}}_\Omega\Big[\partial^4_{\mathrm{t}}u(\mathrm{y},\mathrm{t}_2)\Big].
\end{align}
Now, for a fixed $'\mathrm{t}'$, we rewrite the above expression in the scaled domain $\mathrm{B}$ as follows:
\begin{align}
    \widehat{\nabla u} + \alpha\nabla\bm{\mathrm{M}}^{(0)}_{\mathrm{B}}[\widehat{\nabla u}] = \widehat{\nabla u^\mathrm{i}} - \beta \delta\nabla \bm{\mathrm{V}}_\mathrm{B}^\delta[\hat{u}_{\mathrm{t}\mathrm{t}}] + \alpha\rho_\mathrm{m} \delta^2 \bm{\mathrm{P}}_\mathrm{B}^\delta\Big[\partial^2_{\mathrm{t}}\widehat{\nabla u}(\xi,\mathrm{t}_1)\Big] + \beta\rho_\mathrm{m}\delta^3\bm{\mathrm{J}}_\mathrm{B}^\delta\Big[\partial^4_{\mathrm{t}}\hat{u}(\mathrm{y},\mathrm{t}_2)\Big] .
\end{align}
We study the system of integral equations in the Hilbert space of vector-valued function $\big(\mathrm{L}^{2}(\textcolor{black}{\mathrm{B}})\big)^3.$ For the sake of simplicity, we use $\mathbb{L}^2(\textcolor{black}{\mathrm{B}}) =\big(\mathrm{L}^{2}(\textcolor{black}{\mathrm{B}})\big)^3$. This space can be decomposed into three sub-spaces as a direct sum as following, see \cite{raveski},
\begin{equation*}
    \mathbb{L}^{2} = \mathbb{H}_{0}(\text{div}\;0,\mathrm{B}) \oplus\mathbb{H}_{0}(\text{curl}\;0,\mathrm{B})\oplus \nabla \mathbb{H}_{\text{arm}}.
\end{equation*}
Consider $\big(\mathrm{e}^{(1)}_{\mathrm{n}}\big)_{\mathrm{n} \in \mathbb{N}}$ and $\big(\mathrm{e}^{(2)}_{\mathrm{n}}\big)_{\mathrm{n} \in \mathbb{N}}$ to be any orthonormal basis of the sub-spaces $\mathbb{H}_{0}(\text{div}\;0,\mathrm{B})$ and $\mathbb{H}_{0}(\text{curl}\;0,\mathrm{B})$ respectively. But for the sub-space $\nabla \mathbb{H}_{\text{arm}}$, we consider the complete orthonormal basis $\big(\mathrm{e}^{(3)}_{\mathrm{n}}\big)_{n \in \mathbb{N}}$ derived as the eigenfunctions of the magnetization operator $\nabla\bm{\mathrm{M}}^{\textcolor{black}{{(0)}}}_{\textcolor{black}{\mathrm{B}}}: \nabla \mathbb{H}_{\text{arm}}\rightarrow \nabla \mathbb{H}_{\text{arm}}$, \cite{friedmanI}, with $\big(\mathrm{\lambda}^{(3)}_{\mathrm{n}}\big)_{n \in \mathbb{N}}$ as the corresponding eigenvalues.
\bigbreak
\noindent
Then, from the definition of the sub-space $\mathbb{H}_{0}(\text{div}\;0,\mathrm{B})$ and integration by parts, we obtain
\begin{align}\label{pr1}
\Big\langle \bm{\mathrm{P}}_\mathrm{B}^\delta\Big[\partial^2_{\mathrm{t}}\nabla u(\mathrm{\eta},\mathrm{t}_1)\Big];\mathrm{e}^{(1)}_{\mathrm{n}}\Big\rangle \nonumber&=\Big\langle  \nabla\text{div}\int_\mathrm{B}\frac{1}{{4\pi|\mathrm{\xi}-\mathrm{\eta}|}}\mathrm{c}_0^{-4}|\mathrm{\xi}-\mathrm{\eta}|^2\partial^2_{\mathrm{t}}\nabla u(\mathrm{\eta},\mathrm{t}_1)d\mathrm{\eta};\mathrm{e}^{(1)}_{\mathrm{n}}\Big\rangle_{\mathbb{L}^{2}(\mathrm{B})}\\ \nonumber&=\int_{\mathrm{B}} \mathrm{e}^{(1)}_{\mathrm{n}}\cdot \nabla \text{div}\int_\mathrm{B}\frac{1}{{4\pi|\mathrm{\xi}-\mathrm{\eta}|}}\mathrm{c}_0^{-4}|\mathrm{\xi}-\mathrm{\eta}|^2\partial^2_{\mathrm{t}}\nabla u(\mathrm{\eta},\mathrm{t}_1)d\mathrm{\eta}\\ \nonumber
&=-\int_{\Omega} \nabla\cdot \mathrm{e}^{(1)}_{\mathrm{n}}\  \text{div}\int_\mathrm{B}\frac{1}{{4\pi|\mathrm{\xi}-\mathrm{\eta}|}}\mathrm{c}_0^{-4}|\mathrm{\xi}-\mathrm{\eta}|^2\partial^2_{\mathrm{t}}\nabla u(\mathrm{\eta},\mathrm{t}_1)d\mathrm{\eta} \\ &+ \int_{\partial \Omega} \mathrm{e}^{(1)}_{\mathrm{n}}\cdot\nu \  \text{div}\int_\mathrm{B}\frac{1}{{4\pi|\mathrm{\xi}-\mathrm{\eta}|}}\mathrm{c}_0^{-4}|\mathrm{\xi}-\mathrm{\eta}|^2\partial^2_{\mathrm{t}}\nabla u(\mathrm{\eta},\mathrm{t}_1)d\mathrm{\eta} = 0,
\end{align}
and similarly we can show that $\Big\langle \bm{\mathrm{J}}_\mathrm{B}^\delta\Big[\partial^4_{\mathrm{t}} u(\mathrm{\eta},\mathrm{t}_2)\Big];\mathrm{e}^{(1)}_{\mathrm{n}}\Big\rangle = 0.$
Moreover, using the same arguments, we have
\begin{align}\label{eq:incident}
    \Big\langle \widehat{\nabla u^{\mathrm{i}}};\mathrm{e}^{(1)}_{\mathrm{n}}\Big\rangle_{\mathbb{L}^{2}(\mathrm{B})} = 0,
\end{align}
and
\begin{align}
    \Big\langle \nabla\bm{\mathrm{M}}_{\mathrm{B}}^{\textcolor{black}{(0)}}[\widehat{\nabla u}];\mathrm{e}^{(1)}_{\mathrm{n}}\Big\rangle_{\mathbb{L}^{2}(\mathrm{B})} = 0.
\end{align}
Furthermore, due to the definition of the subspace $\mathbb{H}_{0}(\text{curl}\;0,\mathrm{B})$ and integration by parts, we get
\begin{align}\label{eq:2}
    (1+\alpha)\Big\langle \widehat{\nabla u};\mathrm{e}^{(2)}_{\mathrm{n}}\Big\rangle_{\mathbb{L}^{2}(\mathrm{B})} &=\nonumber  \Big\langle\widehat{\nabla u}^{\textbf{i}};\mathrm{e}^{(2)}_{\mathrm{n}}\Big\rangle_{\mathbb{L}^{2}(\mathrm{B})}-\beta\delta\Big\langle\nabla \bm{\mathrm{V}}_\mathrm{B}^\delta[\hat{u}_{\mathrm{t}\mathrm{t}}];\mathrm{e}^{(2)}_\mathrm{n} \Big\rangle_{\mathbb{L}^{2}(\mathrm{B})} + \alpha\rho_\mathrm{m}\delta^2\Big\langle \bm{\mathrm{P}}_\mathrm{B}^\delta\Big[\partial^2_\mathrm{t}\widehat{\nabla u}\Big];\ \mathrm{e}^{(2)}_\mathrm{n} \Big\rangle_{\mathbb{L}^{2}(\mathrm{B})}
    \\ &+ \beta\rho_\mathrm{m}\delta^3\Big\langle \bm{\mathrm{J}}_\mathrm{B}^\delta\Big[\partial^4_\mathrm{t}\widehat{ u}\Big];\ \mathrm{e}^{(2)}_\mathrm{n} \Big\rangle_{\mathbb{L}^{2}(\mathrm{B})}
\end{align}
which implies that
\begin{align}
    |\Big\langle \widehat{\nabla u};\mathrm{e}^{(2)}_{\mathrm{n}}\Big\rangle_{\mathbb{L}^{2}(\mathrm{B})}| &\lesssim \nonumber \frac{1}{1+\alpha} |\Big\langle\widehat{\nabla u}^{\textbf{i}};\mathrm{e}^{(2)}_{\mathrm{n}}\Big\rangle_{\mathbb{L}^{2}(\mathrm{B})}| + \frac{\beta}{1+\alpha}\delta|\Big\langle\nabla \bm{\mathrm{V}}_\mathrm{B}^\delta[\hat{u}_{\mathrm{t}\mathrm{t}}];\mathrm{e}^{(2)}_\mathrm{n} \Big\rangle_{\mathbb{L}^{2}(\mathrm{B})}| \\ &+ \frac{\alpha\rho_\mathrm{m}}{1+\alpha}\delta^2|\Big\langle \bm{\mathrm{P}}_\mathrm{B}^\delta\Big[\partial^2_\mathrm{t}\widehat{\nabla u}\Big];\ \mathrm{e}^{(2)}_\mathrm{n} \Big\rangle_{\mathbb{L}^{2}(\mathrm{B})}| + \frac{\beta\rho_\mathrm{m}}{1+\alpha}\delta^3|\Big\langle \bm{\mathrm{J}}_\mathrm{B}^\delta\Big[\partial^4_\mathrm{t}\widehat{u}\Big];\ \mathrm{e}^{(2)}_\mathrm{n} \Big\rangle_{\mathbb{L}^{2}(\mathrm{B})}|
\end{align}
Similarly, we derive the following estimate on the subspace $\nabla \mathrm{H}_\text{arm}$ 
\begin{align}
    |\Big\langle \widehat{\nabla u};\mathrm{e}^{(3)}_{\mathrm{n}}\Big\rangle_{\mathbb{L}^{2}(\mathrm{B})}| &\lesssim \nonumber \frac{1}{1+\alpha \lambda^3_\mathrm{n}} |\Big\langle\widehat{\nabla u}^{\textbf{i}};\mathrm{e}^{(3)}_{\mathrm{n}}\Big\rangle_{\mathbb{L}^{2}(\mathrm{B})}| + \frac{\beta}{1+\alpha\lambda^3_\mathrm{n}}\delta|\Big\langle\nabla \bm{\mathrm{V}}_\mathrm{B}^\delta[\hat{u}_{\mathrm{t}\mathrm{t}}];\mathrm{e}^{(3)}_\mathrm{n} \Big\rangle_{\mathbb{L}^{2}(\mathrm{B})}|\\ &+ \frac{\alpha\rho_\mathrm{m}}{1+\alpha\lambda^3_\mathrm{n}}\delta^2|\Big\langle \bm{\mathrm{P}}_\mathrm{B}^\delta\Big[\partial^2_\mathrm{t}\widehat{\nabla u}\Big];\ \mathrm{e}^{(3)}_\mathrm{n} \Big\rangle_{\mathbb{L}^{2}(\mathrm{B})}| + \frac{\beta\rho_\mathrm{m}}{1+\alpha\lambda^3_\mathrm{n}}\delta^2|\Big\langle \bm{\mathrm{J}}_\mathrm{B}^\delta\Big[\partial^4_\mathrm{t}\widehat{u}\Big];\ \mathrm{e}^{(3)}_\mathrm{n} \Big\rangle_{\mathbb{L}^{2}(\mathrm{B})}|
\end{align}
Therefore, due to Parseval's identity, the estimates $\rho_\mathrm{m}\sim1, \frac{\beta}{1+\alpha}\sim 1, \frac{\beta}{1+\alpha\lambda^3_\mathrm{n}} \sim 1, \frac{\alpha}{1+\alpha}\sim 1, \frac{\alpha}{1+\alpha\lambda^3_\mathrm{n}}\sim 1, \frac{\beta}{1+\alpha\lambda^3_\mathrm{n}}\sim 1\ \text{and}\  \frac{\beta}{1+\alpha}\sim 1$, and the continuity of the operators $\bm{\mathrm{V}}^{\textcolor{black}{\delta}}_{\textcolor{black}{\mathrm{B}}}$, $\bm{\bm{\mathrm{P}}}^{\textcolor{black}{\delta}}_{\textcolor{black}{\mathrm{B}}}$ and $\bm{\bm{\mathrm{J}}}^{\textcolor{black}{\delta}}_{\textcolor{black}{\mathrm{B}}}$, we obtain 
\begin{align}\label{en}
    \Vert \widehat{\nabla u}(\cdot,\mathrm{t})\Vert_{\mathbb{L}^{2}(\mathrm{B})}^2 &\lesssim \nonumber\delta^2 \Vert \nabla \bm{\mathrm{V}}_\mathrm{B}^\delta\Big[\hat{u}_{\mathrm{t}\mathrm{t}}(\cdot,\mathrm{t})\Big]\Vert_{\mathbb{L}^{2}(\mathrm{B})}^2 + \delta^4 \Vert \bm{\mathrm{P}}_\mathrm{B}^\delta\Big[\partial^2_\mathrm{t}\widehat{\nabla u}(\cdot,\mathrm{t})\Big]\Vert_{\mathbb{L}^{2}(\mathrm{B})}^2 +  \delta^6 \Vert \bm{\mathrm{J}}_\mathrm{B}^\delta\Big[\partial^4_\mathrm{t}\widehat{ u}(\cdot,\mathrm{t})\Big]\Vert_{\mathbb{L}^{2}(\mathrm{B})}^2
    \\ &\lesssim \nonumber \delta^2 \underbrace{\Vert \partial_\mathrm{t}^\mathrm{k}\hat{u}(\cdot,\mathrm{t})\Vert_{\mathbb{L}^2(\mathrm{B})}^2}_{\textcolor{black}{\sim 1}(\ref{esti2})} + \delta^4 \underbrace{\Vert \partial_\mathrm{t}^\mathrm{k}\widehat{\nabla u}(\cdot,\mathrm{t})\Vert_{\mathbb{L}^2(\mathrm{B})}^2}_{\textcolor{black}{\sim 1}(\ref{gradu})} + \delta^6 \underbrace{\Vert \partial_\mathrm{t}^\mathrm{k}\widehat{ u}(\cdot,\mathrm{t})\Vert_{\mathbb{L}^2(\mathrm{B})}^2}_{\textcolor{black}{\sim 1}(\ref{esti2})}
    \\ &\lesssim \delta^2.
\end{align}
As a result, we get
\begin{align}\label{en1}
    \Vert \nabla u(\cdot,\mathrm{t})\Vert_{\mathbb{L}^{2}(\Omega)} = \delta^{\frac{3}{2}}\Vert \widehat{\nabla u}(\cdot,\mathrm{t})\Vert_{\mathbb{L}^{2}(\mathrm{B})} \sim \delta^{\textcolor{black}{\frac{5}{2}}}.
\end{align}
Similarly, we derive the estimate
\begin{align}
\Vert \partial_\mathrm{t}^\mathrm{k}\nabla u(\cdot,\mathrm{t})\Vert_{\mathbb{L}^2(\Omega)} \lesssim \delta^{\textcolor{black}{\frac{5}{2}}}, \quad \mathrm{t}\in [0,\mathrm{T}], \quad \mathrm{k}=1,2,\cdot \cdot \cdot.
\end{align}
This gives the desired estimate and the proof is complete.
\end{proof}
\noindent
Now, we proceed to prove Proposition \ref{p3} and therefore, we start with recalling the norm definition as described in (\ref{normd1}) and (\ref{normd2})
\begin{align}\nonumber
    \Vert u\Vert_{\mathrm{H}^{\mathrm{s}}(\partial\Omega)} := \Vert u \Vert_{\mathrm{L}^2(\partial \Omega)}^2 + \int_{\partial\Omega}\int_{\partial\Omega}\frac{|u(\mathrm{x})-u(\mathrm{y})|^2}{|\mathrm{x}-\mathrm{y}|^{2+2\mathrm{s}}}d\sigma_\mathrm{x}d\sigma_\mathrm{y},\quad \text{where}\ \mathrm{s} = \frac{1}{2},
\end{align}
and
\begin{align}
 \Vert \varphi \Vert_{\mathrm{H}^{-\frac{1}{2}}(\partial\Omega)} &=\nonumber \sup_{0\neq u \in \mathrm{H}^{\frac{1}{2}}(\partial\Omega )} \dfrac{|\langle \varphi, u \rangle_{\partial \Omega}|}{\Vert u \Vert_{\mathrm{H}^{\frac{1}{2}}(\partial\Omega )}}.
 \end{align}
 We have the following scaling properties.
 \begin{lemma}
 Suppose $0< \delta \le 1$ and $\Omega=\delta\mathrm{B}+\mathrm{z}\subset \mathbb{R}^3$. Then for $u \in \mathrm{H}^{\frac{1}{2}}(\partial\Omega)$ and $\psi \in \mathrm{H}^{-\frac{1}{2}}(\partial\Omega)$ , we have
\begin{align}\label{l1}
    \delta \Vert \hat{u}\Vert_{\mathrm{H}^{\frac{1}{2}}(\partial\mathrm{B})} \le \Vert u\Vert_{\mathrm{H}^{\frac{1}{2}}(\partial\Omega)} \le \delta^\frac{1}{2} \Vert \hat{u}\Vert_{\mathrm{H}^{\frac{1}{2}}(\partial\mathrm{B})},
\end{align}
and
\begin{align}\label{ii}
    \delta^\frac{3}{2} \Vert \hat{\psi}\Vert_{\mathrm{H}^{-\frac{1}{2}}(\partial\mathrm{B})} \le \Vert \psi\Vert_{\mathrm{H}^{-\frac{1}{2}}(\partial\Omega)} \le \delta\Vert \hat{\psi}\Vert_{\mathrm{H}^{-\frac{1}{2}}(\partial\mathrm{B})}.
\end{align}
\end{lemma}
 \begin{proof}
 Let us first consider $\mathrm{x}=\delta\xi + \mathrm{z}$ and $\mathrm{y}=\delta\eta+\mathrm{z}$. Then for $u \in \mathrm{H}^{\frac{1}{2}}(\partial\Omega)$ we have
\begin{align}
    \Vert u\Vert_{\mathrm{H}^{\frac{1}{2}}(\partial\Omega)}^2 &:=\nonumber \Vert u \Vert_{\mathrm{L}^2(\partial \Omega)}^2 + \int_{\partial\Omega}\int_{\partial\Omega}\frac{|u(\mathrm{x})-u(\mathrm{y})|^2}{|\mathrm{x}-\mathrm{y}|^3}d\sigma_\mathrm{x}d\sigma_\mathrm{y}
    \\ \nonumber &= \delta^2\int_{\partial\mathrm{B}} |u(\delta\mathrm{\xi}+\mathrm{z})|^2 d\sigma_\mathrm{\xi} + \delta^4\int_{\partial\mathrm{B}}\int_{\partial\mathrm{B}}\frac{|u(\delta\mathrm{\xi}+\mathrm{z})-u(\delta\mathrm{\eta}+\mathrm{z})|^2}{\delta^3|\mathrm{\xi}-\mathrm{\eta}|^3}d\sigma_\mathrm{\xi}d\sigma_\mathrm{\eta}
    \\ \nonumber &= \delta^2 \Vert \hat{u} \Vert_{\mathrm{L}^2(\partial \mathrm{B})}^2 + \delta \int_{\partial\mathrm{B}}\int_{\partial\mathrm{B}}\frac{|\hat{u}(\mathrm{\xi})-\hat{u}(\mathrm{\eta})|^2}{|\mathrm{\xi}-\mathrm{\eta}|^3}d\sigma_\mathrm{\xi}d\sigma_\mathrm{\eta}.
\end{align}
Consequently, due to the fact $\delta^2\le \delta$ it implies (\ref{l1}). Then using the definition of dual norm (\ref{normd2}), we have for $\psi \in \mathrm{H}^{-\frac{1}{2}}(\partial\Omega)$
\begin{align}
 \Vert \psi \Vert_{\mathrm{H}^{-\frac{1}{2}}(\partial\Omega)} &=\nonumber \sup_{0\neq \varphi \in \mathrm{H}^{\frac{1}{2}}(\partial\Omega )} \dfrac{|\langle \psi,\varphi \rangle_{\partial \Omega}|}{\Vert \varphi \Vert_{\mathrm{H}^{\frac{1}{2}}(\partial\Omega )}}
 \\ \nonumber &\le \nonumber \sup_{0\neq \varphi \in \mathrm{H}^{\frac{1}{2}}(\partial\mathrm{B} )} \dfrac{\delta^2|\langle \hat{\psi},\hat{\varphi} \rangle_{\partial \Omega}|}{\delta\Vert \hat{\varphi} \Vert_{\mathrm{H}^{\frac{1}{2}}(\partial\mathrm{B} )}}
 \\ & \le \delta \Vert \hat{\psi} \Vert_{\mathrm{H}^{-\frac{1}{2}}(\partial\mathrm{B})},
\end{align}
and
\begin{align}
 \Vert \psi \Vert_{\mathrm{H}^{-\frac{1}{2}}(\partial\Omega)} &=\nonumber \sup_{0\neq \varphi \in \mathrm{H}^{\frac{1}{2}}(\partial\Omega )} \dfrac{|\langle \psi,\varphi \rangle_{\partial \Omega}|}{\Vert \varphi \Vert_{\mathrm{H}^{\frac{1}{2}}(\partial\Omega )}}
 \\ \nonumber &\ge \nonumber \sup_{0\neq \varphi \in \mathrm{H}^{\frac{1}{2}}(\partial\mathrm{B} )} \dfrac{\delta^2|\langle \hat{\psi},\hat{\varphi} \rangle_{\partial \Omega}|}{\delta^\frac{1}{2}\Vert \hat{\varphi} \Vert_{\mathrm{H}^{\frac{1}{2}}(\partial\mathrm{B} )}}
 \\ & \ge \delta^\frac{3}{2} \Vert \hat{\psi} \Vert_{\mathrm{H}^{-\frac{1}{2}}(\partial\mathrm{B})}.
\end{align}
Therefore, the above two estimates leads to (\ref{ii}). 
\end{proof}

\begin{lemma}
Suppose $0< \delta \le 1$. Then for $\partial_\nu u \in \mathrm{H}^{-\frac{1}{2}}(\partial\Omega)$, we have
\begin{align}\label{l3}
    \delta^\frac{1}{2} \Vert \partial_\nu\hat{u}\Vert_{\mathrm{H}^{-\frac{1}{2}}(\partial\mathrm{B})} \le \Vert \partial_\nu u\Vert_{\mathrm{H}^{-\frac{1}{2}}(\partial\Omega)} \le \Vert \partial_\nu\hat{u}\Vert_{\mathrm{H}^{-\frac{1}{2}}(\partial\mathrm{B})}.
\end{align}
\end{lemma}
\begin{proof}
Now, we start with the following calculation
\begin{align}
 \Vert \partial_\nu u\Vert_{\mathrm{H}^{-\frac{1}{2}}(\partial\Omega)} 
 &=\nonumber \sup_{0\neq \varphi \in \mathrm{H}^{\frac{1}{2}}(\partial\Omega )} \dfrac{|\langle\partial_\nu u,\varphi \rangle_{\partial \Omega}|}{\Vert \varphi \Vert_{\mathrm{H}^{\frac{1}{2}}(\partial\Omega )}}
 \\ \nonumber &= \sup_{0\neq \varphi \in \mathrm{H}^{\frac{1}{2}}(\partial\Omega )} \dfrac{|\langle\nabla_\mathrm{y} u(\mathrm{y}) \cdot \nu_\mathrm{y},\varphi \rangle_{\partial \Omega}|}{\Vert \varphi \Vert_{\mathrm{H}^{\frac{1}{2}}(\partial\Omega )}}
 \\ \nonumber &= \sup_{0\neq \varphi \in \mathrm{H}^{\frac{1}{2}}(\partial\Omega )} \dfrac{\delta^2|\delta^{-1}\Big\langle\nabla_\mathrm{\eta} u(\delta\eta+\mathrm{z}) \cdot \nu_\mathrm{\eta},\varphi(\delta\eta+\mathrm{z}) \Big\rangle_{\partial \mathrm{B}}|}{\delta\Vert \varphi(\delta\eta+\mathrm{z}) \Vert_{\mathrm{H}^{\frac{1}{2}}(\partial\mathrm{B} )}}
 \\ \nonumber &\le \nonumber \sup_{0\neq \varphi \in \mathrm{H}^{\frac{1}{2}}(\partial\mathrm{B} )} \dfrac{\delta|\langle \partial_\nu\hat{u},\hat{\varphi} \rangle_{\partial \Omega}|}{\underbrace{\delta\Vert \hat{\varphi} \Vert_{\mathrm{H}^{\frac{1}{2}}(\partial\mathrm{B} )}}_{(\ref{l1})}} 
 \\ &=  \Vert \partial_\nu\hat{u} \Vert_{\mathrm{H}^{-\frac{1}{2}}(\partial\mathrm{B})},
\end{align}
and
\begin{align}
 \Vert \partial_\nu u\Vert_{\mathrm{H}^{-\frac{1}{2}}(\partial\Omega)} 
 &=\nonumber \sup_{0\neq \varphi \in \mathrm{H}^{\frac{1}{2}}(\partial\Omega )} \dfrac{|\langle\partial_\nu u,\varphi \rangle_{\partial \Omega}|}{\Vert \varphi \Vert_{\mathrm{H}^{\frac{1}{2}}(\partial\Omega )}}
 \\ \nonumber &= \sup_{0\neq \varphi \in \mathrm{H}^{\frac{1}{2}}(\partial\Omega )} \dfrac{|\langle\nabla_\mathrm{y} u(\mathrm{y}) \cdot \nu_\mathrm{y},\varphi \rangle_{\partial \Omega}|}{\Vert \varphi \Vert_{\mathrm{H}^{\frac{1}{2}}(\partial\Omega )}}
 \\ \nonumber &= \sup_{0\neq \varphi \in \mathrm{H}^{\frac{1}{2}}(\partial\Omega )} \dfrac{\delta^2\Big|\delta^{-1}\Big\langle\nabla_\mathrm{\eta} u(\delta\eta+\mathrm{z}) \cdot \nu_\mathrm{\eta},\varphi(\delta\eta+\mathrm{z}) \Big\rangle_{\partial \mathrm{B}}\big|}{\delta^\frac{1}{2}\Vert \varphi(\delta\eta+\mathrm{z}) \Vert_{\mathrm{H}^{\frac{1}{2}}(\partial\mathrm{B} )}}
 \\ \nonumber &\ge \nonumber \sup_{0\neq \varphi \in \mathrm{H}^{\frac{1}{2}}(\partial\mathrm{B} )} \dfrac{\delta|\langle \partial_\nu\hat{u},\hat{\varphi} \rangle_{\partial \Omega}|}{\underbrace{\delta^\frac{1}{2}\Vert \hat{\varphi} \Vert_{\mathrm{H}^{\frac{1}{2}}(\partial\mathrm{B} )}}_{(\ref{l1})}}
 \\ &= \delta^\frac{1}{2}\Vert \partial_\nu\hat{u} \Vert_{\mathrm{H}^{-\frac{1}{2}}(\partial\mathrm{B})}.
\end{align}
Therefore, using the above two estimates we obtain the desired inequality (\ref{l3}).
\end{proof}
\bigbreak
\noindent
We state the following scaling result.
\begin{lemma}\cite{sini}
Suppose $0<\delta\le1$. For $u \in \mathrm{H}^\mathrm{r}_{0,\sigma}(0,\mathrm{T};\mathrm{L}^2(\Omega))$ with non-negative integer r, we have
\begin{align}\label{l7}
    \Vert u \Vert_{\mathrm{H}^\mathrm{r}_{0,\sigma}(0,\mathrm{T};\mathrm{L}^2(\Omega))} = \delta^2 \Vert \hat{u} \Vert_{\mathrm{H}^\mathrm{r}_{0,\delta\sigma}(0,\mathrm{T}_\delta;\mathrm{L}^2(\mathrm{B}))}.
\end{align}
\end{lemma}
\noindent
We also know that
\begin{align}\label{l11}
    \Vert \nabla u \Vert_{\mathbb{L}^2(\Omega)} = \delta^\frac{1}{2}\Vert \nabla \hat{u} \Vert_{\mathbb{L}^2(\mathrm{B})} \quad \text{and} \quad  \Vert \nabla u \Vert_{\mathbb{L}^2(\Omega)} = \delta^\frac{3}{2}\Vert \widehat{\nabla u} \Vert_{\mathbb{L}^2(\mathrm{B})}.
\end{align}
Moreover, from the above two estimate we obtain
\begin{align}\label{l8}
    \Vert \widehat{\nabla u} \Vert_{\mathbb{L}^2(\mathrm{B})} = \delta^{-1} \Vert \nabla \hat{u} \Vert_{\mathbb{L}^2(\mathrm{B})}.
\end{align}
\noindent
Then we have the following lemma.
\begin{lemma}\label{ll9}
Suppose $0<\delta\le1$. For $\nabla u \in \mathrm{H}^\mathrm{r}_{0,\sigma}(0,\mathrm{T};\mathbb{L}^2(\Omega))$ with non-negative integer r, we have 
\begin{align}\label{l9}
 \Vert\nabla u \Vert_{\mathrm{H}_{0,\sigma}^\mathrm{r}\big(0,\mathrm{T};\mathbb{L}^2(\Omega)\big)}
 = \delta\Vert \nabla\hat{u} \Vert_{\mathrm{H}_{0,\delta\sigma}^\mathrm{r}\big(0,\mathrm{T}_\delta;\mathbb{L}^2(\mathrm{B})\big)}.
\end{align}
\end{lemma}
\begin{proof}
Due to (\ref{ss2}), (\ref{l11}) and (\ref{l8}), we can easily derive that
\begin{align}
    \Vert\nabla u \Vert_{\mathrm{H}_{0,\sigma}^\mathrm{r}\big(0,\mathrm{T};\mathbb{L}^2(\Omega)\big)}^2 &= \nonumber \int_{0}^\mathrm{T}\mathrm{e}^{-2\sigma \mathrm{t}}\sum_{\mathrm{k}=0}^\mathrm{r} \mathrm{T}^{2\mathrm{k}}\Big\Vert \frac{\partial^\mathrm{k}\nabla u(\cdot,\mathrm{t})}{\partial\mathrm{t}^\mathrm{k}}\Big\Vert^2_{\mathbb{L}^2(\Omega)}d\mathrm{t}
    \\ \nonumber &\le \int_{0}^\mathrm{T_\delta}\mathrm{e}^{-2\sigma \delta\mathrm{\tau}}\sum_{\mathrm{k}=0}^\mathrm{r} \cancel{\delta^{2\mathrm{k}}}\mathrm{T}_\delta^{2\mathrm{k}}\underbrace{\cancel{\delta^{-2\mathrm{k}}} \delta^3\Big\Vert \frac{\partial^\mathrm{k}\widehat{\nabla u}(\cdot,\mathrm{\tau})}{\partial\mathrm{\tau}^\mathrm{k}}\Big\Vert^2_{\mathbb{L}^2(\mathrm{B})}}_{\text{Due to}\ (\ref{ss2})\ \text{and} \ (\ref{l11})}\delta d\mathrm{\tau}
    \\ \nonumber &\le \delta^4\int_{0}^\mathrm{T_\delta}\mathrm{e}^{-2\sigma \delta\mathrm{\tau}}\sum_{\mathrm{k}=0}^\mathrm{r}\mathrm{T}_\delta^{2\mathrm{k}}\underbrace{\delta^{-2}\Big\Vert \frac{\partial^\mathrm{k}\nabla \widehat{u}(\cdot,\mathrm{\tau})}{\partial\mathrm{\tau}^\mathrm{k}}\Big\Vert^2_{\mathbb{L}^2(\mathrm{B})}}_{(\ref{l8})}d\mathrm{\tau}
    \\ &= \nonumber\delta^2\Vert \nabla\hat{u} \Vert_{\mathrm{H}_{0,\delta\sigma}^\mathrm{r}\big(0,\mathrm{T}_\delta;\mathbb{L}^2(\mathrm{B})\big)}^2.
\end{align}
The proof is thus complete.
\end{proof}
\noindent
We also observe that
\begin{align}\label{l12}
    \Vert \Delta u \Vert_{\mathrm{L}^2(\Omega)} = \delta^{-\frac{1}{2}}\Vert \Delta \hat{u} \Vert_{\mathrm{L}^2(\mathrm{B})} \quad \text{and} \quad  \Vert \Delta u \Vert_{\mathrm{L}^2(\Omega)} = \delta^\frac{3}{2}\Vert \widehat{\Delta u} \Vert_{\mathrm{L}^2(\mathrm{B})}.
\end{align}
Moreover, from the above two estimate we deduce
\begin{align}\label{l10}
    \Vert \widehat{\Delta u} \Vert_{\mathrm{L}^2(\mathrm{B})} = \delta^{-2} \Vert \Delta \hat{u} \Vert_{\mathrm{L}^2(\mathrm{B})}.
\end{align}
So, we state the following lemma.
\begin{lemma}\label{ll10}
Suppose $0<\delta\le1$. For $\Delta u \in \mathrm{H}^\mathrm{r}_{0,\sigma}(0,\mathrm{T};\mathrm{L}^2(\Omega))$ with non-negative integer r, we have 
\begin{align}\label{l13}
 \Vert\Delta u \Vert_{\mathrm{H}_{0,\sigma}^\mathrm{r}\big(0,\mathrm{T};\mathrm{L}^2(\Omega)\big)}
 = \Vert \Delta\hat{u} \Vert_{\mathrm{H}_{0,\delta\sigma}^\mathrm{r}\big(0,\mathrm{T}_\delta;\mathrm{L}^2(\mathrm{B})\big)}.
\end{align}
\end{lemma}
\begin{proof} Using (\ref{ss2}), (\ref{l12}) and (\ref{l10}), we verify that 
\begin{align}
    \Vert\Delta u \Vert_{\mathrm{H}_{0,\sigma}^\mathrm{r}\big(0,\mathrm{T};\mathrm{L}^2(\Omega)\big)}^2 &= \nonumber \int_{0}^\mathrm{T}\mathrm{e}^{-2\sigma \mathrm{t}}\sum_{\mathrm{k}=0}^\mathrm{r} \mathrm{T}^{2\mathrm{k}}\Big\Vert \frac{\partial^\mathrm{k}\Delta u(\cdot,\mathrm{t})}{\partial\mathrm{t}^\mathrm{k}}\Big\Vert^2_{\mathrm{L}^2(\Omega)}d\mathrm{t}
    \\ \nonumber &\le \int_{0}^\mathrm{T_\delta}\mathrm{e}^{-2\sigma \delta\mathrm{\tau}}\sum_{\mathrm{k}=0}^\mathrm{r} \cancel{\delta^{2\mathrm{k}}}\mathrm{T}_\delta^{2\mathrm{k}}\underbrace{\cancel{\delta^{-2\mathrm{k}}} \delta^3\Big\Vert \frac{\partial^\mathrm{k}\widehat{\Delta u}(\cdot,\mathrm{\tau})}{\partial\mathrm{\tau}^\mathrm{k}}\Big\Vert^2_{\mathrm{L}^2(\mathrm{B})}}_{\text{Due to}\ (\ref{ss2})\ \text{and} \ (\ref{l12})}\delta d\mathrm{\tau}
    \\ \nonumber &\le \delta^4\int_{0}^\mathrm{T_\delta}\mathrm{e}^{-2\sigma \delta\mathrm{\tau}}\sum_{\mathrm{k}=0}^\mathrm{r}\mathrm{T}_\delta^{2\mathrm{k}}\underbrace{\delta^{-4}\Big\Vert \frac{\partial^\mathrm{k}\Delta \widehat{u}(\cdot,\mathrm{\tau})}{\partial\mathrm{\tau}^\mathrm{k}}\Big\Vert^2_{\mathrm{L}^2(\mathrm{B})}}_{(\ref{l10})}d\mathrm{\tau}
    \\ &= \nonumber\Vert \Delta\hat{u} \Vert_{\mathrm{H}_{0,\delta\sigma}^\mathrm{r}(0,\mathrm{T}_\delta;\mathrm{L}^2(\mathrm{B}))}^2.
\end{align}
Hence, this completes the proof.
\end{proof}
\noindent
Now, using Proposition \ref{p1}, Proposition \ref{p3}, Lemma \ref{ll9}, and Lemma \ref{ll10}, we derive the following estimate 
\begin{align}
 \Vert \partial_\nu u(\cdot,\mathrm{t})\Vert_{\mathrm{H}^{-\frac{1}{2}}(\partial\Omega)}
 &\lesssim \nonumber   \Vert\partial_\nu \hat{u} \Vert_{\mathrm{H}_{0,\delta\sigma}^\mathrm{r}(0,\mathrm{T}_\delta;\mathrm{H}^{-\frac{1}{2}}(\partial\mathrm{B}))}
 \\ &\lesssim \nonumber \Vert \nabla \widehat{u} \Vert_{\mathrm{H}_{0,\delta\sigma}^\mathrm{r}\big(0,\mathrm{T}_\delta;\mathbb{H}{(\text{div},\mathrm{B})}\big)}
 \\ &\lesssim \nonumber  \Big( \Vert \nabla \widehat{u} \Vert_{\mathrm{H}_{0,\delta\sigma}^\mathrm{r}(0,\mathrm{T}_\delta;\mathbb{L}^2(\mathrm{B}))}^2 + \Vert \Delta \widehat{u} \Vert_{\mathrm{H}_{0,\delta\sigma}^\mathrm{r}(0,\mathrm{T}_\delta;\mathrm{L}^2(\mathrm{B}))}^2\Big)^\frac{1}{2}
  \\ &\lesssim \nonumber  \Big(\underbrace{\delta^{-2} \Vert \nabla u \Vert_{\mathrm{H}_{0,\sigma}^\mathrm{r}(0,\mathrm{T};\mathbb{L}^2(\Omega))}^2}_{(\ref{l9})} +  \underbrace{\Vert \Delta u \Vert_{\mathrm{H}_{0,\sigma}^\mathrm{r}(0,\mathrm{T};\mathrm{L}^2(\Omega))}^2}_{(\ref{l13})}\Big)^\frac{1}{2}
  \\ &=  \Big( \delta^{-2} \cdot \delta^{5} +  \delta^3\Big)^\frac{1}{2} \sim \delta^{\textcolor{black}{\frac{3}{2}}}.
\end{align}
Therefore, we deduce that
\begin{align}\label{p33}
  \Vert \partial_\nu u(\cdot,\mathrm{t})\Vert_{\mathrm{H}^{-\frac{1}{2}}(\partial\Omega)} \sim \delta^{\textcolor{black}{\frac{3}{2}}}.
\end{align}
In order to prove Proposition \ref{p3}, we need to improve the above estimate. For this, let us go back to the original boundary integral equation
\begin{align}
\nonumber
(1+\frac{\alpha\rho_\mathrm{m}}{2})\partial_\nu u +\gamma \partial_\nu\bm{\mathrm{V}}_\Omega\Big[u_{\mathrm{t}\mathrm{t}}\Big] + \alpha  \bm{\mathcal{K}}^\mathrm{t}\Big[\partial_\nu u\Big] = \partial_\nu u^\mathrm{i},
\end{align}
that we rewrite as
\begin{align}\label{apr1}
(\frac{1}{\alpha}+\frac{\rho_\mathrm{m}}{2})\partial_\nu u + \frac{\gamma}{\alpha} \partial_\nu\bm{\mathrm{V}}_\Omega\Big[u_{\mathrm{t}\mathrm{t}}\Big] + \bm{\mathcal{K}}^\mathrm{t}\Big[\partial_\nu u\Big] = \frac{1}{\alpha}\partial_\nu u^\mathrm{i}.
\end{align}
Using Taylor's series expansion we get
\begin{align}\label{ap1}
  \bm{\mathcal{K}}^\mathrm{t}\Big[\partial_\nu u\Big] 
    &=\nonumber\rho_\mathrm{m}\Bigg[ -\cancel{\mathrm{c}^{-1}_0\int_{\partial\Omega} \partial_\mathrm{t}\partial_\nu u(\mathrm{y}, \mathrm{t}) \frac{(\mathrm{x}-\mathrm{y}).\nu_\mathrm{x}}{|\mathrm{x}-\mathrm{y}|^2}d\sigma_\mathrm{x}d\sigma_\mathrm{y}}
    + \mathrm{c}^{-2}_0\int_{\partial\Omega}\partial_\mathrm{t}^2\partial_\nu u(\mathrm{y}, \mathrm{t}) \frac{(\mathrm{x}-\mathrm{y}).\nu_\mathrm{x}}{|\mathrm{x}-\mathrm{y}|}d\sigma_\mathrm{y} 
    \\ \nonumber &- \mathrm{c}^{-3}_0\int_{\partial\Omega}\partial_\mathrm{t}^3\partial_\nu u(\mathrm{y}, \mathrm{t}_1) (\mathrm{x}-\mathrm{y})\cdot\nu_\mathrm{x} d\sigma_\mathrm{x} 
    - \int_{\partial\Omega}\partial_\nu u(\mathrm{y}, \mathrm{t}) \frac{(\mathrm{x}-\mathrm{y})\cdot\nu_\mathrm{x}}{|\mathrm{x}-\mathrm{y}|^3}d\sigma_\mathrm{y}
    \\ \nonumber &+ \cancel{\mathrm{c}^{-1}_0\int_{\partial\Omega} \partial_\mathrm{t}\partial_\nu u(\mathrm{y}, \mathrm{t}) \frac{(\mathrm{x}-\mathrm{y}).\nu_\mathrm{x}}{|\mathrm{x}-\mathrm{y}|^2}d\sigma_\mathrm{x}d\sigma_\mathrm{y}} 
    - \frac{1}{2} \mathrm{c}^{-2}_0\int_{\partial\Omega}\partial_\mathrm{t}^2\partial_\nu u(\mathrm{y}, \mathrm{t}) \frac{(\mathrm{x}-\mathrm{y})\cdot\nu_\mathrm{x}}{|\mathrm{x}-\mathrm{y}|}d\sigma_\mathrm{y}  
    \\ &+ \frac{1}{3!}\mathrm{c}^{-3}_0\int_{\partial\Omega}\partial_\mathrm{t}^3\partial_\nu u(\mathrm{y}, \mathrm{t}_2) (\mathrm{x}-\mathrm{y})\cdot\nu_\mathrm{x} d\sigma_\mathrm{y}\Bigg],
\end{align}
and
\begin{align}\label{ap2}
\partial_\nu\bm{\mathrm{V}}_\Omega\Big[u_{\mathrm{t}\mathrm{t}}\Big] &:= \nonumber \partial_\nu\int_\Omega\frac{\rho_\mathrm{m}}{4\pi|\mathrm{x}-\mathrm{y}|}\partial_\mathrm{t}^2u(\mathrm{y},\mathrm{t}-\mathrm{c}^{-1}_0|\mathrm{x}-\mathrm{y}|)d\mathrm{y}
\\ &= \rho_\mathrm{m}\Bigg[\partial_\nu\int_\Omega\frac{\partial_\mathrm{t}^2u(\mathrm{y},\mathrm{t})}{4\pi|\mathrm{x}-\mathrm{y}|}d\mathrm{y} - \underbrace{\partial_\nu\int_\Omega\frac{\partial_\mathrm{t}^{\textcolor{black}{3}}u(\mathrm{y},\mathrm{t})}{4\pi|\mathrm{x}-\mathrm{y}|}\mathrm{c}^{-1}_0|\mathrm{x}-\mathrm{y}|d\mathrm{y}}_{=0} + \partial_\nu\int_\Omega\frac{\partial_\mathrm{t}^{\textcolor{black}{4}}u(\mathrm{y},\mathrm{t}_3)}{4\pi|\mathrm{x}-\mathrm{y}|}\mathrm{c}^{-2}_0|\mathrm{x}-\mathrm{y}|^2d\mathrm{y}\Bigg],
\end{align}
where $\mathrm{t}_1, \mathrm{t}_2, \mathrm{t}_3 \in \big(\mathrm{t}-\mathrm{c}^{-1}_0|\mathrm{x}-\mathrm{y}|,\mathrm{t}\big).$ 
\newline

\noindent
Now, let us denote $\mathrm{t}_\mathrm{m} := \max(\mathrm{t}_1, \mathrm{t}_2)$. Consequently, we approximate the stated equation (\ref{apr1}) using (\ref{ap1}) and (\ref{ap2}) as follows:
\begin{align}\label{important}
   \nonumber(\frac{1}{\alpha}+\frac{\rho_\mathrm{m}}{2})\partial_\nu u - \rho_\mathrm{m} \int_{\partial\Omega}\partial_\nu u(\mathrm{y}, \mathrm{t}) \frac{(\mathrm{x}-\mathrm{y})\cdot\nu_\mathrm{x}}{|\mathrm{x}-\mathrm{y}|^3}d\sigma_\mathrm{y} 
    &= \frac{1}{\alpha}\partial_\nu u^\mathrm{i} 
    -\rho_\mathrm{m}\Bigg[ \frac{1}{2} \mathrm{c}^{-2}_0\int_{\partial\Omega}\partial_\mathrm{t}^2\partial_\nu u(\mathrm{y}, \mathrm{t}) \frac{(\mathrm{x}-\mathrm{y})\cdot\nu_\mathrm{x}}{|\mathrm{x}-\mathrm{y}|}d\sigma_\mathrm{y}
    \\ \nonumber&+ \frac{5}{6}\mathrm{c}^{-3}_0\int_{\partial\Omega}\partial_\mathrm{t}^3\partial_\nu u(\mathrm{y}, \mathrm{t}_\mathrm{m} ) (\mathrm{x}-\mathrm{y})\cdot\nu_\mathrm{x} d\sigma_\mathrm{y}
    \\ &+ \frac{\gamma}{\alpha}\partial_\nu\int_\Omega\frac{\partial_\mathrm{t}^2u(\mathrm{y},\mathrm{t})}{4\pi|\mathrm{x}-\mathrm{y}|}d\mathrm{y} + \mathrm{c}^{-2}_0\frac{\gamma}{\alpha} \partial_\nu\int_\Omega|\mathrm{x}-\mathrm{y}|\partial_\mathrm{t}^4u(\mathrm{y},\mathrm{t}_3)d\mathrm{y}\Bigg].
\end{align}
Before proceeding to the next step, we introduce the notations
\begin{align}
    \partial_\nu\bm{\mathcal{N}}_{\textbf{Lap},\Omega}\Big[\partial_\mathrm{t}^2u\Big] :=  \partial_\nu\int_\Omega\frac{\partial_\mathrm{t}^2u(\mathrm{y},\mathrm{t})}{4\pi|\mathrm{x}-\mathrm{y}|}d\mathrm{y},\quad \displaystyle \mathrm{A}(\mathrm{y})= \int_{\partial\Omega}\frac{(\mathrm{x}-\mathrm{y})\cdot\nu}{|\mathrm{x}-\mathrm{y}|}d\sigma_\mathrm{x}, \quad     \mathrm{A}_{\partial\Omega} := \frac{1}{|\partial \Omega|}\int_{\partial\Omega}\int_{\partial\Omega}\frac{(\mathrm{x}-\mathrm{y}).\nu}{|\mathrm{x}-\mathrm{y}|}d\sigma_\mathrm{x}d\sigma_\mathrm{y}.
\end{align}
We also recall $\displaystyle\bm{\mathcal{K}}^*_\textbf{Lap}[\partial_\nu u]:= -\rho_\mathrm{m} \int_{\partial\Omega} \frac{(\mathrm{x}-\mathrm{y})\cdot\nu_\mathrm{x}}{|\mathrm{x}-\mathrm{y}|^3}\partial_\nu u(\mathrm{y}, \mathrm{t})d\sigma_\mathrm{y} \quad \text{and}\quad \bm{\mathrm{T}}^* := \Big[(\frac{1}{\alpha}+\frac{\rho_\mathrm{m}}{2}) +\bm{\mathcal{K}}^*_\textbf{Lap} \Big]^{-1} .$
Then we obtain the following equation
\begin{align}
    &\nonumber\Big((\frac{1}{\alpha}+\frac{\rho_\mathrm{m}}{2}) + \bm{\mathcal{K}}^*_\textbf{Lap}\Big)\Big[\partial_\nu u \Big]
   =\nonumber \frac{1}{\alpha}\Big[\partial_\nu u^\mathrm{i} - \frac{1}{|\partial\Omega|}\int_{\partial\Omega}\partial_\nu u^\mathrm{i}\Big] + \frac{1}{|\partial\Omega|}\frac{1}{\alpha}\int_{\partial\Omega}\partial_\nu u^\mathrm{i}
     - \frac{\rho_\mathrm{m}}{2} \mathrm{c}^{-2}_0\Big[\int_{\partial\Omega}\partial_\mathrm{t}^2\partial_\nu u \frac{(\mathrm{x}-\mathrm{y})\cdot\nu_\mathrm{x}}{|\mathrm{x}-\mathrm{y}|}d\sigma_\mathrm{y} 
    \\ \nonumber &- \frac{1}{|\partial\Omega|}\int_{\partial\Omega}\int_{\partial\Omega}\partial_\mathrm{t}^2\partial_\nu u \frac{(\mathrm{x}-\mathrm{y})\cdot\nu_\mathrm{x}}{|\mathrm{x}-\mathrm{y}|}d\sigma_\mathrm{y}d\sigma_\mathrm{x}\Big] - \frac{\rho_\mathrm{m}}{|\partial\Omega|}\int_{\partial\Omega}\Big(\mathrm{A}(\mathrm{y})-\mathrm{A}_{\partial\Omega}\Big)\partial_\mathrm{t}^2\partial_\nu u(\mathrm{y},\mathrm{t})d\sigma_\mathrm{y}
    \\&- \nonumber \frac{\rho_\mathrm{m}}{|\partial\Omega|}\int_{\partial\Omega}\mathrm{A}_{\partial\Omega}\partial_\mathrm{t}^2\partial_\nu u(\mathrm{y},\mathrm{t})d\sigma_\mathrm{y}
   - \frac{5\rho_\mathrm{m}}{6}\mathrm{c}^{-3}_0\Big[\int_{\partial\Omega}\partial_\mathrm{t}^3\partial_\nu u(\mathrm{y}, \mathrm{t}_\mathrm{m} ) (\mathrm{x}-\mathrm{y})\cdot\nu_\mathrm{x} d\sigma_\mathrm{y} 
     \\ &- \nonumber\frac{1}{|\partial\Omega|}\int_{\partial\Omega} \int_{\partial\Omega}\partial_\mathrm{t}^3\partial_\nu u(\mathrm{y}, \mathrm{t}_\mathrm{m} ) (\mathrm{x}-\mathrm{y})\cdot\nu_\mathrm{x} d\sigma_\mathrm{y}d\sigma_\mathrm{x}\Big]
    -\frac{5\rho_\mathrm{m}}{6}\mathrm{c}^{-3}_0\frac{1}{|\partial\Omega|}\int_{\partial\Omega} \int_{\partial\Omega}\partial_\mathrm{t}^3\partial_\nu u(\mathrm{y}, \mathrm{t}_\mathrm{m} ) (\mathrm{x}-\mathrm{y})\cdot\nu_\mathrm{x} d\sigma_\mathrm{y}d\sigma_\mathrm{x}
    \\ \nonumber &- \frac{\gamma}{\alpha}\Big[\partial_\nu\bm{\mathcal{N}}_{\textbf{Lap},\Omega}[\partial_\mathrm{t}^2u]- \frac{1}{|\partial\Omega|}\int_{\partial\Omega}\partial_\nu\bm{\mathcal{N}}_{\textbf{Lap},\Omega}[\partial_\mathrm{t}^2u]\Big] -\nonumber \frac{\gamma}{\alpha}\frac{1}{|\partial\Omega|}\int_{\partial\Omega}\partial_\nu\bm{\mathcal{N}}_{\textbf{Lap},\Omega}[\partial_\mathrm{t}^2u] - \mathrm{c}^{-2}_0\frac{\gamma}{\alpha} \Big[\partial_\nu\int_\Omega|\mathrm{x}-\mathrm{y}|\partial_\mathrm{t}^4u(\mathrm{y},\mathrm{t}_3)d\mathrm{y}
    \\ &- \frac{1}{|\partial\Omega|} \int_{\partial\Omega}\partial_\nu\int_\Omega|\mathrm{x}-\mathrm{y}|\partial_\mathrm{t}^4u(\mathrm{y},\mathrm{t}_3)d\mathrm{y}\Big]- \mathrm{c}^{-2}_0\frac{\gamma\rho_\mathrm{m}}{\alpha}\frac{1}{|\partial\Omega|} \int_{\partial\Omega}\partial_\nu\int_\Omega|\mathrm{x}-\mathrm{y}|\partial_\mathrm{t}^4u(\mathrm{y},\mathrm{t}_3)d\mathrm{y}.
\end{align}
Now, considering $\rho_\mathrm{m}\sim 1$, we take the $\mathrm{H}^{-\frac{1}{2}}(\partial\Omega)$-norm on the both hand sides of the above equation to obtain
\begin{align}\label{nu}
    & \nonumber\Vert \partial_\nu u \Vert_{\mathrm{H}^{-\frac{1}{2}}(\partial\Omega)} 
    \\ &\nonumber\lesssim \frac{1}{\alpha}\Vert \bm{\mathrm{T}}^*\Vert_{\mathcal{L}(\mathrm{H}^{-\frac{1}{2}}_0(\partial\Omega))}\Vert \partial_\nu u^\mathrm{i} - \frac{1}{|\partial\Omega|}\int_{\partial\Omega}\partial_\nu u^\mathrm{i} \Vert_{\mathrm{H}^{-\frac{1}{2}}_0(\partial\Omega)} + \frac{1}{\alpha}\frac{1}{|\partial\Omega|}\Vert\int_{\partial\Omega}\partial_\nu u^\mathrm{i}\Vert_{\mathrm{H}^{-\frac{1}{2}}(\partial\Omega)} \Vert\bm{\mathrm{T}}^*\Vert_{\mathcal{L}(\mathrm{H}^{-\frac{1}{2}}(\partial\Omega))}
    \\ \nonumber &+ \Vert \bm{\mathrm{T}}^*\Vert_{\mathcal{L}(\mathrm{H}^{-\frac{1}{2}}_0(\partial\Omega))}\Vert \int_{\partial\Omega}(\partial_\nu u)_{\mathrm{t}\mathrm{t}} \frac{(\mathrm{x}-\mathrm{y})\cdot\nu_\mathrm{x}}{|\mathrm{x}-\mathrm{y}|}d\sigma_\mathrm{y} 
    -\frac{1}{|\partial\Omega|}\int_{\partial\Omega}\int_{\partial\Omega}(\partial_\nu u)_{\mathrm{t}\mathrm{t}} \frac{(\mathrm{x}-\mathrm{y})\cdot\nu_\mathrm{x}}{|\mathrm{x}-\mathrm{y}|}d\sigma_\mathrm{y}d\sigma_\mathrm{x}\Vert_{\mathrm{H}^{-\frac{1}{2}}_0(\partial\Omega)}
    \\ \nonumber &+  \frac{1}{|\partial\Omega|}\Vert \int_{\partial\Omega}(\mathrm{A}(\cdot)-\mathrm{A}_{\partial\Omega})\partial_\mathrm{t}^2\partial_\nu u \Vert_{\mathrm{H}_0^{-\frac{1}{2}}(\partial\Omega)}\Vert\bm{\mathrm{T}}^*\Vert_{\mathcal{L}(\mathrm{H}_0^{-\frac{1}{2}}(\partial\Omega))}
    + \frac{1}{|\partial\Omega|}\Vert\int_{\partial\Omega}\partial_\mathrm{t}^2\partial_\nu u(\cdot,\mathrm{t}) \mathrm{A}_{\partial\Omega}\Vert_{\mathrm{H}^{-\frac{1}{2}}(\partial\Omega)}\Vert\bm{\mathrm{T}}^*\Vert_{\mathcal{L}(\mathrm{H}^{-\frac{1}{2}}(\partial\Omega))}
    \\ \nonumber &+ \Vert \bm{\mathrm{T}}^*\Vert_{\mathcal{L}(\mathrm{H}^{-\frac{1}{2}}_0(\partial\Omega))} \Vert \int_{\partial\Omega}(\partial_\nu u)_{\mathrm{t}\mathrm{t}\mathrm{t}}(\mathrm{x}-\mathrm{y})\cdot\nu_\mathrm{x} d\sigma_\mathrm{y}
    -\frac{1}{|\partial\Omega|}\int_{\partial\Omega} \int_{\partial\Omega}(\partial_\nu u)_{\mathrm{t}\mathrm{t}\mathrm{t}} (\mathrm{x}-\mathrm{y})\cdot\nu_\mathrm{x} d\sigma_\mathrm{y}d\sigma_\mathrm{x}\Vert_{\mathrm{H}^{-\frac{1}{2}}_0(\partial\Omega)}
    \\\nonumber &+ \frac{1}{|\partial\Omega|}\Vert \int_{\partial\Omega} \int_{\partial\Omega}\partial_\mathrm{t}^3\partial_\nu u (\mathrm{x}-\mathrm{y})\cdot\nu_\mathrm{x} d\sigma_\mathrm{y}d\sigma_\mathrm{x}\Vert_{\mathrm{H}^{-\frac{1}{2}}(\partial\Omega)}\Vert\bm{\mathrm{T}}^*\Vert_{\mathcal{L}(\mathrm{H}^{-\frac{1}{2}}(\partial\Omega))} \\ \nonumber&+ \frac{\gamma}{\alpha} \frac{1}{|\partial\Omega|}\Vert \int_{\partial\Omega}\partial_\nu\bm{\mathcal{N}}_{\textbf{Lap},\Omega}[\partial_\mathrm{t}^2u]\Vert_{\mathrm{H}^{-\frac{1}{2}}(\partial\Omega)}\Vert\bm{\mathrm{T}}^*\Vert_{\mathcal{L}(\mathrm{H}^{-\frac{1}{2}}(\partial\Omega))}
    \\ \nonumber &+ \frac{\gamma}{\alpha} \Vert \partial_\nu\bm{\mathcal{N}}_{\textbf{Lap},\Omega}[\partial_\mathrm{t}^2u]- \frac{1}{|\partial\Omega|}\int_{\partial\Omega}\partial_\nu\bm{\mathcal{N}}_{\textbf{Lap},\Omega}[\partial_\mathrm{t}^2u]\Vert_{\mathrm{H}_0^{-\frac{1}{2}}(\partial\Omega)}\Vert\bm{\mathrm{T}}^*\Vert_{\mathcal{L}(\mathrm{H}_0^{-\frac{1}{2}}(\partial\Omega))} 
    \\ \nonumber&+ \mathrm{c}^{-2}_0\frac{\gamma}{\alpha} \frac{1}{|\partial\Omega|}\Vert \int_{\partial\Omega}\partial_\nu\int_\Omega|\mathrm{x}-\mathrm{y}|\partial_\mathrm{t}^4u(\mathrm{y},\mathrm{t}_3)d\mathrm{y}\Vert_{\mathrm{H}^{-\frac{1}{2}}(\partial\Omega)}\Vert\bm{\mathrm{T}}^*\Vert_{\mathcal{L}(\mathrm{H}^{-\frac{1}{2}}(\partial\Omega))}
    \\ &+ \mathrm{c}^{-2}_0\frac{\gamma}{\alpha}\Vert \partial_\nu\int_\Omega|\mathrm{x}-\mathrm{y}|\partial_\mathrm{t}^4u(\mathrm{y},\mathrm{t}_3)d\mathrm{y} - \frac{1}{|\partial\Omega|} \int_{\partial\Omega}\partial_\nu\int_\Omega|\mathrm{x}-\mathrm{y}|\partial_\mathrm{t}^4u(\mathrm{y},\mathrm{t}_3)d\mathrm{y}\Vert_{\mathrm{H}_0^{-\frac{1}{2}}(\partial\Omega)}\Vert\bm{\mathrm{T}}^*\Vert_{\mathcal{L}(\mathrm{H}_0^{-\frac{1}{2}}(\partial\Omega))}.
\end{align}
We have the following bounds: (for the first two we refer to \cite{ammari, ammari1}) and the third is from (\ref{p33})
\begin{align}\label{bounds}
    \Vert \bm{\mathrm{T}}^*\Vert_{\mathcal{L}(\mathrm{H}^{-\frac{1}{2}}_0(\partial\Omega))} \sim 1\quad, \quad
    \Vert \bm{\mathrm{T}}^*\Vert_{\mathcal{L}(\mathrm{H}^{-\frac{1}{2}}(\partial\Omega))} \sim \alpha, \quad \text{and}\ \Vert \partial_\mathrm{t}^\mathrm{k}\partial_\nu u(\cdot, \mathrm{t})\Vert_{\mathrm{H}^{-\frac{1}{2}}(\partial\Omega)} \sim \delta \quad \text{for}\quad \mathrm{k} =0,1,....
\end{align}
Next, we define and estimate each term of the expression (\ref{nu}). In particular, we want to emphasize that we will estimate each term using (\ref{bounds}) and $\alpha \sim \delta^{-2}$. 
\bigbreak
\noindent
We then use the triangle inequality, smoothness of $\partial_\nu u^\mathrm{i}$ to estimate $\mathrm{S}_1$. 
\begin{align}\label{s1}
\bullet\ \text{Estimation of}\ \mathrm{S}_1 &:=\nonumber \frac{1}{\alpha}\Vert \bm{\mathrm{T}}^*\Vert_{\mathcal{L}(\mathrm{H}^{-\frac{1}{2}}_0(\partial\Omega))}\Big\Vert \partial_\nu u^\mathrm{i} - \frac{1}{|\partial\Omega|}\int_{\partial\Omega}\partial_\nu u^\mathrm{i}\Big\Vert_{\mathrm{H}^{-\frac{1}{2}}_0(\partial\Omega)} 
\\ &\lesssim \frac{1}{\alpha} \Vert \partial_\nu u^\mathrm{i}\Vert_{\mathrm{H}^{-\frac{1}{2}}_0(\partial\Omega)} \lesssim \frac{1}{\alpha}\Vert 1 \Vert_{\mathrm{L}^2(\partial\Omega)} \sim \delta^3.
\end{align}
Next, as $\displaystyle\Big|\int_{\partial\Omega}\partial_\nu u^\mathrm{i} \Big| = \Big|\int_{\partial\Omega}\Delta u^\mathrm{i} \Big| = \frac{\rho_\mathrm{m}}{\mathrm{k}_\mathrm{m}} \Big|\int_{\Omega}\partial_\mathrm{t}^2 u^\mathrm{i} \Big| = \mathcal{O}(\delta^3),$ we have
\begin{align}\label{s2}
\bullet\ \text{Estimation of}\ \mathrm{S}_2 &:=\nonumber \frac{1}{\alpha}\frac{1}{|\partial\Omega|}\Vert \bm{\mathrm{T}}^*\Vert_{\mathcal{L}(\mathrm{H}^{-\frac{1}{2}}(\partial\Omega))}\Big\Vert\int_{\partial\Omega}\partial_\nu u^\mathrm{i}\Big\Vert_{\mathrm{H}^{-\frac{1}{2}}(\partial\Omega)}
\\ &\lesssim \frac{1}{\alpha}\frac{1}{|\partial\Omega|}\Vert \bm{\mathrm{T}}^*\Vert_{\mathcal{L}(\mathrm{H}^{-\frac{1}{2}}(\partial\Omega))} \Vert 1 \Vert_{\mathrm{L}^2(\partial\Omega)} \Big|\int_{\partial\Omega}\partial_\nu u^\mathrm{i} \Big| \sim \delta^2.
\end{align}
Now, let us do the following calculation using norm-definition (\ref{normd1})
\begin{align}
    \Big\Vert \frac{(\mathrm{x}-\cdot)\cdot \nu_\mathrm{x}}{|\mathrm{x}-\cdot|} \Big\Vert_{\mathrm{H}^{\frac{1}{2}}_0(\partial\Omega)}^2 &=\nonumber \Big\Vert \frac{(\mathrm{x}-\cdot)\cdot \nu_\mathrm{x}}{|\mathrm{x}-\cdot|} \Big\Vert_{\mathrm{L}^2_0(\partial\Omega)}^2 + \int_{\partial\Omega}\int_{\partial\Omega}\dfrac{\Big(\frac{(\mathrm{x}-\mathrm{y})\cdot \nu_\mathrm{x}}{|\mathrm{x}-\mathrm{y}|}- \frac{(\mathrm{x}-\mathrm{z})\cdot \nu_\mathrm{x}}{|\mathrm{x}-\mathrm{z}|}\Big)^2}{|\mathrm{y}-\mathrm{z}|^3}d\sigma_\mathrm{y}d\sigma_\mathrm{z}
    \\ &= \nonumber \delta^2 \Big\Vert \frac{(\mathrm{\xi}-\cdot)\cdot \nu_\mathrm{\xi}}{|\mathrm{\xi}-\cdot|} \Big\Vert_{\mathrm{L}^2_0(\partial\mathrm{B})}^2 + \delta \int_{\partial\Omega}\int_{\partial\Omega}\dfrac{\Big(\frac{(\mathrm{\xi}-\mathrm{\eta})\cdot \nu_\mathrm{\xi}}{|\mathrm{\xi}-\mathrm{\eta}|}- \frac{(\mathrm{\xi}-\mathrm{\lambda})\cdot \nu_\mathrm{\xi}}{|\mathrm{\xi}-\mathrm{\lambda}|}\Big)^2}{|\mathrm{\eta}-\mathrm{\lambda}|^3}d\sigma_\mathrm{\eta}d\sigma_\mathrm{\lambda}
    \\ &= \nonumber\delta \Big\Vert \frac{(\mathrm{\xi}-\cdot)\cdot \nu_\mathrm{\xi}}{|\mathrm{\xi}-\cdot|} \Big\Vert_{\mathrm{H}^{\frac{1}{2}}_0(\partial\mathrm{B})}^2.
\end{align}
Moreover, as $\frac{(\mathrm{\xi}-\cdot)\cdot \nu_\mathrm{\xi}}{|\mathrm{\xi}-\cdot|} \sim 1$, we obtain
\begin{align}\label{es6}
    \Big\Vert \frac{(\mathrm{x}-\cdot)\cdot \nu_\mathrm{x}}{|\mathrm{x}-\cdot|} \Big\Vert_{\mathrm{H}^{\frac{1}{2}}_0(\partial\Omega)} \sim \delta^\frac{1}{2}.
\end{align}
Then using the estimate (\ref{es6}) we obtain
\begin{align}\label{s3}
\bullet\ \nonumber&\text{Estimation of}\ \mathrm{S}_3 \\ &:=\nonumber
\Vert \bm{\mathrm{T}}^*\Vert_{\mathcal{L}(\mathrm{H}^{-\frac{1}{2}}_0(\partial\Omega))} \Big\Vert \int_{\partial\Omega}\partial_\mathrm{t}^2\partial_\nu u \frac{(\mathrm{x}-\mathrm{y})\cdot\nu_\mathrm{x}}{|\mathrm{x}-\mathrm{y}|}d\sigma_\mathrm{y} 
-\frac{1}{|\partial\Omega|}\int_{\partial\Omega}\int_{\partial\Omega}\partial_\mathrm{t}^2\partial_\nu u \frac{(\mathrm{x}-\mathrm{y})\cdot\nu_\mathrm{x}}{|\mathrm{x}-\mathrm{y}|}d\sigma_\mathrm{y}d\sigma_\mathrm{x}\Big\Vert_{\mathrm{H}^{-\frac{1}{2}}_0(\partial\Omega)}
\\ \nonumber &\le \Big\Vert \int_{\partial\Omega}\partial_\mathrm{t}^2\partial_\nu u \frac{(\mathrm{x}-\mathrm{y})\cdot\nu_\mathrm{x}}{|\mathrm{x}-\mathrm{y}|}d\sigma_\mathrm{y}\Big\Vert_{\mathrm{H}^{-\frac{1}{2}}_0(\partial\Omega)} + \frac{1}{|\partial\Omega|} \Big\Vert\int_{\partial\Omega}\int_{\partial\Omega}\partial_\mathrm{t}^2\partial_\nu u \frac{(\mathrm{x}-\mathrm{y})\cdot\nu_\mathrm{x}}{|\mathrm{x}-\mathrm{y}|}d\sigma_\mathrm{y}d\sigma_\mathrm{x}\Big\Vert_{\mathrm{H}^{-\frac{1}{2}}_0(\partial\Omega)}
\\ &\lesssim \nonumber \Vert 1 \Vert_{\mathrm{L}^2(\partial\Omega)}  \Big|\int_{\partial\Omega}\partial_\mathrm{t}^2\partial_\nu u(\mathrm{y},\mathrm{t}) \frac{(\mathrm{x}-\mathrm{y})\cdot\nu_\mathrm{x}}{|\mathrm{x}-\mathrm{y}|}d\sigma_\mathrm{y}\Big| + \frac{1}{|\partial\Omega|} \Vert 1 \Vert_{\mathrm{L}^2(\partial\Omega)}\Big|\int_{\partial\Omega}\int_{\partial\Omega}\partial_\mathrm{t}^2\partial_\nu u \frac{(\mathrm{x}-\mathrm{y})\cdot\nu_\mathrm{x}}{|\mathrm{x}-\mathrm{y}|}d\sigma_\mathrm{y}d\sigma_\mathrm{x}\Big|
\\ &\lesssim \nonumber \delta \Vert \partial_\mathrm{t}^2\partial_\nu u(\cdot,\mathrm{t}) \Vert_{\mathrm{H}^{-\frac{1}{2}}_0(\partial\Omega)} \Big\Vert \frac{(\mathrm{x}-\cdot)\cdot \nu_\mathrm{x}}{|\mathrm{x}-\cdot|} \Big\Vert_{\mathrm{H}^{\frac{1}{2}}_0(\partial\Omega)} + \delta^{-1} \Vert 1 \Vert_{\mathrm{H}^{\frac{1}{2}}_0(\partial\Omega)} \Big\Vert \int_{\partial\Omega}\partial_\mathrm{t}^2\partial_\nu u \frac{(\mathrm{x}-\mathrm{y})\cdot\nu_\mathrm{x}}{|\mathrm{x}-\mathrm{y}|}d\sigma_\mathrm{y}\Big\Vert_{\mathrm{H}^{-\frac{1}{2}}_0(\partial\Omega)}
\\ &\sim \delta^{\textcolor{black}{3}}.
\end{align}
As, $\displaystyle\mathrm{A}_{\partial\Omega} = \frac{1}{|\partial\Omega|}\int_{\partial\Omega}\int_{\partial\Omega}\frac{(\mathrm{x}-\mathrm{y})\cdot\nu_\mathrm{x}}{|\mathrm{x}-\mathrm{y}|} d\sigma_\mathrm{x}d\sigma_\mathrm{y} \sim \delta^2
$, we deduce the following estimate
\begin{align}\label{s5}
\bullet\ \text{Estimation of}\ \mathrm{S}_5 &:=\nonumber \Vert\bm{\mathrm{T}}^*\Vert_{\mathcal{L}(\mathrm{H}^{-\frac{1}{2}}(\partial\Omega))} \frac{1}{|\partial\Omega|}\Big\Vert\int_{\partial\Omega}\partial_\mathrm{t}^2\partial_\nu u(\cdot,\mathrm{t}) \mathrm{A}_{\partial\Omega}\Big\Vert_{\mathrm{H}_0^{-\frac{1}{2}}(\partial\Omega)} 
\\ &\lesssim \nonumber \Vert\bm{\mathrm{T}}^*\Vert_{\mathcal{L}(\mathrm{H}^{-\frac{1}{2}}(\partial\Omega))}\frac{1}{|\partial\Omega|} \Vert 1 \Vert_{\mathrm{L}^2(\partial\Omega)} \Big|\int_{\partial\Omega}\partial_\mathrm{t}^2\partial_\nu u(\cdot,\mathrm{t}) \frac{1}{|\partial\Omega|}\int_{\partial\Omega}\int_{\partial\Omega}\frac{(\mathrm{x}-\mathrm{y})\cdot\nu_\mathrm{x}}{|\mathrm{x}-\mathrm{y}|} d\sigma_\mathrm{x}\sigma_\mathrm{y}\Big|
\\ &\lesssim\delta^{-1} \Big|\int_{\partial\Omega}\partial_\mathrm{t}^2\partial_\nu u(\cdot,\mathrm{t})\Big| \lesssim \delta^{-1}\Vert 1\Vert_{\mathrm{L}^2(\Omega)}\Vert \partial_\mathrm{t}^4 u(\cdot,\mathrm{t}) \Vert_{\mathrm{L}_0^2(\partial\Omega)}\sim \delta^2.
\end{align}
As above, we show that
\begin{align}\label{s4}
    \bullet\ \text{Estimation of}\ \mathrm{S}_4 &:= \frac{1}{|\partial\Omega|}\Big\Vert \int_{\partial\Omega}\partial_\mathrm{t}^2\partial_\nu u(\cdot,\mathrm{t}) \Big(\mathrm{A}(\mathrm{y})-\mathrm{A}_{\partial\Omega}\Big)\Big\Vert_{\mathrm{H}_0^{-\frac{1}{2}}(\partial\Omega)}\Vert\bm{\mathrm{T}}^*\Vert_{\mathcal{L}(\mathrm{H}_0^{-\frac{1}{2}}(\partial\Omega))} \sim \delta^4.
\end{align}

\begin{align}\label{s6}
     \bullet\ &\text{Estimation of}\ \mathrm{S}_6\nonumber \\ &:= \nonumber\Vert \bm{\mathrm{T}}^*\Vert_{\mathcal{L}(\mathrm{H}^{-\frac{1}{2}}_0(\partial\Omega))}\Big\Vert \int_{\partial\Omega}(\partial_\nu u)_{\mathrm{t}\mathrm{t}\mathrm{t}}(\mathrm{x}-\mathrm{y})\cdot\nu_\mathrm{x} d\sigma_\mathrm{y}
    -\frac{1}{|\partial\Omega|}\int_{\partial\Omega} \int_{\partial\Omega}(\partial_\nu u)_{\mathrm{t}\mathrm{t}\mathrm{t}} (\mathrm{x}-\mathrm{y})\cdot\nu_\mathrm{x} d\sigma_\mathrm{y}d\sigma_\mathrm{x}\Big\Vert_{\mathrm{H}^{-\frac{1}{2}}_0(\partial\Omega)}
    \\ \nonumber &\le \Big\Vert \int_{\partial\Omega}(\partial_\nu u)_{\mathrm{t}\mathrm{t}\mathrm{t}}(\mathrm{x}-\mathrm{y})\cdot\nu_\mathrm{x} d\sigma_\mathrm{y}\Big\Vert_{\mathrm{H}^{-\frac{1}{2}}_0(\partial\Omega)} + \Big\Vert\frac{1}{|\partial\Omega|}\int_{\partial\Omega} \int_{\partial\Omega}(\partial_\nu u)_{\mathrm{t}\mathrm{t}\mathrm{t}} (\mathrm{x}-\mathrm{y})\cdot\nu_\mathrm{x} d\sigma_\mathrm{y}d\sigma_\mathrm{x}\Big\Vert_{\mathrm{H}^{-\frac{1}{2}}_0(\partial\Omega)}
    \\ \nonumber &\lesssim \Vert 1 \Vert_{\mathrm{L}^2(\partial\Omega)}\Big| \int_{\partial\Omega}(\mathrm{x}-\mathrm{y})\cdot\nu_\mathrm{x}\partial_\mathrm{t}^3\partial_\nu u(\mathrm{y},\mathrm{t}_\mathrm{m}) d\sigma_\mathrm{y} \Big|
    \\ &\lesssim \delta\Big\Vert \partial_\mathrm{t}^3\partial_\nu u(\cdot,\mathrm{t}^*_\mathrm{m}) \Big\Vert_{\mathrm{H}_0^{-\frac{1}{2}}(\partial\Omega)} \Vert (\mathrm{x}-\cdot)\cdot\nu_\mathrm{x}\Vert_{\mathrm{L}^2(\partial\Omega)} \sim \delta^{\textcolor{black}{\frac{9}{2}}}.
\end{align}
\begin{align}\label{s7}
    \bullet\ \text{Estimation of}\ \mathrm{S}_7 &:= \nonumber\Vert\bm{\mathrm{T}}^*\Vert_{\mathcal{L}(\mathrm{H}^{-\frac{1}{2}}(\partial\Omega))} \frac{1}{|\partial\Omega|}\Big\Vert \int_{\partial\Omega} \int_{\partial\Omega} (\mathrm{x}-\mathrm{y})\cdot\nu_\mathrm{x} \partial_\mathrm{t}^3\partial_\nu u(\mathrm{y},\mathrm{t}_\mathrm{m}) d\sigma_\mathrm{y}d\sigma_\mathrm{x}\Big\Vert_{\mathrm{H}^{-\frac{1}{2}}(\partial\Omega)}
    \\ \nonumber &\lesssim \Vert\bm{\mathrm{T}}^*\Vert_{\mathcal{L}(\mathrm{H}^{-\frac{1}{2}}(\partial\Omega))}\frac{1}{|\partial\Omega|} \Big\Vert\int_{\partial\Omega}\partial_\mathrm{t}^3\partial_\nu u(\mathrm{y},\mathrm{t}_\mathrm{m}) \int_{\partial\Omega}(\mathrm{x}-\mathrm{y})\cdot\nu_\mathrm{x} d\sigma_\mathrm{x}d\sigma_\mathrm{y} \Big\Vert_{\mathrm{H}^{-\frac{1}{2}}(\partial\Omega)}
    \\ \nonumber &\lesssim \delta^{-2} \frac{1}{|\partial\Omega|} |\Omega| \Big\Vert\int_{\partial\Omega}\partial_\mathrm{t}^3\partial_\nu u(\mathrm{y},\mathrm{t}_\mathrm{m}) d\sigma_\mathrm{y} \Big\Vert_{\mathrm{H}^{-\frac{1}{2}}(\partial\Omega)}
    \\ \nonumber &\lesssim \delta^{-1}\Vert 1 \Vert_{\mathrm{L}^2(\partial\Omega)} \Big| \int_{\partial\Omega}\partial_\mathrm{t}^3\partial_\nu u(\mathrm{y},\mathrm{t}_\mathrm{m}) d\sigma_\mathrm{y} \Big|
    \\ &\lesssim \Vert 1 \Vert_{\mathrm{L}^2(\partial\Omega)} \Big\Vert \partial^3_\mathrm{t}\partial_\nu u(\cdot,\mathrm{t}_\mathrm{m})\Vert_{\mathrm{H}^{-\frac{1}{2}}(\partial\Omega)} \sim \delta^2.
\end{align}
Furthermore, we use the divergence theorem to the Newtonian potential $\bm{\mathcal{N}}_{\textbf{Lap},\Omega}[\partial_\mathrm{t}^2u]$ to derive
\begin{align}\label{s8}
 \bullet\ \text{Estimation of}\ \mathrm{S}_8 &:= \nonumber \frac{\gamma}{\alpha}\Vert\bm{\mathrm{T}}^*\Vert_{\mathcal{L}(\mathrm{H}^{-\frac{1}{2}}(\partial\Omega))} \frac{1}{|\partial\Omega|}\Big\Vert \int_{\partial\Omega}\partial_\nu\bm{\mathcal{N}}_{\textbf{Lap},\Omega}[\partial_\mathrm{t}^2u]\Big\Vert_{\mathrm{H}^{-\frac{1}{2}}(\partial\Omega)}
 \\ &\lesssim \nonumber \frac{1}{|\partial\Omega|} \Vert 1 \Vert_{\mathrm{L}^2(\partial\Omega)} \Big|\int_{\partial\Omega}\partial_\nu\bm{\mathcal{N}}_{\textbf{Lap},\Omega}[\partial_\mathrm{t}^2u]\Big|
 \\ &\lesssim \nonumber \delta^{-1} \Big|\int_{\Omega}\Delta\bm{\mathcal{N}}_{\textbf{Lap},\Omega}[\partial_\mathrm{t}^2u]\Big|
 \\ &= \delta^{-1} \Big|\int_{\Omega}\partial_\mathrm{t}^2u(\mathrm{y},\mathrm{t})d\mathrm{y}\Big| \lesssim \delta^{-1}\Vert 1\Vert_{\mathrm{L}^2(\Omega)} \Vert \partial_\mathrm{t}^2u(\cdot,\mathrm{t})\Vert_{\mathrm{L}^2(\Omega)} \sim \delta^2.
\end{align}
Then, we use the continuity of the Newtonian potential, triangle inequality, similar estimate as (\ref{s8}), (\ref{ss9}) to determine
\begin{align}\label{s9}
    \bullet\ \text{Estimation of}\ \mathrm{S}_9 &:= \nonumber \frac{\gamma}{\alpha} \Big\Vert \partial_\nu\bm{\mathcal{N}}_{\textbf{Lap},\Omega}[\partial_\mathrm{t}^2u]- \frac{1}{|\partial\Omega|}\int_{\partial\Omega}\partial_\nu\bm{\mathcal{N}}_{\textbf{Lap},\Omega}[\partial_\mathrm{t}^2u]\Big\Vert_{\mathrm{H}_0^{-\frac{1}{2}}(\partial\Omega)}\Vert\bm{\mathrm{T}}^*\Vert_{\mathcal{L}(\mathrm{H}_0^{-\frac{1}{2}}(\partial\Omega))} 
    \\ &\le \nonumber\frac{\gamma}{\alpha} \Big\Vert\partial_\nu\bm{\mathcal{N}}_{\textbf{Lap},\Omega}[\partial_\mathrm{t}^2u] \Big\Vert_{\mathrm{H}_0^{-\frac{1}{2}}(\partial\Omega)} + \frac{\Vert 1\Vert_{\mathrm{L}^2(\partial\Omega)}}{|\partial\Omega|}\Big|\int_{\partial\Omega}\partial_\nu\bm{\mathcal{N}}_{\textbf{Lap},\Omega}[\partial_\mathrm{t}^2u]\Big|
    \\ &\sim \delta^2.
\end{align}
In addition
\begin{align}
\nonumber
    \Big|\int_{\partial\Omega}\partial_\nu\int_\Omega|\mathrm{x}-\mathrm{y}|\partial_\mathrm{t}^4u(\mathrm{y},\mathrm{t}_3)d\mathrm{y}d\sigma_\mathrm{x}\Big| &= \Big|\int_{\Omega}\Delta\int_\Omega|\mathrm{x}-\mathrm{y}|\partial_\mathrm{t}^4u(\mathrm{y},\mathrm{t}_3)d\mathrm{y}d\mathrm{x}\Big|
    \\&= \nonumber\Big|\int_{\Omega}\int_\Omega\Delta|\mathrm{x}-\mathrm{y}|\partial_\mathrm{t}^4u(\mathrm{y},\mathrm{t}_3)d\mathrm{y}d\mathrm{x}\Big|,
\end{align}
therefore, we derive
\begin{align}\label{n}
    \Big|\int_{\partial\Omega}\partial_\nu\int_\Omega|\mathrm{x}-\mathrm{y}|\partial_\mathrm{t}^4u(\mathrm{y},\mathrm{t}_3)d\mathrm{y}d\sigma_\mathrm{x}\Big| \sim \delta^5.
\end{align}
Consequently, we obtain from (\ref{n}) that  
\begin{align}\label{s10}
    \bullet\ \text{Estimation of}\ \mathrm{S}_{10} &:=  \mathrm{c}^{-2}_0\frac{\gamma}{\alpha} \frac{1}{|\partial\Omega|}\Big\Vert \int_{\partial\Omega}\partial_\nu\int_\Omega|\mathrm{x}-\mathrm{y}|\partial_\mathrm{t}^4u(\mathrm{y},\mathrm{t}_3)d\mathrm{y}\Big\Vert_{\mathrm{H}^{-\frac{1}{2}}(\partial\Omega)}\Vert\bm{\mathrm{T}}^*\Vert_{\mathcal{L}(\mathrm{H}^{-\frac{1}{2}}(\partial\Omega))} \sim \delta^4.
\end{align}
In a similar manner, we show that
\begin{align}\label{s11}
   \bullet\ &\text{Estimation of}\ \mathrm{S}_{11} \\&:= \nonumber \mathrm{c}^{-2}_0\frac{\gamma}{\alpha}\Big\Vert \partial_\nu\int_\Omega|\mathrm{x}-\mathrm{y}|\partial_\mathrm{t}^4u\Big(\mathrm{y},\mathrm{t}^*_3\Big)d\mathrm{y} -\nonumber \frac{1}{|\partial\Omega|} \int_{\partial\Omega}\partial_\nu\int_\Omega|\mathrm{x}-\mathrm{y}|\partial_\mathrm{t}^4u(\mathrm{y},\mathrm{t}^*_3)d\mathrm{y}\Big\Vert_{\mathrm{H}_0^{-\frac{1}{2}}(\partial\Omega)}\Vert\bm{\mathrm{T}}^*\Vert_{\mathcal{L}(\mathrm{H}_0^{-\frac{1}{2}}(\partial\Omega))}
  \sim \delta^6.
\end{align}
Hence, with the help of the estimates (\ref{s1}), (\ref{s2}), (\ref{s3}), (\ref{s4}), (\ref{s5}), (\ref{s6}),(\ref{s7}), (\ref{s8}), (\ref{s9}), (\ref{s10}) and (\ref{s11}) we deduce that
\begin{align}
\Vert \partial_\mathrm{t}^\mathrm{k}\partial_\nu u(\cdot, \mathrm{t})\Vert_{\mathrm{H}^{-\frac{1}{2}}(\partial\Omega)} \sim \delta^2 \quad \text{for}\quad \mathrm{k} =0,1,...
\end{align}
Therefore it completes the proof of Proposition \ref{p3}.

%--------------------------------------------------------------------------
%--------------------------------------------------------------------------

\section{Appendix}\label{appen}          % Section

\subsection{Well-posedness and Regularity of the Problem (\ref{mainfor}): Proof of Lemma \ref{wellpose}} \label{well1}   % Subsection
\textcolor{black}{
To show the well-posedness of problem (\ref{mainfor}), we use the approach proposed by Bamberger and Ha Duong \cite{hduong} and Sayas in \cite{sayas} based on the Fourier-Laplace transform.
\bigbreak
\noindent
To begin with, we consider the following elliptic problem: 
\begin{align}\label{laplace}
    \mathrm{k}^{-1} \bm{\mathrm{s}}^2 \Tilde{u}^\mathrm{s}(\mathrm{x},\bm{\mathrm{s}}) + \text{div} \rho^{-1} \nabla \Tilde{u}^\mathrm{s}(\mathrm{x},\bm{\mathrm{s}}) = \mathrm{F}^\ell(\mathrm{x},\bm{\mathrm{s}}).
\end{align}
The above equation can be seen as the Laplace transform to the equation $\mathrm{k}^{-1}u_{\mathrm{t}\mathrm{t}}^\mathrm{s}- \text{div}\rho^{-1}\nabla u^\mathrm{s} = \mathrm{F}(\mathrm{x},\mathrm{t})$, with respect to the time variable, where $\bm{\mathrm{s}}= \sigma + i\omega \in \mathbb{C}$ is the transform parameter with $\sigma \in \mathbb{R}, \sigma>\sigma_0>0,$ for some constant $\sigma_0$, and $\omega\in \mathbb{R}.$ We followed the convention that $\mathrm{F}^\ell(\mathrm{x},\cdot)$ is the Laplace transformation of $\mathrm{F}(\mathrm{x},\cdot)$. 
\bigbreak
\noindent
Next, we develop a variational method for the aforementioned problem (\ref{laplace}) and utilize the Lax-Milgram Lemma.
\noindent
By multiplying equation (\ref{laplace}) by the complex conjugate of $\mathrm{v}\in \mathrm{H}^1(\mathbb{R}^3)$, and integrating over $\mathbb R^3$, we obtain a sesquilinear mapping $\mathrm{a}(\Tilde{u}^\mathrm{s},\mathrm{v}): \mathrm{H}^1(\mathbb R^3) \times \mathrm{H}^1(\mathbb R^3) \to \mathbb{C}$ and an antilinear mapping $\mathrm{b}(\mathrm{v}): \mathrm{H}^1(\mathbb R^3) \to \mathbb{C}$, such that
\begin{align}\label{bi-li}
     \mathrm{a}(\Tilde{u}^\mathrm{s},\mathrm{v}) = \mathrm{b}(\mathrm{v}) \quad \text{for all}\; \mathrm{v}\in \mathrm{H}^1(\mathbb R^3),
\end{align}
where
\begin{align}
\nonumber
    \mathrm{a}(\Tilde{u}^\mathrm{s},\mathrm{v}) = \int_{\mathbb{R}^3}\mathrm{k}^{-1} 
    \bm{\mathrm{s}}^2 \Tilde{u}^\mathrm{s}\;\overline{\mathrm{v}} d\mathrm{x} + \int_{\mathbb{R}^3} \rho^{-1}\nabla\Tilde{u}^\mathrm{s}\cdot\overline{\nabla\mathrm{v}} d\mathrm{x},
\end{align}
and
\begin{align}
\nonumber
    \mathrm{b}(\mathrm{v}) = \big\langle \mathrm{F}^\ell(\cdot,\bm{\mathrm{s}}),\mathrm{v}\big\rangle,
\end{align}
where $\big\langle \cdot,\cdot\big\rangle$ denotes the duality pairing between $\mathrm{H}^{1}(\mathbb R^3)$ and $\mathrm{H}^{-1}(\mathbb R^3).$
Now, to verify the coercivity of the above bi-linear form, we choose $\mathrm{v} = \bm{\mathrm{s}}\Tilde{u}^\mathrm{s}$ and use integration by parts to obtain
\begin{align}\label{bi}
    \mathrm{a}(\Tilde{u}^\mathrm{s},\bm{\mathrm{s}}\Tilde{u}^\mathrm{s}) = \int_{\mathbb{R}^3}\mathrm{k}^{-1} \overline{\bm{\mathrm{s}}}|\bm{\mathrm{s}}|^2|\Tilde{u}^\mathrm{s}|^2 d\mathrm{x} + \int_{\mathbb{R}^3} \rho^{-1}\overline{\bm{\mathrm{s}}}|\nabla\Tilde{u}^\mathrm{s}|^2 d\mathrm{x}.
\end{align}
Consequetly, $\Tilde{u}^\mathrm{s}(\cdot,\mathrm{s})$ satisfies
\begin{align}
 \overline{\bm{\mathrm{s}}}|\bm{\mathrm{s}}|^2 \|\mathrm{k}^{-1}\Tilde{u}^\mathrm{s}(\cdot,\bm{\mathrm{s}})\|^2_{\mathrm{L}^2(\mathbb R^3)} + \overline{\bm{\mathrm{s}}} \left\|\rho^{-1} \nabla\Tilde{u}^\mathrm{s}(\cdot,\bm{\mathrm{s}})\right\|^2_{\mathrm{L}^2(\mathbb R^3)} = \overline{\bm{\mathrm{s}}}  \big\langle\mathrm{F}^\ell(\cdot,\bm{\mathrm{s}}), \Tilde{u}^\mathrm{s}(\cdot,\bm{\mathrm{s}}) \big\rangle.
\end{align}
Then, after taking the real part of the above expression, we have $\Re\big(\mathrm{a}(\Tilde{u}^\mathrm{s},\bm{\mathrm{s}}\Tilde{u}^\mathrm{s})\big)\ge 0$ and consequently, we deduce
\begin{align}\label{ses}
    | \mathrm{a}(\Tilde{u}^\mathrm{s},\bm{\mathrm{s}}\Tilde{u}^\mathrm{s})| \ge \mathrm{C}(\sigma_0) \Vert \Tilde{u}^\mathrm{s}(\cdot,\bm{\mathrm{s}})\Vert^2_{\mathrm{H}^{1}(\mathbb R^3)},
\end{align}
where $\mathrm{C}(\sigma_0)$ is a positive constant.
\newline
Next, we use duality between the function spaces $\mathrm{H}^1(\mathbb R^3)$ and $\mathrm{H}^{-1}(\mathbb R^3)$ to obtain
\begin{align}\label{anti}
    |\mathrm{b}(\bm{\mathrm{s}}\Tilde{u}^\mathrm{s})| \le |\bm{\mathrm{s}}| \Vert \mathrm{F}^\ell(\cdot,\bm{\mathrm{s}}) \Vert_{\mathrm{H}^{-1}(\mathbb R^3)}\Vert \Tilde{u}^\mathrm{s}(\cdot,\bm{\mathrm{s}})\Vert_{\mathrm{H}^{1}(\mathbb R^3)}.
\end{align}
We deduce by combining (\ref{ses}) and (\ref{anti}) that (\ref{bi-li}) has one and only one solution and it satisfies
\begin{align}\label{4.6}
    \Vert \Tilde{u}^\mathrm{s}(\cdot,\bm{\mathrm{s}})\Vert_{\mathrm{H}^{1}(\mathbb R^3)} \le \frac{|\bm{\mathrm{s}}|}{\mathrm{C}(\sigma_0)}\Vert \mathrm{F}^\ell(\cdot,\bm{\mathrm{s}}) \Vert_{\mathrm{H}^{-1}(\mathbb R^3)}. 
\end{align}
Let us now define the inverse Laplace transform of $\Tilde{u}^\mathrm{s}(\mathrm{x},\cdot)$ for $\Re(\bm{\mathrm{s}}) =\sigma>0$ as
\begin{align}\label{invlap}
    u^\mathrm{s}(\mathrm{x},\mathrm{t}):= \frac{1}{2\pi i}\int_{\sigma-i\infty}^{\sigma+i\infty}e^{\bm{\mathrm{s}}\mathrm{t}}\Tilde{u}^\mathrm{s}(\mathrm{x},\bm{\mathrm{s}})d\bm{\mathrm{s}} = \frac{1}{2\pi}\int_{-\infty}^{\infty}e^{(\sigma+i\omega)\mathrm{t}}\Tilde{u}^\mathrm{s}(\mathrm{x},\sigma+i\omega)d\omega.
\end{align}
Due to the estimate with respect to $`\bm{\mathrm{s}}`$ in (\ref{4.6}), $u^\mathrm{s}(\mathrm{x},\mathrm{t})$ is well-defined. In addition, one can show that $u^\mathrm{s}(\mathrm{x},\mathrm{t})$ does not depend on $\sigma$ by utilizing a classical method of contour integration, see \cite[pp. 39]{sayas}.
If we consider the Fourier transform w.r.t time variable $\Fourier_\mathrm{t}$, then we have $\Fourier_{\mathrm{t}\to \omega}\big(e^{-\sigma \mathrm{t}} \partial^\mathrm{k}_\mathrm{t}u^\mathrm{s}(\mathrm{x},\mathrm{t})\big)=\bm{\mathrm{s}}^\mathrm{k}\Tilde{u}^\mathrm{s}(\mathrm{x},\bm{\mathrm{s}})$ with $\bm{\mathrm{s}}=\sigma+i\omega.$ Thus, we get for $\mathrm{r}\in \mathbb N$ the following
\begin{align}
\nonumber
  \|u^\mathrm{s}\|^2_{\mathrm{H}^{\mathrm{r}}_{0,\sigma}\big((0,\mathrm{T}); \mathrm{H}^1(\mathbb R^3)\big)} & = \int_0^{\mathrm{T}} e^{-2\sigma \mathrm{t}} \sum_{\mathrm{k}=0}^{\mathrm{r}} \mathrm{T}^{2\mathrm{k}} \|\partial^\mathrm{k}_\mathrm{t} u^\mathrm{s}(\cdot,\mathrm{t})\|^2_{\mathrm{H}^1(\mathbb R^3)} \,d\mathrm{t} \\
  & \nonumber \lesssim \int_{\mathbb{R}_+} \int_{\mathbb{R}^3} e^{-2\sigma \mathrm{t}} \sum_{\mathrm{k}=0}^{\mathrm{r}} \Big[|\partial^\mathrm{k}_\mathrm{t} u^\mathrm{s}(\mathrm{x},\mathrm{t})|^2 + |\partial^\mathrm{k}_\mathrm{t} \nabla u^\mathrm{s}(\mathrm{x},\mathrm{t})|^2\Big] \, d\mathrm{x}\, d\mathrm{t} \\
  & \nonumber \lesssim \int_{\mathbb{R}^3} \int_{\mathbb{R}} \sum_{\mathrm{k}=0}^{\mathrm{r}} \Big[\big|\Fourier ( e^{-\sigma \mathrm{t}} \partial^\mathrm{k}_\mathrm{t} u^\mathrm{s}(\mathrm{x},\mathrm{t})\big|^2+ \big|\Fourier ( e^{-\sigma \mathrm{t}} \partial^\mathrm{k}_\mathrm{t} \nabla u^\mathrm{s}(\mathrm{x},\mathrm{t})\big|^2\Big] \, d\mathrm{t}\, d\mathrm{x} \\
  & \nonumber\lesssim \sum_{\mathrm{k}=0}^{\mathrm{r}} \int_{\sigma+i\mathbb{R}} |\bm{\mathrm{s}}|^{2\mathrm{k}} \|\Tilde{u}^\mathrm{s}(\cdot,\bm{\mathrm{s}})\|^2_{\mathrm{H}^1(\mathbb R^3)} \, d\bm{\mathrm{s}} \\
  & \nonumber \lesssim \sum_{\mathrm{k}=0}^{\mathrm{r}+1} \int_{\sigma+i\mathbb{R}} |\bm{\mathrm{s}}|^{2\mathrm{k}} \|\mathrm{F}^\ell(\cdot,\bm{\mathrm{s}})\|^2_{\mathrm{H}^{-1}(\mathbb R^3)} \, d\bm{\mathrm{s}} \simeq \|\mathrm{F}\|^2_{\mathrm{H}^{\mathrm{r}+1}_{0,\sigma}\big((0,\mathrm{T}); \mathrm{H}^{-1}(\mathbb R^3)\big)}.
\end{align}
We now show that the function $u^\mathrm{s}$, defined in equation (\ref{invlap}), is a weak solution to the problem described in equation (\ref{wellposed}) $\big(\text{or}\ (\ref{mainfor})\big)$. To do so, we consider the following weak formulation of the problem $\mathrm{k}^{-1}u_{\mathrm{t}\mathrm{t}}^\mathrm{s}- \text{div}\rho^{-1}\nabla u^\mathrm{s} = \mathrm{F}(\mathrm{x},\mathrm{t})$:
\begin{align}
    \big\langle \mathrm{k}^{-1}\frac{d^2}{d\mathrm{t}^2}u^\mathrm{s}(\cdot,\mathrm{t}),\mathrm{v}\big\rangle +   \big\langle \rho^{-1}\nabla u^\mathrm{s}(\cdot,\mathrm{t}),\nabla\mathrm{v}\big\rangle =  \big\langle \mathrm{F},\mathrm{v} \big\rangle\quad \text{for a.e.}\; \mathrm{t}\in (0,\mathrm{T})\; \text{and}\ \forall \mathrm{v}\in \mathrm{H}^1(\mathbb R^3).
\end{align}
We see that
\begin{align}
    &\nonumber\big\langle \mathrm{k}^{-1}\frac{d^2}{d\mathrm{t}^2}u^\mathrm{s}(\cdot,\mathrm{t}),\mathrm{v}\big\rangle +   \big\langle \rho^{-1}\nabla u^\mathrm{s}(\cdot,\mathrm{t}),\nabla\mathrm{v}\big\rangle 
    \\ \nonumber&=\nonumber \int_{\mathbb R^3}\mathrm{k}^{-1}\int_{\sigma-i\infty}^{\sigma+i\infty}e^{\mathrm{s}\mathrm{t}}\bm{\mathrm{s}}^2\Tilde{u}^\mathrm{s}(\mathrm{x},\bm{\mathrm{s}})\overline{\mathrm{v}(\mathrm{x})}d\bm{\mathrm{s}}d\mathrm{x}  +  \int_{\mathbb R^3}\rho^{-1}\int_{\sigma-i\infty}^{\sigma+i\infty}e^{\mathrm{s}\mathrm{t}} \nabla\Tilde{u}^\mathrm{s}(\mathrm{x},\bm{\mathrm{s}})\cdot\overline{\nabla\mathrm{v}(\mathrm{x})}d\bm{\mathrm{s}}d\mathrm{x} 
    \\ \nonumber &= \int_{\sigma-i\infty}^{\sigma+i\infty}\int_{\mathbb R^3}e^{\mathrm{s}\mathrm{t}}\Big(\mathrm{k}^{-1}\bm{\mathrm{s}}^2\Tilde{u}^\mathrm{s}(\mathrm{x},\bm{\mathrm{s}})\overline{\mathrm{v}(\mathrm{x})}+\rho^{-1}\nabla\Tilde{u}^\mathrm{s}(\mathrm{x},\bm{\mathrm{s}})\cdot\overline{\nabla\mathrm{v}(\mathrm{x})}\Big)d\mathrm{x}d\bm{\mathrm{s}}
    \\ \nonumber &= \int_{\sigma-i\infty}^{\sigma+i\infty} e^{\mathrm{s}\mathrm{t}}\big\langle \mathrm{F}^\ell(\cdot,\bm{\mathrm{s}}),\mathrm{v}\big\rangle d\bm{\mathrm{s}}\quad \text{by}\; (\ref{bi-li})  \\ \nonumber&= \big\langle \mathrm{F}(\cdot,\mathrm{t}),\mathrm{v}\big\rangle, \quad \text{where}\ \big\langle \cdot,\cdot\big\rangle\; \text{denotes the duality pairing between}\;\mathrm{H}^{1}(\mathbb R^3)\ \text{and}\ \mathrm{H}^{-1}(\mathbb R^3).
\end{align}
The proof is complete.}

\subsection{Proof of Proposition \ref{propo}}      % Subsection

We provides the estimates of $\mathrm{a}_\mathrm{i}$'s for $\mathrm{i}=1,2,..,5$ stated in Proposition \ref{propo}. We start with the term $\mathrm{a}_1:= \frac{1}{\alpha}\partial_\mathrm{t}^2\partial_\nu u^\mathrm{i}$:
\begin{align}
  \Vert\mathrm{a}_1\Vert_{\mathrm{H}^{-\frac{1}{2}}_0(\partial\Omega)} = \frac{1}{\alpha}\Vert \partial_\mathrm{t}^2\partial_\nu u^\mathrm{i}\Vert_{\mathrm{H}^{-\frac{1}{2}}_0(\partial\Omega)} \lesssim \frac{1}{\alpha}\Vert 1 \Vert_{\mathrm{L}^2(\partial\Omega)} \sim \delta^3.
\end{align}
Next, we consider the term $\displaystyle \mathrm{a}_2 := \int_{\partial\Omega}\partial_\mathrm{t}^4\partial_\nu u(\mathrm{y},\mathrm{t}) \frac{(\mathrm{x}-\mathrm{y})\cdot\nu_\mathrm{x}}{|\mathrm{x}-\mathrm{y}|}d\sigma_\mathrm{y}$. As $ \Big\Vert \frac{(\mathrm{x}-\cdot)\cdot \nu_\mathrm{x}}{|\mathrm{x}-\cdot|} \Big\Vert_{\mathrm{H}^{\frac{1}{2}}_0(\partial\Omega)} \sim \delta^\frac{1}{2}$, see (\ref{es6}), we obtain the following estimate
\begin{align}\label{est7}
\nonumber
    \Big|\int_{\partial\Omega}\partial_\mathrm{t}^4\partial_\nu u(\mathrm{y},\mathrm{t}) \frac{(\mathrm{x}-\mathrm{y})\cdot\nu_\mathrm{x}}{|\mathrm{x}-\mathrm{y}|}d\sigma_\mathrm{y}\Big| &\lesssim \Vert \partial_\mathrm{t}^4\partial_\nu u(\cdot,\mathrm{t}) \Vert_{\mathrm{H}^{-\frac{1}{2}}_0(\partial\Omega)} \Big\Vert \frac{(\mathrm{x}-\cdot)\cdot \nu_\mathrm{x}}{|\mathrm{x}-\cdot|} \Big\Vert_{\mathrm{H}^{\frac{1}{2}}_0(\partial\Omega)}
    \\ &= \delta^{\frac{1}{2}}\Vert \partial_\mathrm{t}^4\partial_\nu u(\cdot,\mathrm{t}) \Vert_{\mathrm{H}^{-\frac{1}{2}}_0(\partial\Omega)}.
\end{align}
Then, with $\Vert \partial_\mathrm{t}^4\partial_\nu u(\cdot,\mathrm{t}) \Vert_{\mathrm{H}^{-\frac{1}{2}}(\partial\Omega)} \sim \delta^2$, we obtain 
\begin{align}\label{est9}
    \Vert\mathrm{a}_2\Vert_{\mathrm{H}^{-\frac{1}{2}}_0(\partial\Omega)} :=\Big\Vert \int_{\partial\Omega} \frac{(\mathrm{x}-\mathrm{y})\cdot\nu_\mathrm{x}}{|\mathrm{x}-\mathrm{y}|}\partial_\mathrm{t}^4\partial_\nu u(\mathrm{y},\mathrm{t}) d\sigma_\mathrm{y}\Big\Vert_{\mathrm{H}^{-\frac{1}{2}}_0(\partial\Omega)} &\nonumber\lesssim \delta^{\frac{1}{2}}\Vert \partial_\mathrm{t}^4\partial_\nu u(\cdot,\mathrm{t}) \Vert_{\mathrm{H}^{-\frac{1}{2}}_0(\partial\Omega)} \Vert 1 \Vert_{\mathrm{L}^2(\partial \Omega)}
    \\ &= \delta^\frac{3}{2}\Vert \partial_\mathrm{t}^4\partial_\nu u(\cdot,\mathrm{t}) \Vert_{\mathrm{H}^{-\frac{1}{2}}_0(\partial\Omega)} \sim \delta^\frac{7}{2}.
\end{align}
Similarly, for the term $\mathrm{a}_3:= \displaystyle\frac{5}{6}\mathrm{c}^{-3}_0 \int_{\partial\Omega}(\mathrm{x}-\mathrm{y})\cdot\nu_\mathrm{x}\partial_\mathrm{t}^5\partial_\nu u(\mathrm{y},\mathrm{t}_\mathrm{m}) d\sigma_\mathrm{y}$, we have with $\Vert \partial_\mathrm{t}^4\partial_\nu u(\cdot,\mathrm{t}) \Vert_{\mathrm{H}^{-\frac{1}{2}}(\partial\Omega)} \sim \delta^2$,
\begin{align}\label{est2}
    \Vert\mathrm{a}_3\Vert_{\mathrm{H}^{-\frac{1}{2}}_0(\partial\Omega)} :=\Big\Vert \frac{5}{6}\mathrm{c}^{-3}_0 \int_{\partial\Omega}(\mathrm{x}-\mathrm{y})\cdot\nu_\mathrm{x}\partial_\mathrm{t}^5\partial_\nu u(\mathrm{y},\mathrm{t}_\mathrm{m}) d\sigma_\mathrm{y}\Big\Vert_{\mathrm{H}^{-\frac{1}{2}}_0(\partial\Omega)} \lesssim \delta^3 \Vert \partial_\mathrm{t}^5\partial_\nu u(\cdot,\mathrm{t}_\mathrm{m}) \Vert_{\mathrm{H}^{-\frac{1}{2}}_0(\partial\Omega)} \sim \delta^5.
\end{align}
We consider the term $\displaystyle \mathrm{a}_4 := \frac{\gamma}{\alpha} \partial_\nu\int_\Omega\frac{1}{4\pi|\mathrm{x}-\mathrm{y}|}\partial_\mathrm{t}^4u(\mathrm{y},\mathrm{t})d\mathrm{y}$. As, $\Delta\bm{\mathcal{N}}_{\textbf{Lap},\Omega}\Big[\partial_\mathrm{t}^4u\Big] = -\partial_\mathrm{t}^4u$, then \\ $\Big \Vert \partial_\nu \bm{\mathcal{N}}_{\textbf{Lap},\Omega}\Big[\partial_\mathrm{t}^4u\Big]\Big\Vert_{\mathrm{L}^2(\partial\Omega)} = \delta^2 \Big\Vert \partial_\nu \bm{\mathcal{N}}_{\textbf{Lap},\mathrm{B}}\Big[\partial_\mathrm{t}^4\hat{u}\Big]\Big\Vert_{\mathrm{L}^2(\partial\mathrm{B})}.$

\noindent
Thereafter, we use continuity of that $\bm{\mathcal{N}}_{\textbf{Lap},\mathrm{B}}: \mathrm{L}^2(\mathrm{B})\to \mathrm{H}^2(\mathrm{B})$ and continuous embedding $\mathrm{L}^2(\partial\mathrm{B})\hookrightarrow \mathrm{H}_0^{-\frac{1}{2}}(\partial\mathrm{B})$
to deduce the following
\begin{align}\label{ss9}
    \Big\Vert \partial_\nu \bm{\mathcal{N}}_{\textbf{Lap},\Omega}\Big[\partial_\mathrm{t}^4u\Big]\Big\Vert_{\mathrm{H}_0^{-\frac{1}{2}}(\partial\Omega)} &\le\nonumber \delta \Big\Vert \widehat{\partial_\nu \bm{\mathcal{N}}_{\textbf{Lap},\Omega}\Big[\partial_\mathrm{t}^4u\Big]}\Big\Vert_{\mathrm{H}_0^{-\frac{1}{2}}(\partial\mathrm{\mathrm{B}})}
    \\ &\le \nonumber \delta \Big\Vert \widehat{\partial_\nu \bm{\mathcal{N}}_{\textbf{Lap},\Omega}\Big[\partial_\mathrm{t}^4u\Big]}\Big\Vert_{\mathrm{L}^2(\partial\mathrm{\mathrm{B}})}
    \\ &\le \nonumber \delta^2 \Big\Vert \partial_\nu \bm{\mathcal{N}}_{\textbf{Lap},\mathrm{B}}\Big[\partial_\mathrm{t}^4\hat{u}\Big]\Big\Vert_{\mathrm{L}^2(\partial\mathrm{\mathrm{B}})}
    \\ &\lesssim \delta^2 \Vert \partial_\mathrm{t}^4\hat{u}\Vert_{\mathrm{L}^2(\mathrm{\mathrm{B}})} \sim \delta^2.
\end{align}
Therefore, we obtain
\begin{align}\label{est45}
 \Vert\mathrm{a}_4\Vert_{\mathrm{H}^{-\frac{1}{2}}_0(\partial\Omega)} :=  \Big\Vert\frac{\gamma}{\alpha} \partial_\nu\int_\Omega\frac{1}{4\pi|\mathrm{x}-\mathrm{y}|}\partial_\mathrm{t}^4u(\mathrm{y},\mathrm{t})d\mathrm{y}\Big\Vert_{\mathrm{H}_0^{-\frac{1}{2}}(\partial\Omega)} \sim \delta^4.
\end{align}
Finally, we consider $\displaystyle \mathrm{a}_5:= \mathrm{c}^{-2}_0\frac{\gamma}{\alpha} \partial_\nu\int_\Omega|\mathrm{x}-\mathrm{y}|\partial_\mathrm{t}^6u(\mathrm{y},\mathrm{t}_3)d\mathrm{y}.$ This term can be estimated as
\begin{align}\label{est451}
    \Vert\mathrm{a}_5\Vert_{\mathrm{H}^{-\frac{1}{2}}_0(\partial\Omega)} &:= \nonumber \mathrm{c}^{-2}_0\frac{\gamma}{\alpha} \Vert 1\Vert_{\mathrm{L}^2(\Omega)}\Big|\int_\Omega\frac{(\mathrm{x}-\mathrm{y})\cdot\nu_\mathrm{x}}{|\mathrm{x}-\mathrm{y}|}\partial_\mathrm{t}^6u(\mathrm{y},\mathrm{t}_3)d\mathrm{y}\Big|
    \\ &\lesssim \mathrm{c}^{-2}_0\frac{\gamma}{\alpha} \delta \Big\Vert \frac{(\mathrm{x}-\cdot)\cdot\nu_\mathrm{x}}{|\mathrm{x}-\cdot|}\Big\Vert_{\mathrm{L}^2(\Omega)} \Vert\partial_\mathrm{t}^6u(\cdot,\mathrm{t}_3) \Vert_{\mathrm{L}^2(\Omega)} \sim \delta^6.
\end{align}

\subsection{Proof of Lemma \ref{lemma}}          % Subsection

We start with the expression derived in (\ref{formula})
\begin{align}
    u^\mathrm{s}(\mathrm{x},\mathrm{t})
    = \frac{\alpha \rho_\mathrm{m}\mathrm{p}^{-\frac{1}{2}}}{4\pi} \textcolor{black}{\frac{\rho_\mathrm{c}}{\mathrm{k}_\mathrm{c}}\mathrm{c}_0^2}\frac{1}{|\partial\Omega|}\int_{\partial\Omega}\frac{1}{|\mathrm{x}-\mathrm{y}|}d\sigma_\mathrm{y}\int_0^{\mathrm{t}-\mathrm{c}_0^{-1}|\mathrm{x}-\mathrm{z}|} \sin\Big(\mathrm{p}^{-\frac{1}{2}}(\mathrm{t}-\mathrm{c}_0^{-1}|\mathrm{x}-\mathrm{z}|-\tau)\Big)\int_{\partial\Omega} \partial_\nu u^\mathrm{i} d\tau  + \mathcal{O}(\delta^{2-\mathrm{q}}).
\end{align}
where we recall the definitions of $\mathrm{p} := \textcolor{black}{ \frac{\alpha\rho_\mathrm{m}}{2}\frac{\rho_\mathrm{c}}{\mathrm{k}_\mathrm{c}}\mathrm{A}_{\partial\Omega}}$, $\alpha:=\frac{1}{\rho_\mathrm{c}}-\frac{1}{\rho_\mathrm{m}}$, and \textcolor{black}{$\displaystyle
    \mathrm{A}_{\partial\Omega} := \frac{1}{|\partial \Omega|}\int_{\partial\Omega}\int_{\partial\Omega}\frac{(\mathrm{x}-\mathrm{y})\cdot\nu}{|\mathrm{x}-\mathrm{y}|}d\sigma_\mathrm{x}d\sigma_\mathrm{y}= \delta^2 \mathrm{A}_{\partial\mathrm{B}}
$}.
\newline

\noindent
\textcolor{black}{Then,  due to the scaling property (\ref{cond-bubble}), we see that}
\begin{align}
    \alpha = \rho_\mathrm{c}^{-1} + \mathcal{O}(1)\quad \text{and}\quad \beta = \mathrm{k}_\mathrm{c}^{-1} + \mathcal{O}\big(1\big).
\end{align}
Consequently, we perform the following calculations.
\bigbreak
\noindent
\textcolor{black}{Let us consider the term $\frac{\alpha \rho_\mathrm{m}\mathrm{p}^{-\frac{1}{2}}}{4\pi} \frac{\rho_\mathrm{c}}{\mathrm{k}_\mathrm{c}}\mathrm{c}_0^2$, which can be rewritten as follows
\begin{align}
    \mathrm{D}:= &\nonumber\frac{\alpha\rho_\mathrm{m}}{4\pi}\frac{\rho_\mathrm{c}}{\mathrm{k}_\mathrm{c}}\mathrm{c}_0^2 \frac{\mathrm{p}^\frac{1}{2}}{\mathrm{p}}
    \\ \nonumber&= \frac{\cancel{\alpha\rho_\mathrm{m}}}{4\pi}\frac{\cancel{\rho_\mathrm{c}}}{\cancel{\mathrm{k}_\mathrm{c}}}\mathrm{c}_0^2\frac{2\cancel{\mathrm{k}_\mathrm{c}}}{\mathrm{A}_{\partial\Omega}\cancel{\alpha\rho_\mathrm{m}\rho_\mathrm{c}}}\mathrm{p}^\frac{1}{2}
    \\ &= \frac{\mathrm{c}_0^2}{2\pi\mathrm{A}_{\partial\Omega}}\mathrm{p}^\frac{1}{2}.
\end{align}
Then, as we have $\mathrm{p}^\frac{1}{2}= \sqrt{\frac{\rho_\mathrm{m}}{2}\frac{\mathrm{A}_{\partial\mathrm{B}}}{\overline{\mathrm{k}_\mathrm{c}}}}$, we deduce that
\begin{align}\label{cons}
    \mathrm{D} = \frac{\mathrm{c}_0^2}{4\pi} \frac{1}{\delta^2} \frac{\rho_\mathrm{m}}{\overline{\mathrm{k}_\mathrm{c}}}\sqrt{\frac{2\overline{\mathrm{k}_\mathrm{c}}}{\mathrm{A}_{\partial\mathrm{B}}\rho_\mathrm{m}}}.
\end{align}
}
\noindent
Hence, using (\ref{cons}), we obtain
\begin{align}
   u^\mathrm{s}(\mathrm{x},\mathrm{t})\nonumber
  = \nonumber \textcolor{black}{\frac{\mathrm{c}_0^2}{4\pi} \frac{1}{\delta^2} \frac{\rho_\mathrm{m}}{\overline{\mathrm{k}_\mathrm{c}}}\sqrt{\frac{2\overline{\mathrm{k}_\mathrm{c}}}{\mathrm{A}_{\partial\mathrm{B}}\rho_\mathrm{m}}}}\frac{1}{|\partial\Omega|}\int_{\partial\Omega}\frac{1}{|\mathrm{x}-\mathrm{y}|}d\sigma_\mathrm{y}\int_0^{\mathrm{t}-\mathrm{c}_0^{-1}|\mathrm{x}-\mathrm{z}|} &\nonumber\sin\big(\mathrm{p}^{-\frac{1}{2}}(\mathrm{t}-\mathrm{c}_0^{-1}|\mathrm{x}-\mathrm{z}|-\tau)\big)\Big(\int_{\partial\Omega} \partial_\nu u^\mathrm{i}(\mathrm{y},\mathrm{t})d\sigma_\mathrm{y}\Big) d\tau 
    \\ &+\nonumber \mathcal{O}(\delta^{2-\mathrm{q}}).
\end{align}
Let us denote $\omega_\mathrm{M} := \sqrt{
\frac{2\overline{\mathrm{k}_\mathrm{c}}}{\mathrm{A}_{\partial\mathrm{B}}\rho_\mathrm{m}}}.$ Thereafter, we arrive at
\begin{align}\label{final}
    u^\mathrm{s}(\mathrm{x},\mathrm{t})
    =  \textcolor{black}{\frac{\mathrm{c}_0^2}{4\pi} \frac{1}{\delta^2} \frac{\rho_\mathrm{m}}{\overline{\mathrm{k}_\mathrm{c}}}\omega_\mathrm{M}}\frac{1}{|\partial\Omega|}\int_{\partial\Omega}\frac{1}{|\mathrm{x}-\mathrm{y}|}d\sigma_\mathrm{y}\int_0^{\mathrm{t}-\mathrm{c}_0^{-1}|\mathrm{x}-\mathrm{z}|} &\nonumber\sin\big(\mathrm{p}^{-\frac{1}{2}}(\mathrm{t}-\mathrm{c}_0^{-1}|\mathrm{x}-\mathrm{z}|-\tau)\big)\Big(\int_{\partial\Omega} \partial_\nu u^\mathrm{i}(\mathrm{y},\mathrm{t})d\sigma_\mathrm{y}\Big) d\tau 
    \\ &+ \mathcal{O}(\delta^{2-\mathrm{q}}).
\end{align}
Using Taylor's series expansion and integration by parts, we derive the following estimate
\begin{align}\label{toy}
\int_{\partial\Omega} \partial_\nu u^\mathrm{i}(\mathrm{y},\mathrm{\tau})d\sigma_\mathrm{y} = \int_{\Omega} \Delta u^\mathrm{i}(\mathrm{y},\mathrm{\tau})d\mathrm{y}  
= \frac{\rho_\mathrm{m}}{\mathrm{k}_\mathrm{m}}\int_{\Omega} u_{\mathrm{t}\mathrm{t}}^\mathrm{i}(\mathrm{y},\mathrm{\tau})d\mathrm{y}
 = \frac{\rho_\mathrm{m}}{\mathrm{k}_\mathrm{m}}|\Omega| u_{\mathrm{t}\mathrm{t}}^\mathrm{i}(\mathrm{z},\mathrm{\tau}) + \mathcal{O}(\delta^4).
\end{align}
After inserting the estimate (\ref{toy}) in (\ref{final}) and using the fact that \textcolor{black}{$\mathrm{c}_0^{-2}=\frac{\rho_\mathrm{m}}{\mathrm{k}_\mathrm{m}} $}, we obtain
\begin{align}\label{final}
    u^\mathrm{s}(\mathrm{x},\mathrm{t})
    = \nonumber\frac{\omega_\mathrm{M}\rho_\mathrm{m}|\mathrm{B}| }{4\pi\overline{\mathrm{k}_\mathrm{c}}} \delta  \frac{1}{|\partial\Omega|}\int_{\partial\Omega}\frac{1}{|\mathrm{x}-\mathrm{y}|}d\sigma_\mathrm{y}\int_0^{\mathrm{t}-\mathrm{c}_0^{-1}|\mathrm{x}-\mathrm{z}|} \sin\big(\omega_\mathrm{M}(\mathrm{t}-\mathrm{c}_0^{-1}|\mathrm{x}-\mathrm{z}|-\tau)\big)u_{\mathrm{t}\mathrm{t}}^\mathrm{i}(\mathrm{z},\mathrm{\tau}) d\tau + \mathcal{O}(\delta^{2-\mathrm{q}}).
\end{align}
This completes the Proof.
\bigbreak
%-------------------------------------------------------------------------
%-------------------------------------------------------------------------

\noindent
\textcolor{black}{\textbf{Acknowledgement:} We would like to thank the referees for their careful reading and valuable suggestions, which made our manuscript much improved.}

\end{document}